\documentclass[11pt]{article}
\usepackage[a4paper]{geometry} 
\usepackage{tabularx}
\usepackage{subcaption}
\usepackage{showlabels}
\usepackage[utf8]{inputenc}
\usepackage{amssymb}
\usepackage{makeidx}
\usepackage[english]{babel}
\usepackage{graphicx}
\usepackage[dvipsnames]{xcolor}
\usepackage{amsfonts,amsmath}
\usepackage{siunitx}
\usepackage{oldgerm}
\usepackage{relsize}
\usepackage{mathrsfs}
\usepackage[active]{srcltx}
\usepackage{verbatim}
\usepackage{enumitem} 
\usepackage[toc,page]{appendix}
\usepackage{aliascnt,bbm}
\usepackage{array}
\usepackage[colorlinks = true, citecolor = blue]{hyperref}
\usepackage[textwidth=3cm,textsize=footnotesize]{todonotes}
\usepackage{xargs}
\usepackage{cellspace}
\usepackage{comment}
\usepackage[Symbolsmallscale]{upgreek}
\usepackage{xstring}
\usepackage{array,multirow,makecell}
\usepackage{natbib}
\usepackage{accents}
\usepackage{dsfont}
\usepackage{aliascnt}
\usepackage{cleveref}
\usepackage{booktabs}

\newcommand{\rat}[1]{\renewcommand{\arraystretch}{#1}}

\usepackage{mdframed}
\usepackage[export]{adjustbox}

\makeatletter
\newtheorem{theorem}{Theorem}
\crefname{theorem}{theorem}{Theorems}
\Crefname{Theorem}{Theorem}{Theorems}

\newaliascnt{lemma}{theorem}
\newtheorem{lemma}[lemma]{Lemma}
\aliascntresetthe{lemma}
\crefname{lemma}{lemma}{lemmas}
\Crefname{Lemma}{Lemma}{Lemmas}

\newaliascnt{corollary}{theorem}
\newtheorem{corollary}[corollary]{Corollary}
\aliascntresetthe{corollary}
\crefname{corollary}{corollary}{corollaries}
\Crefname{Corollary}{Corollary}{Corollaries}

\newaliascnt{proposition}{theorem}
\newtheorem{proposition}[proposition]{Proposition}
\aliascntresetthe{proposition}
\crefname{proposition}{proposition}{propositions}
\Crefname{Proposition}{Proposition}{Propositions}

\newaliascnt{definition}{theorem}
\newtheorem{definition}[definition]{Definition}
\aliascntresetthe{definition}
\crefname{definition}{definition}{definitions}
\Crefname{Definition}{Definition}{Definitions}

\newaliascnt{definition-proposition}{theorem}

\aliascntresetthe{definition-proposition}
\crefname{definition-proposition}{definition-proposition}{definitions-propositions}
\Crefname{Definition-Proposition}{Definition-Proposition}{Definitions-Propositions}

\newaliascnt{remark}{theorem}
\newtheorem{remark}[remark]{Remark}
\aliascntresetthe{remark}
\crefname{remark}{remark}{remarks}
\Crefname{Remark}{Remark}{Remarks}

\crefname{example}{example}{examples}
\Crefname{Example}{Example}{Examples}

\crefname{figure}{figure}{figures}
\Crefname{Figure}{Figure}{Figures}

\crefname{table}{table}{tables}
\Crefname{Table}{Table}{Tables}

\Crefname{assumption}{\textbf{H}\hspace{-4pt}}{\textbf{H}\hspace{-4pt}}
\crefname{assumption}{\textbf{H}}{\textbf{H}}

\Crefname{assumptionL}{\textbf{L}\hspace{-6pt}}{\textbf{A}\hspace{-6pt}}
\crefname{assumptionL}{\textbf{L}}{\textbf{L}}

\Crefname{assumptionB}{\textbf{B}\hspace{-3pt}}{\textbf{S}\hspace{-3pt}}
\crefname{assumptionB}{\textbf{B}}{\textbf{B}}

\Crefname{probleme}{\textbf{Problem}\hspace{-3pt}}{\textbf{Problem}\hspace{-3pt}}
\crefname{probleme}{\textbf{Problem}}{\textbf{Problem}}

\Crefname{assumptionG}{\textbf{G}\hspace{-3pt}}{\textbf{G}\hspace{-3pt}}
\crefname{assumptionG}{\textbf{G}}{\textbf{G}}

\newmdenv[
    rightline=true,
    bottomline=false,
    topline=false,
    leftline=false,
    rightmargin = 0pt,
    innertopmargin=0pt,
    innerbottommargin=0pt,
    innerrightmargin=0pt,
    innerleftmargin=0pt,
    linewidth=0.75pt]{rline}
\newmdenv[
    rightline=true,
    bottomline=false,
    topline=false,
    leftline=false,
    rightmargin = 0pt,
    innertopmargin=0pt,
    innerbottommargin=0pt,
    innerrightmargin=7pt,
    innerleftmargin=0pt,
    linewidth=0.75pt]{rline2}

\let\orgdescriptionlabel\descriptionlabel
\renewcommand*{\descriptionlabel}[1]{%
  \let\orglabel\label
  \let\label\@gobble
  \phantomsection
  \edef\@currentlabel{#1\unskip}
  \let\label\orglabel
  \orgdescriptionlabel{#1}
}

\usepackage{etoolbox}
\makeatletter
\patchcmd{\@makecaption}
  {\parbox}
  {\advance\@tempdima-\fontdimen2} 
  {}{}
\makeatother

\renewcommand{\epsilon}{\varepsilon}
\renewcommand{\P}{\operatorname{P}}
\newcommand{\Q}{\operatorname{Q}}
\newcommand{\argmin}{\operatorname{argmin}}

\def\bigone{{1}}

\newenvironment{proof}[1][{\textit{Proof:}}]{\begin{trivlist} \item[\em{\hskip \labelsep #1}]}{\ensuremath{\square} \end{trivlist}}

\renewcommand{\bf}[1]{\mathbf{#1}}
\newcommand{\E}{\mathsf{E}}
\newcommand{\Id}{\operatorname{I}}

\providecommand{\eps}{\varepsilon}

\renewcommand{\P}{\operatorname{P}}
\renewcommand{\Q}{\operatorname{Q}}

\newcommand{\divergence}{\operatorname{div}}

\renewcommand{\argmin}{\operatornamewithlimits{argmin}}

\newcommand{\R}{\mathbb R}
\newcommand{\X}{\mathcal{X}}

\newcommand{\TC}{\alpha}
\newcommand{\T}{\top}

\newcommand{\G}{\mathcal{G}}
\renewcommand{\H}{\mathcal{H}}

\def\DD{\operatorname{D}}

\def\xb{\mathsf{Z}} 
\def\yb{\mathsf{Y}} 
\newcommand{\likel}[1]{%
    \IfEqCase{#1}{%
        {l}{p_{\operatorname{log}}}%
        {p}{p_{\operatorname{pro}}}%
    }[\PackageError{logl}{Undefined option to probl: #1}{}]%
}%
\newcommand{\logl}[1]{%
    \IfEqCase{#1}{%
        {l}{\ell_{\operatorname{log}}}%
        {p}{\ell_{\operatorname{pro}}}%
    }[\PackageError{logl}{Undefined option to logl: #1}{}]%
}%
\newcommand{\pib}[1]{%
    \IfEqCase{#1}{%
    	{g}{\pi}
        {l}{\pi_{\operatorname{log}}}%
        {p}{\pi_{\operatorname{pro}}}%
    }[\PackageError{Ub}{Undefined option to Ub: #1}{}]%
}%
\newcommand{\Ub}[1]{%
    \IfEqCase{#1}{%
        {l}{U_{\operatorname{log}}}%
        {p}{U_{\operatorname{pro}}}%
    }[\PackageError{Ub}{Undefined option to Ub: #1}{}]%
}%

\newcommand{\hh}{\widehat{h}}
\newcommand{\sh}{h^*}

\newcommandx{\bVn}[3][1=]{\overline{V}^{\operatorname{#1}}_{n,#3}(#2)}
\newcommandx{\Vn}[3][1=]{{V}^{\operatorname{#1}}_{n,#3}(#2)}

\newcommandx{\ESV}[2][1=]{\overline{V}^{\operatorname{#1}}_{n}(#2)}
\newcommandx{\SV}[2][1=]{{V}^{#1}_{n}(#2)}
\newcommandx{\AV}[2][1=]{{V}^{#1}_{\infty}(#2)}
\def\rmd{\mathrm{d}}

\def\ie{i.e. }

\newcommandx\sequence[3][2=,3=]
{\ifthenelse{\equal{#3}{}}{\ensuremath{\{ #1_{#2}\}}}{\ensuremath{\{ #1_{#2}, \eqsp #2 \in #3 \}}}}
\newcommandx\sequencePar[3][2=,3=]
{\ifthenelse{\equal{#3}{}}{\ensuremath{\{ #1({#2})\}}}{\ensuremath{\{ #1({#2}), \eqsp #2 \in #3 \}}}}

\def\rset{\ensuremath{\mathbb{R}}}

\def\nset{\ensuremath{\mathbb{N}}}
\def\zset{\ensuremath{\mathbb{Z}}}

\def\eqsp{\,}
\newcommand{\eqdef}{\overset{\text{\tiny def}} =}

\newcommand{\an}[1]{\todo[color=green!20,inline]{{AN:} #1}}

\newcommandx{\norm}[2][1=]{\ifthenelse{\equal{#1}{}}{\left\Vert #2 \right\Vert}{\left\Vert #2 \right\Vert^{#1}}}

\newcommandx{\lawlim}[1][1=\mu]{\overset{\P_{#1}-\text{weakly}}{\underset{n\to+\infty}{\Longrightarrow}}}

\def\iid{i.i.d.}

\newcommandx{\varinf}[1][1=]{\ifthenelse{\equal{#1}{}}{\sigma^2_\infty}{\sigma^2_{\infty,#1}}}
\newcommand{\set}[2]{\left\{#1\,:\eqsp #2\right\}}

\newcommand{\Norm}{\mathcal{N}}

\DeclareMathOperator{\Med}{Med}

\def\ls{l}
\def\rme{\mathrm{e}}
\def\rmi{\mathrm{i}}
\newcommand{\tvnorm}[1]{\| #1 \|_{\operatorname{TV}}}
\def\Xset{\mathsf{X}}
\def\Cset{\mathsf{S}}
\def\Xsigma{\mathcal{X}}

\def\tf{\tilde{f}}
\def\th{\tilde{h}}

\newcommandx{\CPE}[3][1=]{\mathsf{E}_{#1}\left[\left. #2 \, \right| #3 \right]} 
\def\sPoif{\hat{f}}
\newcommand{\ltwo}{\ensuremath{\mathrm{L}^2}}
\newcommand\lp{\ensuremath{\mathrm{L}^p}}
\newcommandx{\PE}[2][1=]{\ensuremath{\mathsf{E}_{#1}\big[ #2 \big]}}
\newcommandx{\PElr}[2][1=]{\ensuremath{\mathsf{E}_{#1}\left[ #2 \right]}}

\newcommandx{\PVar}[1][1=]{\ensuremath{\operatorname{Var}_{#1}}}
\newcommandx{\covcoeff}[3][1=\pi]{\rho_{#1}^{(#2)}(#3)}
\newcommandx{\ecovcoeff}[3][1=]{\hat{\rho}_{#1}^{(#2)}(#3)}
\newcommand{\Vnorm}[2]{\ensuremath{\left\Vert#1\right\Vert_{#2}}}
\newcommandx{\VnormFunc}[3][1=]{\ensuremath{\|#2\|_{{#3}}^{#1}}}
\newcommand{\ps}[2]{\left\langle#1,#2 \right\rangle}

\newcommand{\coint}[1]{\left[#1\right)}

\def\driftfunc{W}
\newcommand{\KL}[2]{\operatorname{KL}(#1 \| #2)}

\makeatother
\def\lipU{\operatorname{L}_U}

\def\mU{\operatorname{m}_U}
\def\tmU{\tilde{\operatorname{m}}_U}
\def\kU{\operatorname{k}_U}
\def\MU{\operatorname{M}_U}

\bibpunct{[}{]}{,}{n}{,}{,}
\date{\vspace{-5ex}}
\newcommand{\fraca}[2]{#1/#2}

\newcommand{\indibr}[1]{\mathbbm{1}_{\{#1\}}}
\def\ULA{\operatorname{ULA}}
\def\MALA{\operatorname{MALA}}
\def\RWM{\operatorname{RWM}}
\def\ESVM{\operatorname{ESVM}}
\def\EVM{\operatorname{EVM}} 

\begin{document}

\title{Variance reduction for Markov chains with application to MCMC}

\author{D. Belomestny~\footnote{Duisburg-Essen University, Germany, and HSE University, Russia, \texttt{denis.belomestny@uni-due.de}. },\,  L. Iosipoi~\footnote{HSE University, Russia,   \texttt{iosipoileonid@gmail.com}. } \, E. Moulines~\footnote{Ecole Polytechnique, France, and HSE University, Russia, \texttt{eric.moulines@polytechnique.edu}.},\,   A. Naumov~\footnote{HSE University, Russia,   \texttt{anaumov@hse.ru}. }, and  S. Samsonov~\footnote{HSE University, Russia,  \texttt{svsamsonov@hse.ru}.}}
\maketitle

\abstract{In this paper we propose a novel variance reduction approach for additive functionals of Markov chains based on minimization of an estimate for the asymptotic variance of these functionals over suitable classes of control variates. A distinctive  feature of the proposed approach is  its ability to significantly reduce  the overall finite sample variance. This feature is theoretically demonstrated  by means of a deep non asymptotic analysis of a variance reduced functional as well as by  a thorough simulation study. In particular we apply our method to various MCMC  Bayesian estimation problems where it favourably compares to the existing variance reduction approaches. }
\section{Introduction}
Variance reduction methods  play nowadays a prominent role as a complexity reduction tool in simulation based numerical algorithms like Monte Carlo (MC) or Markov Chain Monte Carlo (MCMC).   Introduction to many of variance reduction techniques can be found in \citet{christian1999monte}, \citet{rubinstein2016simulation}, \citet{GobetBook}, and \citet{glasserman2013monte}. While variance reduction techniques for MC algorithms are well studied, MCMC   algorithms are still waiting for efficient variance reduction methods. Recently one witnessed a revival of interest in this area  with numerous applications to Bayesian statistics, see for example \citet{dellaportas2012control}, \citet{mira2013zero}, \citet{durmus:moulines:2019}, and references therein. The main difficulty in constructing efficient variance reduction methods for MCMC  lies in the dependence between the successive values   of the underlying Markov chain which can significantly increase  the overall variance and needs to be accounted for.
\par
Suppose that we wish to compute $\pi(f)\eqdef\PE[\pi]{f(X)}$, where $X$
is a random vector  with a distribution $\pi$ on   $\Xset\subseteq\mathbb{R}^{d}$ and
$f:\Xset\to\mathbb{R}$ with $f\in \ltwo(\pi)$.
Let \( (X_k)_{k \ge 0} \) be a time homogeneous Markov chain with values in \(\Xset\). 
Denote by $P$ its Markov kernel  and define for any bounded measurable function $f$ 
\begin{equation*}
Pf(x)= \int_\Xset P(x,\rmd y) f(y) \, , \quad x\in \Xset.
\end{equation*}
Assume that $P$ has the unique invariant distribution $\pi$, that is, 
$\int_\Xset \pi(\rmd x) P(x, \rmd y) = \pi(\rmd y)$.
Under appropriate conditions, the Markov kernel $P$ may be shown to converge 
to the stationary distribution \(\pi\), that is, for any $x \in \Xset$,
\begin{equation*}
\lim_{n \to \infty} \tvnorm{P^n(x,\cdot) - \pi}=0,
\end{equation*}
where  $\tvnorm{\mu - \nu} = \sup_{A \in \Xsigma} |\mu(A) - \nu(A)|$
and
{$\Xsigma$ is the Borel $\sigma$-field associated to $\Xset$}.
More importantly, under rather weak assumptions, the ergodic averages
\begin{equation*}
	\pi_n(f)\eqdef n^{-1} \sum_{k=0}^{n-1} f(X_k)
\end{equation*}
satisfy, for any initial distribution, a central limit theorem (CLT) of the form
\begin{eqnarray*}
	\sqrt{n}\bigl[\pi_n(f)-\pi(f)\bigr]
	=n^{-1/2}\sum_{k=0}^{n-1}\bigl[f(X_k)-\pi(f)\bigr]
	\xrightarrow{\mathcal{D}} 
	\mathcal{N}\bigl(0,\AV{f}\bigr)
\end{eqnarray*}
with the asymptotic variance \(\AV{f}\) given by
\begin{equation}
\label{V infinity}
	\AV{f}\eqdef\lim_{n\to \infty} n\, \E_\pi\bigl[\{\pi_n(f) - \pi(f)\}^2\bigr]= \E_\pi\bigl[\tf^2\bigr] + 2 \sum_{k=1}^\infty \E_\pi\bigl[ \tf P^k \tf\bigr],
\end{equation}
where $\tf= f - \pi(f)$. 
This motivates to use ergodic averages $\pi_n(f)$ as a natural estimate for $\pi(f)$.
For a broader discussion of the Markov chain CLT and conditions under which CLT holds, see
\citet{jones2004}, \citet{roberts2004}, and \citet{douc:moulines:priouret:2018}.
\par
One important and widely used class of  variance reduction methods for Markov chains is the method of control variates which is based on subtraction of a zero-mean random variable (control variate) from \(\pi_n(f)\). 
There are several methods to construct such control variates. 
If $\nabla\log\pi$ is known, one can use popular zero-variance control variates 
based on the Stein's identity, see \citet{Assaraf1999} and \citet{mira2013zero}.
A non-parametric extension of such control variates is suggested in \citet{oates2017control}
and \citet{oates2016convergence}.
Control variates can be also obtained using the Poisson equation. Namely,
it was observed by \citet{henderson1997variance}  that the function \(U_g \eqdef g-Pg\) has zero mean with respect to \(\pi\), provided that \(\pi(|g|) < \infty\).  Then the choice \(g=\sPoif\) with \(\sPoif\) satisfying the so-called Poisson equation \(\sPoif(x) - P \sPoif(x) = \tf(x)\)  leads to $f-U_{\sPoif}=f - \sPoif +P \sPoif= \pi(f)$ hence yielding a zero-variance control variate for the empirical mean under \(\pi.\)    Although the Poisson equation   involves the quantity of interest \(\pi(f)\)  and can not be  solved explicitly in most cases, the above idea still can be used to construct some  approximations for the zero-variance control variate \(\sPoif(x)-P \sPoif(x)\). For example,  \citet{henderson1997variance} proposed to compute approximations to the solution of the Poisson equation  for specific Markov chains with particular emphasis on models arising in stochastic network theory. In \citet{dellaportas2012control} and \citet{durmus:moulines:2019}  regression-type control variates are developed and studied for reversible Markov chains. It is assumed in \citet{dellaportas2012control}  that the one-step conditional expectations  can be computed analytically  for a set of basis functions. The authors in \citet{durmus:moulines:2019} proposed another approach tailored to diffusion setting which does  require the computation of integrals of basis functions and only involves  the application of the underlying differential generator.
\par
There is a fundamental issue related to the control variates method.
Since one usually needs to  consider a large class of control variates, one has to choose
a criterion to select the ``best'' control variate from this class. In the literature, such a choice is often based on the least squares criterion or on  the sample variance,
see, for example, 
\citet{mira2013zero}, \citet{oates2017control}, \citet{south2018regularised}.
Note that such  criteria can not properly take into account the correlation structure of the underlying Markov chain and hence can only reduce the first term in \eqref{V infinity}.
\par
In this paper, we propose a novel variance reduction method for Markov chains
based on the \textit{empirical spectral variance minimization}. The proposed method can be viewed as  a generalization of
the approach in \citet{biz2018,belomestny2017variance} to Markov chains.
In a nutshell, given a class of control variates $\G$,
that is, functions $g \in \G$ with $\pi(g)=0$  we consider the  estimator
\[
	\pi_{n}(f-\widehat{g}_n) \eqdef n^{-1}\sum_{k=0}^{n-1} \{f(X_k)-\widehat{g}_n(X_k)\}
\]
with $\widehat{g}_n \eqdef \argmin_{g \in \G} V_n(f - g)$, where  $V_n(f)$ stands for an estimator of the asymptotic variance \(V_\infty(f)\) defined in~\eqref{V infinity}.
This generalization turns out to be  challenging for at least two reasons. First, there is no simple way to estimate the asymptotic variance \(V_\infty(f)\) for Markov chains.
Due to  inherent serial correlation, estimating $V_\infty(f)$ requires specific techniques such as spectral and batch means methods; see \citet{flegal:jones:2010} for a survey on variance estimators and their statistical properties. 
Second, a nonasymptotic analysis of the estimate \(\widehat{g}_n\) is highly nontrivial and requires careful treatment. We perform this analysis for a rather general class of geometrically ergodic Markov chains including the well known Unadjusted Langevin Algorithm (ULA), Metropolis-Adjusted Langevin Algorithm (MALA) and Random Walk Metropolis (RWM).  In particular, we show that under some restrictions on \(\G\),  the rate of the excess for the asymptotic variance can be controlled with high probability as follows:
\[
	V_\infty(f-\widehat{g}_n) - \inf_{g \in \G} V_\infty(f-g) = O\left(n^{-\alpha}\right)
\]
for some \(\alpha\in [1/2,1).\)
Let us stress that our results are rather generic and can cover various types of control variates.
Apart from a comprehensive theoretical analysis we conduct an extensive simulation study including Bayesian inference via MCMC for logistic regression, {Gaussian mixtures} and Bayesian inference of ODE models. We show that for various MCMC algorithms our approach leads to a further significant  variance reduction as compared to the least-squares-type criteria.
\par
The paper is organised as follows. In \Cref{sec: evm} we introduce a general empirical variance minimisation procedure for Markov chains and analyse its properties. In \Cref{sec:mcmc} we apply our theoretical results to a widely used ULA and MALA. In \Cref{sec:num} we conduct a thorough numerical study of the proposed approach. Finally all proofs are collected in \Cref{sec:proofs} and \Cref{sec:appendix}.

\paragraph{Notations}
Let \(\| \cdot \|\) denote the standard Euclidean norm. We say that \(f: \R^d \rightarrow \R\) is \(L-\)Lipschitz function if  \(|f(x)- f(x')| \le L \|x - x'\|\) for any \(x, x' \in \mathbb{R}^{d}\).
\par
For any probability measure $\xi$ on $(\Xset,\Xsigma)$, we denote by $\P_{\xi}$  the unique probability under which $(X_n)_{n\geqslant 0}$ is a Markov chain with Markov kernel $P$ and initial distribution $\xi$. We denote by $\E_{\xi}$  the expectation under the distribution $\P_{\xi}$.
For $\xi$ a probability measure on $(\Xset,\Xsigma)$ and $A \in \Xsigma$, we denote by $\xi P(A)= \int \xi(\rmd x) P(x,A)$; for $h:\Xset \to \rset_+$ a  measurable function,
we denote by $P h(x)= \int P(x,\rmd y) h(y)$.
Given two Markov kernels $P$ and $Q$ on $\Xset \times \Xsigma$, where $\Xsigma$ is the Borel $\sigma$-field on $\Xset$,  we define $PQ (x,A)= \iint P(x,\rmd y) Q(y,A)$. We also define $P^n$ inductively by $P^n= P P^{n-1}$.
Let $\driftfunc: \Xset \to \coint{1,\infty}$ be a measurable function.
The $W$-norm of a function $h:\Xset\to\rset$ is defined as
$\| h \|_{\driftfunc} = \sup_{x \in \Xset} \{|h(x)|/\driftfunc(x)\}$.
For any two probability measures $\mu$ and $\nu$ on $(\Xset,\Xsigma)$ satisfying $\mu(\driftfunc) < \infty$ and $\nu(\driftfunc) < \infty$, the $W$-norm of $\mu-\nu$ is defined as $\Vnorm{\mu-\nu}{\driftfunc} = \sup_{\|f \|_{\driftfunc} \leq 1} |\mu(f) - \nu(f)|$.
\par
We  also use the 2-Wasserstein distance and the Kullback-Leibler divergence in our analysis.
The \(2\)-Wasserstein distance between probability measures \(\mu\) and \(\nu\)
is denoted by $W_2(\mu,\nu) \eqdef \inf_{\zeta} \bigl(  \int_{\Xset\times\Xset} \|x-y\|^2 \,  \rmd\zeta (x,y)\bigr)^{1/2}$, where the infimum is taken over all probability measures \(\zeta\) on the product space \(\Xset\times\Xset\) with marginal distributions \(\mu\) and \(\nu\).
The Kullback-Leibler divergence for \(\mu\) and \(\nu\) is
defined as \(\KL{\mu}{\nu}= \PE[\mu]{\log (\rmd\mu/ \rmd \nu)} \) if \(\nu \ll \mu\) and \(\KL{\mu}{\nu} = \infty\) otherwise. We  say that the probability measure \(\mu\) satisfies the transportation cost-information inequality $\operatorname{T}_2(\TC)$ if there is a constant \(\TC > 0\) such that for any probability measure \(\nu\)
\begin{equation}
\label{eq:transportation-cost-inequality}
W_2(\mu,\nu) \leq \sqrt{2 \TC \KL{\mu}{\nu}}.
\end{equation}
For a real-valued function $h$ on $\Xset\subset\R^d$ and a $\sigma$-finite measure $\lambda$ on $(\Xset,\Xsigma)$
we write $\| h \|_{\lp(\lambda)} =  (\int_{\Xset} |h(x)|^p \lambda(\rmd x))^{1/p}$ with $1\leq p < \infty$.
The set of all functions $h$ with $\| h \|_{\lp(\lambda)}<\infty$ is denoted by $\lp(\lambda)=\lp(\Xset,\lambda)$.
\par
Finally, the Sobolev space is defined as
$W^{s,p}(\Xset) = \left \{ u \in \lp(\lambda) : \DD^{\alpha}u \in \lp(\lambda), \ \forall |\alpha| \leqslant s \right \}$,
where $\lambda$ is the Lebesgue measure,
$\alpha = (\alpha_1,\ldots,\alpha_d)$  is a multi-index  with $|\alpha|=\alpha_1+\ldots+\alpha_d$, and
$\DD^{\alpha}$ stands for differential operator of the form
$\DD^{\alpha} = {\partial^{|\alpha|}}/{\partial x_1^{\alpha_1}\ldots\partial x_d^{\alpha_d}}$.
Here all derivatives are understood in the weak sense.
The weighted Sobolev space $W^{s,p}(\Xset,\langle x \rangle^{\beta})$
for a  polynomial weighting function $\langle x \rangle^{\beta}= (1+\|x\|^2)^{\beta/2}$, $\beta \in \R$, is defined by
\begin{align}
\label{eq:wss}
W^{s,p}(\Xset,\langle x \rangle^{\beta} ) = \bigl\{ u:\ u\cdot \langle x \rangle^{\beta} \in W^{s,p}(\Xset) \bigr\}.
\end{align}
The Sobolev norm is defined as
$\|u\|_{W^{s,p}(\Xset,\langle x \rangle^{\beta})} = \sum_{|\alpha| \leq s} \bigl\| \DD^{\alpha} \bigl(u  \langle x \rangle^{\beta} \bigr) \bigr\|_{\lp(\lambda)}$.
We say that $U \subset W^{s,p}(\Xset,\langle x \rangle^{\beta})$ is norm-bounded if
there exists $c>0$, such that $\|u\|_{W^{s,p}(\Xset,\langle x \rangle^{\beta})} \leq c$ for any $u\in U$.
\par
In what follows, we use the symbol \(\lesssim\) for inequality up to an absolute constant.

\section{Main results}
\label{sec: evm}
\subsection{Empirical spectral variance minimisation (ESVM)}
In this paper, we propose a novel approach to choose a control variate from the set $\G$   referred to as the
\textit{Empirical Spectral Variance Minimisation} (ESVM).
To shorten notation, let us denote by $\H=\H(\G)$  a class of functions $h(x) = f(x) - g(x)$, with $g\in\G$. The main idea of the ESVM approach  is to select a control variate which minimizes
a finite sample estimate for the asymptotic variance $\AV{h}$.
There are several estimates for \(V_\infty(h)\) available in the literature, see \citet{flegal:jones:2010}.
For the sake of clarity we consider only the spectral variance estimator which provides
the most generic way to estimate \(V_\infty(h)\). It is defined as follows.
Let $P$ be a Markov kernel admitting a unique invariant probability $\pi$ and set $\th \eqdef h - \pi(h)$ (assuming $\pi(|h|) < \infty$).
For $s \in \zset_+$, define the stationary lag $s$ autocovariance 
$\covcoeff[\pi]{h}{s}\eqdef \PE[\pi]{\th(X_{s}) \th(X_0)}$
and the lag $s$ sample autocovariance via
\begin{equation}
  \label{eq:empirical-autorcovariance}
 \ecovcoeff[n]{h}{s} \eqdef n^{-1} \sum_{k=0}^{n-s-1} \{h(X_k) - \pi_n(h)\} \{h(X_{k+s}) - \pi_n(h)\} ,
\end{equation}
where $\pi_n(h) \eqdef  n^{-1} \sum_{j=0}^{n-1} h(X_j)$. The spectral variance estimator is based on  truncation and weighting of the sample autocovariance function,
  \begin{equation}
  \label{eq:sv}
   \SV{h} \eqdef \sum_{s=-(b_n-1)}^{b_n-1} w_n(s) \ecovcoeff[n]{h}{|s|}\eqsp,
  \end{equation}
where $w_n$ is the \emph{lag window} and $b_n$ is the \emph{truncation point}.
The truncation point is a sequence of integers and
the lag window is a kernel of the form \(w_n(s)=w(s/b_n)\), where \(w\) is a symmetric non-negative function supported on \([-1,1]\) which fulfils $|w(s)| \leq 1$ for \(s \in [-1,1]\) and \(w(s)=1\) for \(s \in [-1/2,1/2]\).
Other possible choices of the lag window \(w_n\) can be considered, see \citet{flegal:jones:2010}.
In the ESVM approach we choose a control variate by minimizing the spectral variance
\begin{align}\label{eq:vm}
	\widehat{h} \eqdef \argmin_{h \in \H}  \SV{h}.
\end{align}
As the class \(\H\) can be too large making the resulting optimization problem
\eqref{eq:vm} computationally intractable,  we consider a smaller class.
Given $\eps>0$, let  $\H_\eps\subset\H$ consist of centres of the minimal $\eps$-covering net of $\H$ with respect
to the $\ltwo(\pi)$ distance. Further set
\begin{align}\label{eq:vmn}
	\widehat{h}_\eps \eqdef \argmin_{h \in \H_{\eps}} \SV{h}.
\end{align}
In what follows, we assume that $\H$ is a norm-bounded set in $\ltwo(\pi)$.
Hence the set \(\H_{\eps}\) is finite.
The estimates of the form \eqref{eq:vmn}
are referred to as skeleton or sieve estimates in the statistical literature
 (see, for example, \citet{wong95}, \citet{devroy95}, and \citet{geer00}).
\subsection{Theoretical analysis}
In this section, we analyze the proposed ESVM procedure in terms of the excess of the asymptotic variance.
Namely, we provide non-asymptotic bounds of the form:
\begin{equation}\label{eq:rates}
	V_\infty(\widehat{h}_\eps) - \inf_{h \in \H} V_\infty(h) = O\left(n^{-\alpha} \right),
	\quad 1/2<\alpha<1,
\end{equation}
holding with high probability.
\par
Before we proceed to theoretical results, let us define a quantity
which is used to choose a radius $\eps$ of the covering net $\H_\eps$ over which $\widehat{h}_\eps$ is computed.
Given any $\eps>0$, let \(H_{\ltwo(\pi)}(\mathcal{H},\eps)\) be a metric entropy of \(\mathcal{H}\) in \(\ltwo(\pi)\), that is, $H_{\ltwo(\pi)}(\mathcal{H},\eps) \eqdef \log |\H_\eps|$, where
$|\H_\eps|$ is cardinality of $\H_\eps$ (which is assumed to be finite).
Define by \(\gamma_{\ltwo(\pi)}(\mathcal{H},n)\) a so-called \emph{fixed point}
\begin{equation}
\label{eq:definition-fixed-point}
	\gamma_{\ltwo(\pi)}(\mathcal{H},n)
	\eqdef
	\inf \{ \eta>0: \
		H_{\ltwo(\pi)}(\mathcal{H},\eta)\leq {n}\eta^2
	\}.
\end{equation}
Note that a number $\eta>0$ satisfying $H_{\ltwo(\pi)}(\mathcal{H},\eta)\leq {n}\eta^2$ is finite because of monotonicity
of the metric entropy and the mapping $\eta\to {n}\eta^2$ in $\eta$. The quantity \(\gamma_{\ltwo(\pi)}(\mathcal{H},n)\) is used to control the cardinality of $\H_\eps.$
Indeed by choosing $\eps \ge \gamma_{\ltwo(\pi)}(\mathcal{H},n)$ we get $|\H_\eps|\leq \rme^{{n}\eps^2}$.
It is easily seen from the above definition that the fixed point  is a decreasing
function in $n$. Let us {discuss a typical behaviour of \(\gamma_{\ltwo(\pi)}(\mathcal{H},n)\) as \(n\to \infty\)} when $\H$ is a subset of the weighted Sobolev space $W^{s,p}(\Xset,\langle x \rangle^{\beta})$, see  \eqref{eq:wss} for definition.
The following result can be derived from \citet{nickl2007bracketing}.
\begin{proposition}\label{th:main_ex}
	Let $\H$ be a (non-empty) norm-bounded subset of $W^{s,p}(\R^{d}, \langle x \rangle^{\beta})$, where
	$1 < p <  \infty$, {$\beta \in \R$}, and $s - d/p > 0$. Let also for some $\alpha>0$,
	$\|\langle x \rangle^{\alpha - \beta}\|_{\ltwo(\pi)} < \infty$.
	Then it holds
	\[
	\gamma_{\ltwo(\pi)}(\H, n) \lesssim
	\begin{cases}
		n^{-\frac{1}{2+d/s}}  &\text{for } \alpha>s-d/p,\\
		n^{-\frac{1}{2+(\alpha/d+ 1/p)^{-1}}} &\text{for } \alpha < s - d/p.\\
	\end{cases}	
	\]	
\end{proposition}
Now let us turn to assumptions needed for \eqref{eq:rates} to hold.
Our first assumption is the \textit{geometric ergodicity} of the Markov chain \((X_k)_{k\geq 0}\).
Let \(W: \Xset \to [1,\infty)\) be a measurable function.
\begin{description}
\item[(GE)\label{GE}]
The Markov kernel $P$  admits a unique invariant probability measure \(\pi\) such that $\pi(W)<\infty$ and there exist  $\varsigma >0, 0< \rho < 1$ such that  for all $x \in \Xset$ and $n \in \nset,$
\begin{equation*}
\Vnorm{P^n(x,\cdot)-\pi}{\driftfunc}\leq \varsigma \driftfunc(x) \rho^{ n}.
\end{equation*}
\end{description}

\begin{description}
\item[(BR)\label{BR}]
There exist a non-empty set $\Cset \subset \Xset$ and real numbers $u>1, J>0$ and $\ls>0$ such that
\begin{equation}\label{eq: geom_int}
\sup_{x \in \Cset} \E_x[u^{-\sigma}] \le J \quad \text{ and } \quad \sup_{x \in \Cset} W(x) \le \ls,
\end{equation}
where $\sigma$ is the return time to the set $\Cset$.
\end{description}

\begin{remark}\label{DS_remark}
Let us introduce drift and small set conditions.
\begin{description}
\item[(DS)\label{DS}] The Markov kernel $P$ is irreducible, aperiodic and
\begin{itemize}
\item there exist measurable function \(W: \Xset \to [1,\infty)\), \(\lambda \in [0, 1)\), \(b < \infty\), and $\ls < \infty$ such that $\lambda + 2b/(1+\ls) < 1$ and
\begin{equation}\label{eq:def_driftcond}
P W \le \lambda W + b \indibr{W \le \ls}.
\end{equation}
\item there exist $m, \varepsilon>0$ such that for all $x, x' \in \{W \le \ls\}$, $\|\delta_x P^m - \delta_{x'} P \|_{\operatorname{TV}} \le 2(1 - \eps)$.
\end{itemize}
\end{description}
It follows from \citet[Theorem~19.5.1]{douc:moulines:priouret:2018}) that \ref{DS} implies \ref{GE} and by \citet[Proposition~14.1.2]{douc:moulines:priouret:2018}) \ref{DS} implies \ref{BR}.
Explicit expressions for the constants $\varsigma$ and $\rho$ may be found in~\citet[Theorem~19.4.1]{douc:moulines:priouret:2018}). Note also that \ref{GE} implies that $P$ is
  positive, aperiodic and condition \ref{DS} is satisfied for some
  small set $\Cset$ and some function $W_0$ verifying $W\leq W_0 \leq \varsigma_0 W$ and constants
  $\varsigma_0 < \infty$, $b_0 < \infty$, $\lambda_0 \in \coint{0,1}$. Hence \ref{GE} implies \ref{BR} for some constants $u > 1$ and $J>0$ (see \citet[Theorem~15.2.4]{douc:moulines:priouret:2018}.
\end{remark}
We also need a \textit{Gaussian concentration} for $\SV{h}$, which
requires an additional assumption on the class $\H$.
It is important to note
that $\SV{h}$ is a quadratic form of $(h(X_j))_{j=0}^{n-1}$.
As a result, without much surprise, concentration results
for the quadratic forms of Markov Chains shall play a key role in our analysis.
We shall consider below  two situations. While the first {situation} corresponds to bounded functions \(h,\) the second one deals with Lipschitz continuous functions \(h.\) In the second case we additionally assume a contraction in \(\operatorname{L}^2\)-Wasserstein distance.  Thus we assume either
\begin{description}
\vspace{-0.1cm}
\item[(B)\label{assumptionB}]  {\it Bounded case}: There exist $B>0$ such that \(\sup_{h\in\H}{|h|_{\infty}}\leq B\) with \({|h|_{\infty}}=\sup_{x\in \Xset}|h(x)|\)
\vspace{-0.45cm}
\end{description}
or
\par
\begin{description}
\vspace{-0.1cm}
\item[(L)\label{assumptionL}]  {\it Lipschitz case}: Functions $h\in\H$ are $L$-Lipschitz.
\vspace{-0.1cm}
\end{description}
together with
\begin{description}
\item[(CW)\label{assumptionW}]  The Markov kernel $P(x,\cdot)$ belongs to $\operatorname{T}_2(\TC)$ for any $x\in \Xset$
and some $\TC>0$. Moreover, there exists \(0 < r < 1\), such that \(
	W_2(P(x, \cdot),P(y, \cdot))  \leq r \| x - y \|
\) for any $x,y \in \Xset$.
\end{description}
The rate of convergence for the variance excess 
is given in the following theorem.
\begin{theorem}\label{th:main_slow}
Assume \ref{GE} and either \ref{assumptionL}+\ref{assumptionW} or \ref{assumptionB}+\ref{BR}. Set \(b_n=2(\log(\fraca{1}{\rho}))^{-1}\log(n)\) and take
	$\eps = \gamma_{\ltwo(\pi)}(\mathcal{H},n)$. Then for any \(\delta \in (0,1),\) there is \(n_0=n_0(\delta)>0\) such that for any $n \ge n_0$ and $x_0 \in \Xset_0$ with $\P_{x_0}-$probability at least $1-\delta,$ it holds
	\begin{equation*}\label{eq:th1}
		\AV{\hh_\eps} - \inf_{h\in\H}\AV{h}
		\lesssim
		 C_1 \log(n) \gamma_{\ltwo(\pi)}(\mathcal{H},n)
		+ C_2\frac{\log(n)\log(\fraca{1}{\delta})}{\sqrt{n}},
	\end{equation*}	
	where \(\lesssim\) stands for inequality up to an absolute constant,
	\begin{equation*}
		C_1 = \frac{K^2}{\log(\fraca{1}{\rho})}
		\quad
		C_2 =
		\frac{\varsigma^{1/2} (\pi(\driftfunc) + \driftfunc(x_0))}{(1-\rho)^{1/2}\log(\fraca{1}{\rho})}
		\biggl(K^2 + \sup_{h \in \H}\VnormFunc[2]{h}{\driftfunc^{1/2}}\biggr),
	\end{equation*}
        $\Xset_0 = \Xset$, $K^2 =  \sqrt{\TC}L^2/(1-r)$ under~\ref{assumptionL}+\ref{assumptionW}  and  $\Xset_0 = \Cset$, $K^2=\beta B^2$ under~\ref{assumptionB}+\ref{BR}, with
\begin{equation*}
	 \beta=\frac{\varsigma \ls}{1- \rho}\biggl(\frac1{\log u}+\frac{J \varsigma \ls }{1-\rho}\biggr)\enspace.
\end{equation*}
\end{theorem}
\par
In view of Proposition~\ref{th:main_ex},  \Cref{th:main_slow} may be summarized by saying that the excess variance $\AV{\hh_\eps} - \inf_{h\in\H}\AV{h}$ is bounded with high probability by a multiple of  $n^{-1/2+\eta}$
for some $\eta>0$ depending on the capacity of the class $\H$. In  statistical literature,
such rates are referred to as slow rates of convergence.
These rates can be improved by imposing additional conditions on $\H$.
To this end let consider the case when $\H$ contains a constant function.
Since \(\pi(h)=\pi(f)\) for all \(h\in \H,\) this constant must be equal to \(\pi(f)\),
and hence $\inf_{h \in \H} \SV{h}= 0$.
In this case, we obtain tighter bounds.
\begin{theorem}\label{th:main_fast}
Assume \ref{GE}, \ref{assumptionL}, and \ref{assumptionW}. Assume also that $\H$ contains a constant function $h^*(x) \equiv const$.
Fix the size of the lag window  \(b_n=2(\log(\fraca{1}{\rho}))^{-1}\log(n)\) and take $\eps = \gamma_{\ltwo(\pi)}(\mathcal{H},n)$.
Then for any \(\delta \in (0,1),\) there is \(n_0=n_0(\delta)>0\) such that for all $n \ge n_0$ and $x_0 \in \Xset$ it holds with $\P_{x_0}-$probability at least $1-\delta,$ 
	\begin{align}\label{eq:th1}
	\AV{\hh_\eps}
	\lesssim
	 C_1 \log(n) \gamma^2_{\ltwo(\pi)}(\mathcal{H},n)
	+C_2 \frac{\log(n)\log(\fraca{1}{\delta})}{n},
	\end{align}	
	where
	\begin{equation*}
		C_1 = \frac{\TC L^2}{(1-r)^2\log(\fraca{1}{\rho})}
		\quad\text{and}\quad
		C_2 =
		\frac{\TC L^2}{(1-r)^2\log(\fraca{1}{\rho})}+\frac{\varsigma (\pi(\driftfunc) + \driftfunc(x_0))}{(1-\rho)^{1/2}\log(\fraca{1}{\rho})}\,
		\sup_{h\in\H}\VnormFunc[2]{h}{\driftfunc^{1/2}}.
	\end{equation*}
\end{theorem}
\par
In view of Proposition~\ref{th:main_ex}, \Cref{th:main_fast} asserts that under an additional assumption that $\H$ contains a constant function,
the excess variance $\AV{\hh_\eps} - \inf_{h\in\H}\AV{h}$ can be bounded by a multiple of  $n^{-1+\eta}$
for some $\eta>0$ depending on  $\H$.
\par
\section{Application to Markov Chain Monte Carlo}
\label{sec:mcmc}
In this section we consider the application of the ESVM  approach to MCMC-type algorithms. The main goal of MCMC algorithms is to estimate expectations with respect to a probability measure $\pi$ on $\rset^d$, $d \geq 1$, with a density \(\pi\) of the form $\pi(x) =
 \rme^{-U(x)}/\int_{\rset^d} \rme^{-U(y)} \rmd y$ with respect to the Lebesgue
 measure, where \(U\) is a nonnegative potential.  Let $x^*$ be such that $\nabla U(x^*) = 0$ and without loss of generality we assume $x^* = 0$. Consider the following conditions on the potential \(U\).
\begin{description}
\item[(LD1)\label{H1}]
The function \(U\) is continuously differentiable on \(\R^d\) with Lipschitz continuous  gradient: there exists $\lipU > 0$ such that for all $x, y \in \rset^d$,
\[
\|\nabla U(x) - \nabla U(y)\| \leq \lipU \|x-y\| \,.
\]
\end{description}
\begin{description}
\item[(LD2)\label{H2}]
\(U\) is strongly convex:   there exists a constant ${m_U} > 0$, such that for all \( x,y \in \R^d\) it holds that
\[
U(y) \geq U(x) + \ps{\nabla U(x)}{y-x} + \mU\|x-y\|^2/2 \,.
\]
\end{description}
\begin{description}
\item[(LD3)\label{H3}]
There exist $K_1 \geq 0$ and $\tmU >0$ such that {for any $x \in \R^d$ with $\|x\|>K_1$}
   and any $y \in \rset^d$, $\ps{\DD^2 U(x) y }{y} \geq \tmU \norm[2]{y}$. Moreover, there exists $\MU \geq 0$ such that for any $x \in \rset^d$, $    \norm{\DD^3 U(x)} \leq \MU $.
  \end{description}
\paragraph{Unadjusted Langevin Algorithm}
The Langevin stochastic differential equation
 associated with $\pi$ is defined by
\begin{equation}
\label{eq:langevin_2}
\rmd Y_t = -\nabla U (Y_t) \rmd t + \sqrt{2} \rmd B_t \eqsp,
\end{equation}
where $(B_t)_{t\geq0}$ is the standard $d$-dimensional Brownian motion. Under mild technical
conditions, the Langevin diffusion admits $\pi$ as its unique
invariant distribution. We consider the sampling method based on the
Euler-Maruyama discretization of \eqref{eq:langevin_2}. This scheme referred to as {unadjusted} Langevin algorithm (ULA), defines the  discrete-time Markov chain
$(X_k)_{k \geq 0}$ given by
\begin{equation}
\label{eq:euler-proposal-2}
X_{k+1}= X_k - \gamma \nabla U(X_k) + \sqrt{2 \gamma} Z_{k+1} \eqsp,
\end{equation}
where $(Z_k)_{k \geq 1}$ is an \iid\ sequence of $d$-dimensional standard Gaussian
random variables and $\gamma>0$ is a  step size; see  \citet{roberts:tweedie:1996}.
We denote by $P^{\ULA}_\gamma$ the Markov kernel associated to the chain \eqref{eq:euler-proposal-2}.
It is known that under \ref{H1} and \ref{H2} or \ref{H3}, $P^{\ULA}_\gamma$ has a stationary distribution $\pi_\gamma$ which is close to $\pi$ (in a sense that one can bound the distance between $\pi_\gamma$ and $\pi$, e.g.,
in total variation and Wasserstein distances, see \citet{ad2014}, \citet{dm2017}).
\begin{proposition}
\label{thm: ULA-GE}
\begin{enumerate}
\item Assume \ref{H1}, \ref{H2}. Then for any $0< \gamma  <2/(\mU+\lipU)$, $P^{\ULA}_\gamma$ satisfies \ref{GE} {with the invariant distribution $\pi_\gamma$} and $\driftfunc(x) = \| x\|^2$. Moreover, $P^{\ULA}_\gamma$ fulfils \ref{assumptionW} with
\[
	\TC = 2 \gamma \quad \text{and}\quad r = \sqrt{1 - \gamma \kU},
\] where \(\kU \eqdef 2\mU \lipU/(\mU+\lipU)$.
\item Assume \ref{H1}, \ref{H3}.
 Then for any $0< \gamma  <\tmU/(4\lipU^2)$, $P^{\ULA}_\gamma$ satisfies \ref{GE}, \ref{BR} {with the invariant distribution $\pi_\gamma$},
 $\driftfunc(x) = \|x \|^2$, and $\Cset = \set{x\in\R^d}{\|x\| \leq R}$ with sufficiently large radius $R>0$.
\end{enumerate}
 \end{proposition}
\begin{proof}
\begin{enumerate}
\item For the proof of \ref{GE} see \citet[Proposition~2]{durmus:moulines:2018} and \citet[Theorem~12]{dm2017} and remark~\ref{DS_remark}.
To prove \ref{assumptionW} we observe that \(P^{\ULA}_\gamma(x,\cdot) = \Norm(x-\gamma \nabla U(x), 2\gamma \Id_d)\).
Hence, for all $\gamma > 0$, we get using \citet[Theorem~9.2.1]{GentilBakryLedoux}, $P^{\ULA}_\gamma(x,\cdot) \in \operatorname{T}_2(2\gamma)$,
that is \(P^{\ULA}_\gamma(x,\cdot)\) fulfils  \eqref{eq:transportation-cost-inequality}. Assuming that \ref{H1} and \ref{H2} hold, we may show using \citet[Proposition 3]{durmus:moulines:2018} that for any \(0 < \gamma \leq 2/(\mU+\lipU))\) and any \(x,y \in \Xset\), $W_2(P^{\ULA}_\gamma(x,\cdot),P^{\ULA}_\gamma(y,\cdot)) \leq \sqrt{1-\gamma \kU} \, d(x,y)$.
\par
\item See~\citet[Lemma 19 and Proposition 16]{durmus:moulines:2019}.
\end{enumerate}
\end{proof}

\paragraph{Metropolis Adjusted Langevin Algorithm (MALA)}
Here we consider a popular modification of ULA  called Metropolis Adjusted Langevin Algorithm (MALA). At each iteration, a new candidate $Y_{k+1}$ is proposed according to
\begin{equation}
\label{eq:ula-proposal}
Y_{k+1}= X_k - \gamma \nabla U(X_k) + \sqrt{2 \gamma} Z_{k+1} \eqsp,
\end{equation}
where $(Z_k)_{k \geq 1}$ is an \iid\ sequence of $d$-dimensional standard Gaussian
random vectors and $\gamma>0$ is a  step size. 
This proposal is accepted with probability $\alpha(X_k, Y_{k+1})$, where
\[
\alpha(x, y) \eqdef \min \Biggl (1, \frac{\pi(y) q_\gamma(y, x)}{\pi(x) q_\gamma(x, y)} \Biggr),
\]
where $q_\gamma(x,y) = (4\pi \gamma)^{-d/2} \exp(-\|y - x + \gamma \nabla U(x)\|^2/(4\gamma))$.
We denote by $P^{\MALA}_\gamma$ the Markov kernel associated to the MALA chain.
\begin{proposition}
\label{thm: MALA-GE}
Assume \ref{H1}, \ref{H3}.
Then there exists $\overline \gamma>0$ such that for any $\gamma \in [0, \overline \gamma]$, $P^{\MALA}_\gamma$ satisfies \ref{GE}, \ref{BR} 
{with the invariant distribution $\pi$}, \(\driftfunc(x) = \|x\|^2$, and  $\Cset = \set{x\in\R^d}{\|x\| \leq R}$ with sufficiently large radius $R>0$.
\end{proposition}
\begin{proof}
See~\citet[Proposition 21 and 23]{durmus:moulines:2019}.
\end{proof}

\paragraph{Random Walk Metropolis (RWM)}
At each iteration, a new candidate $Y_{k+1}$ is proposed according to
\begin{equation}
\label{eq:rwm-proposal}
Y_{k+1} = X_k + \sqrt{\gamma} Z_{k+1}, 
\end{equation}
where $(Z_k)_{k \geq 1}$ is an \iid\ sequence of $d$-dimensional standard Gaussian
random vectors and $\gamma>0$. This proposal is accepted with probability $\alpha(X_k,Y_{k+1})$, where
\[
\alpha(x, y) = \min \bigl (\fraca{\pi(y)}{\pi(x)}, 1 \bigr)
\]
We denote by $P^{\RWM}_{\gamma}$ the Markov kernel associated to the RWM chain.
Assumption \ref{GE} is discussed in~\citet{RobTweed1996} and~\citet{JarHansen2000} under various conditions. In particular the following result for super-exponential densities holds.
\begin{proposition}
Assume \ref{H1}, \ref{H3}. Then $P^{\RWM}_{\gamma}$ satisfies \ref{GE}, \ref{BR} 
{with the invariant distribution $\pi$}, 
$\driftfunc(x) = c \pi^{-1/2}(x)$ for some $c > 0$, and $\Cset = \set{x\in\R^d}{\|x\| \leq R}$ with sufficiently large radius $R>0$.
\end{proposition}
\begin{proof}
See~\citet[Theorem 4.2]{JarHansen2000}.
\end{proof}

\section{Numerical study}
\label{sec:num}
In this section we study numerical performance of the ESVM method for simulated and real-world data. Python implementation is available at \url{https://github.com/svsamsonov/esvm}.
\par
Following \citet{Assaraf1999}, \citet{mira2013zero},~\citet{oates2010stein}, we choose $\mathcal{G}$ to be a class of Stein control variates of the form
\begin{equation}
\label{eq:stein}
g_\Phi= - \ps{\Phi}{\nabla U} + {\divergence(\Phi)},
\end{equation}
where $\Phi: \Theta \to \rset^d\) with \(\Theta \subset \rset^d$, {$\divergence(\Phi)$ is the divergence of $\Phi$}, and $U$ is the potential associated with $\pi$, that is, $\pi(x) \propto \rme^{-U(x)}$, see~\Cref{sec:mcmc}.  Under \ref{H1} and \ref{H2}, for continuously differentiable functions $\Phi$, $\pi(g_\Phi)=0$, see \citet[Lemma 1]{oates2010stein}. This suggests to consider a class $\H= \set{h = f - g_\Phi}{g_\Phi \in \G}$. 
Our standard choice will be $\Phi(x) = b$ or  $\Phi(x)= A x + b$, where $A \in \R^{d \times d}$ is a matrix and $b \in \R^d$ is a vector. They will be referred to as the first- and second-order control variates respectively. It is worth noting that polynomial-based control variates are not exhaustive and one can use other control variates.
For instance, in the Gaussian mixture model considered below, polynomial-based control variates do not fit structure of the problem, so a class of radial basis functions will be used.
\par
In the ESVM method, we choose the trapezoidal non-negative kernel $w$ supported on $[-1,1]:$
\begin{equation}
\label{eq:kernel_coeffs}
w(s) = \begin{cases}
	2s+2, & -1 \leq s < -1/2, \\
	1, & -1/2 \leq s \leq 1/2,\\
	-2s + 2, &1/2 < s \leq 1. \\
\end{cases}
\end{equation}
{Our experiments with other kernels, for instance, $w(s) = \frac{1}{2} + \frac{1}{2}\cos{\pi s}$ did not reveal any sensitivity of ESVM to a particular kernel choice. In fact, even the simplest kernel $w(s) = \indibr{|s| \leq \frac{1}{2}}$ showed results comparable with ones for $w(s)$ given in \eqref{eq:kernel_coeffs}. Another parameter of ESVM to be chosen 
is the lag-window size $b_n$. In practice, it is not convenient to choose $b_n$ according to \Cref{th:main_slow} and \Cref{th:main_fast}, since it involves parameters of the Markov chain which are not usually available. Therefore, we choose $b_n$ by analyzing the sample autocorrelation function (ACF) of the Markov chain, see discussion below. Moreover, our experiments show that ESVM is not much sensitive to particular choice of $b_n$. For a wide range of possible values our procedure shows reasonably good performance.}  
\par
Numerical study is organized as follows. First we use ULA, MALA, or RWM algorithm to sample a training trajectory
of the size $n=n_{\text{burn}}+n_{\text{train}}$.
We consider the first $n_{\text{burn}}$ observations as a burn-in period, and exclude them from  subsequent computations. Then 
we compute optimal parameters $\hat{A}_{\ESVM}$, $\hat{b}_{\ESVM}$ which minimise the spectral variance $\SV{h}$ with $n=n_{\text{train}}$ 
and obtain the resulting control variate $\hat{h}_{\ESVM}$. For comparison purposes, we also compute parameters $\hat{A}_{\EVM}$, $\hat{b}_{\EVM}$ 
based on minimisation of the empirical variance {$V_n'(h) = (n-1)^{-1}\sum_{k=0}^{n-1} \{h(X_k) - \pi_n(h)\}^2$ with $n=n_{\text{train}}$}
and obtain the corresponding control variate $\hat{h}_{\EVM}$.
Variance reduction using $\hat{h}_{\EVM}$ will be referred to as the EVM algorithm, see \citet{belomestny2017variance}, \citet{mira2013zero}, and \citet{papamarkou2014}.
We use the BFGS optimisation method to find the optimal parameters for both ESVM and EVM algorithms.
\par
To evaluate performance of ESVM and EVM, we use the same MCMC algorithm to sample 
$N_{\text{test}} = 100$  independent training trajectories of size $n=n_{\text{burn}}+n_{\text{test}}$. Then for each trajectory we exclude first $n_{\text{burn}}$ observations and compute three different estimates for $\pi(f)$: (i) vanilla estimate (ergodic average of $f$ without variance reduction); (ii) EVM estimate (ergodic average of $\hat{h}_{\EVM}$); (iii) ESVM estimate (ergodic average of $\hat{h}_{\ESVM}$).
For each test trajectory, we define the Variance Reduction Factors (VRF) as the ratios {$\fraca{\SV{f}}{\SV{\hat{h}_{\ESVM}}}$ or $\fraca{\SV{f}}{\SV{\hat{h}_{\EVM}}}$ with $n=n_{\text{test}}$}. We report the average VRF over $N_{\text{test}}$ trajectories together with the corresponding boxplots of ergodic averages. On these boxplots we display the lower and upper quartiles for each estimation procedure. We will refer to the methods based on the first-order control variates as ESVM-1 and EVM-1, and for the second-order ones as ESVM-2 and EVM-2, respectively.
The values $b_n$, $n_{\text{burn}}$, $n_{\text{train}}$, $n_{\text{test}}$ together with parameters of MCMC algorithms for each example considered below are presented in \Cref{sec:Fig_Tables}, \Cref{tab:Table_setup}.

\paragraph{Gaussian Mixture Model (GMM)\label{sec:GMM}}
Let $\pi$ be a mixture of two Gaussian distributions, that is,
$\pi= \rho\Norm(\mu, \Sigma) + (1- \rho)\Norm(-\mu, \Sigma)$ for
\(\rho \in [0,1]\).
It is straightforward to check that \ref{H1} holds. If
\(\mu\) and \(\Sigma\) are such that \(\|\Sigma^{-1}\mu\|^2 \leq \lambda_{\text{min}}(\Sigma^{-1})\), the density \(\pi\) satisfies \ref{H2}. Otherwise, we have \ref{H3}.
\par
We set \(\rho = 1/2\), \(d = 2\), \(\mu = (0.5,0.5)^{\T}\),  and consider two instances of the covariance matrix: \(\Sigma = \Id\) and \(\Sigma = \Sigma_0\), where \(\Sigma_0\) is a randomly initialised symmetric matrix with \(\lambda_{\text{min}}(\Sigma_0)\geq 0.1\). The quantities of interest are $\E_{\pi}[X_1]$ and $\E_{\pi}[ X_1^2]$. 
\par
{
First let us briefly discuss how one can choose the lag-window size $b_n$.
Let us look at the sample ACF plot of the first coordinate given in \Cref{fig:gmm_truncation}. 
One may observe that ACF decreases fast enough for any MCMC algorithm, and it seems reasonable to set $b_n = 50$ or close to it. 
Moreover, we analyse performance of ESVM for different choices of $b_n$ by running the ULA algorithm to estimate $\E_{\pi}[X_1]$
and letting   $b_n$ to run  over the values from $1$ to $5000$. The corresponding VRFs are given also in \Cref{fig:gmm_truncation}.
Here, to compute the spectral variance over test trajectories, we use fixed $b_n^{\text{test}} = n^{1/3}_{\text{test}}$, no matter which value of $b_n$ was used during the training. Note that even for $b_n = 1$ on train (that is, taking into account only the first-order autocovariance) ESVM outperforms EVM, and for values $b_n \in [10, 1000]$ we observe  the optimal performance of ESVM.}
\par
Numerical results for estimating $\E_{\pi}[X_1]$ are presented in \Cref{tab:Table_GMM_d_2}. The corresponding boxplots for $\E_{\pi}[X_1]$ are given in \Cref{fig:gmm_1st},
and for $\E_{\pi}[ X_1^2]$ are given in \Cref{sec:Fig_Tables}, \Cref{fig:gmm_white_2nd} and \Cref{fig:gmm_coloured_2nd}. 
For the sake of convenience, all the estimates are centred by their analytically computed expectations. Note that ESVM outperforms EVM  in  both cases \(\Sigma = \Id\) and \(\Sigma = \Sigma_0\) and for all samplers used.

\begin{figure}[!htb]\centering
\begin{minipage}{0.88\linewidth}
\captionsetup{font=small}
\caption{GMM with \(\Sigma = \Sigma_0\). Left: Sample autocorrelation function for $X_1$. Right: average variance reduction factors for different choices of $b_n$.}\label{fig:gmm_truncation} 
  \includegraphics[width=0.49\linewidth]{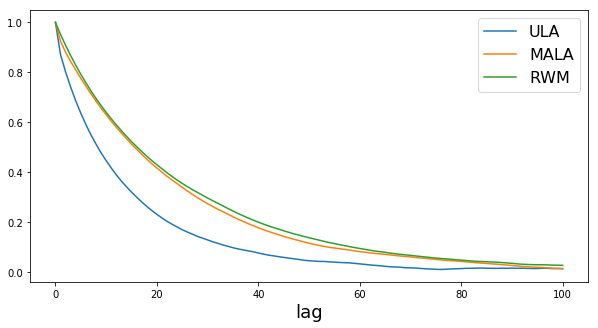}
  \includegraphics[width=0.49\linewidth]{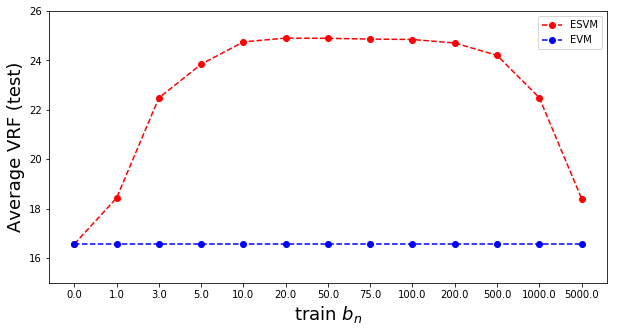}
\end{minipage}
\end{figure}

\begin{table}[!htb]\centering
\rat{1.1}
\captionsetup{font=small}
\caption{Variance Reduction Factors in GMM with \(\Sigma = \Id\) (top) and \(\Sigma = \Sigma_0\) (bottom).}\label{tab:Table_GMM_d_2}\vspace{-0.1cm}
\begin{minipage}{0.88\linewidth}\centering
\resizebox{0.7\textwidth}{!}{
\begin{tabular}{@{}lcccccccc@{}}
\toprule
 & \phantom{abc} &  \multicolumn{3}{c}{$\E_\pi[X_1]$} & \phantom{abc} & \multicolumn{3}{c}{$\E_\pi[X_1^2]$} \\
\cmidrule{3-5} \cmidrule{7-9} 
Method & \phantom{abc} & ULA  & MALA & RWM  &  & ULA  & MALA & RWM  \\
\toprule
ESVM &&  $\bf{9.1}$ & $\bf{6.1}$ & $\bf{8.2}$&&  $\bf{609.2}$ & $\bf{319.6}$ & $\bf{531.2}$\\
EVM && $4.5$ & $3.6$ & $5.3$ && $607.8$ & $316.3$ & $528.7$\\
\bottomrule
\end{tabular}}
 \vspace{0.8em} \\
\resizebox{0.7\textwidth}{!}{\centering
\begin{tabular}{@{}lcccccccc@{}}
\toprule
 & \phantom{abc} &  \multicolumn{3}{c}{$\E_\pi[X_1]$} & \phantom{abc} & \multicolumn{3}{c}{$\E_\pi[X_1^2]$} \\
\cmidrule{3-5} \cmidrule{7-9} 
Method & \phantom{abc} & ULA  & MALA & RWM  &  & ULA  & MALA & RWM  \\
\toprule
ESVM &&  $\bf{24.6}$ & $\bf{7.9}$ & $\bf{22.2}$ &&  $\bf{15.2}$ & $\bf{9.4}$ & $\bf{15.3}$ \\
EVM && $16.5$ & $7.5$ & $14.3$  &&  $9.2$ &  $5.0$ & $9.3$\\
\bottomrule
\end{tabular}}
\end{minipage}
\end{table}

%

\begin{figure}[!htb]\centering
\begin{minipage}{0.88\linewidth}
\captionsetup{font=small}
\caption{Estimation of $\E_\pi[X_1]$ in GMM with \(\Sigma = \Id\) (top row) and \(\Sigma = \Sigma_0\) (bottom row).
In each row  boxplots are given for ULA, MALA, and RWM, respectively.} \label{fig:gmm_1st}\vspace{-0.1cm}
  \includegraphics[width=0.328\linewidth]{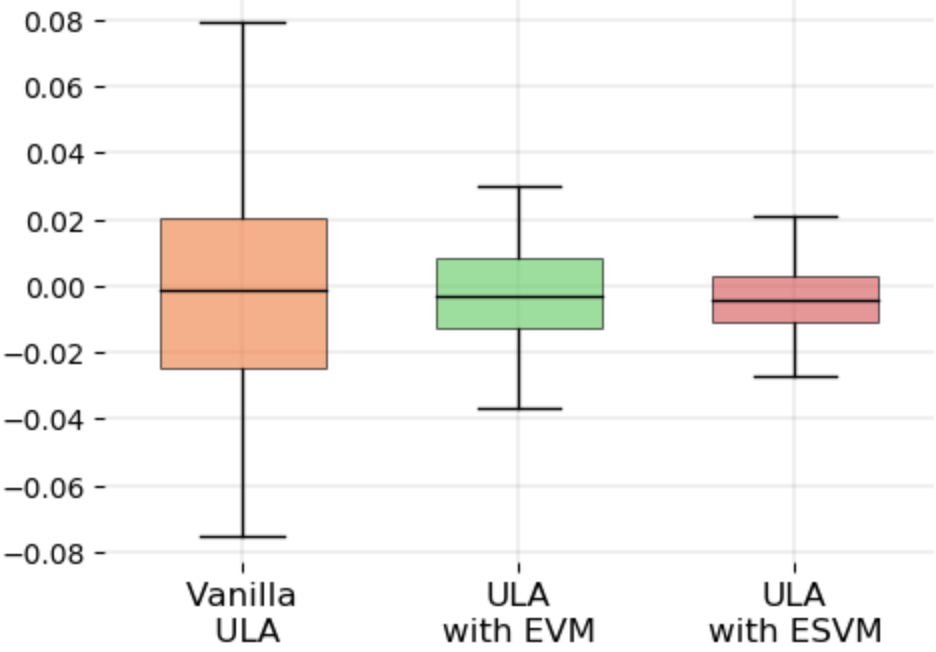}
  \includegraphics[width=0.328\linewidth]{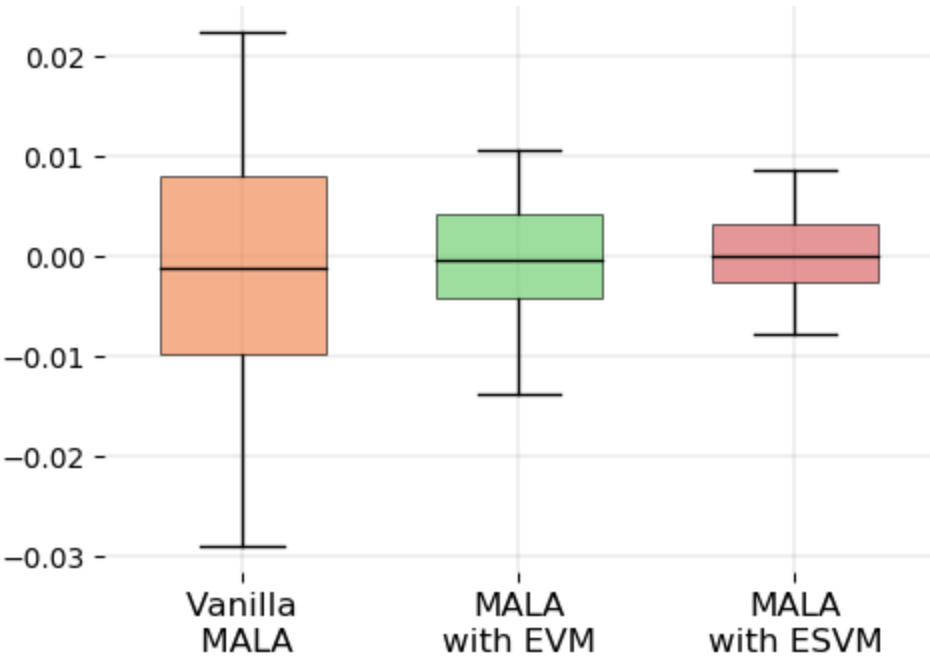}
  \includegraphics[width=0.328\linewidth]{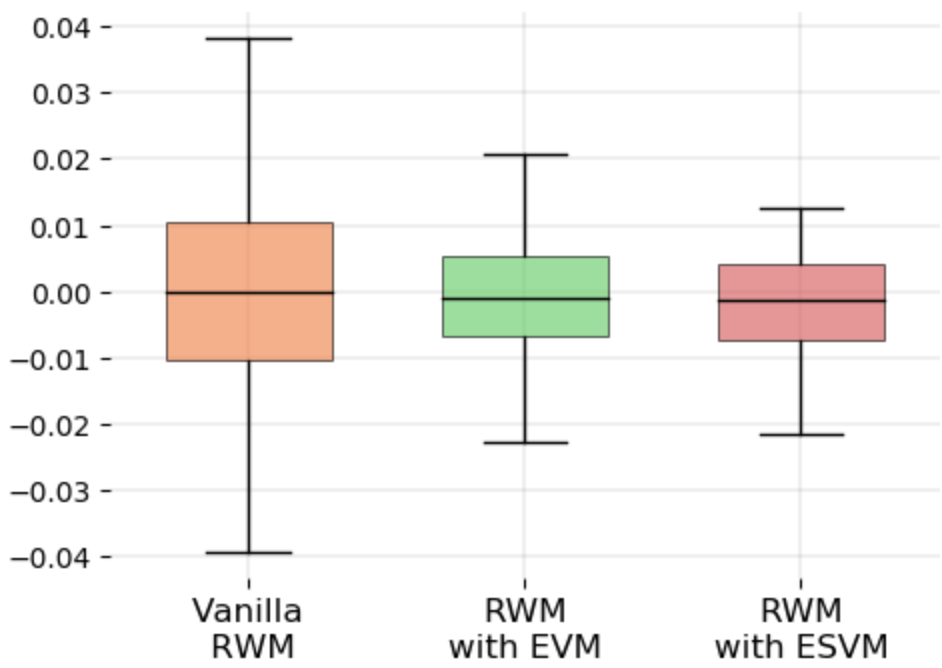}
  
  \includegraphics[width=0.328\linewidth]{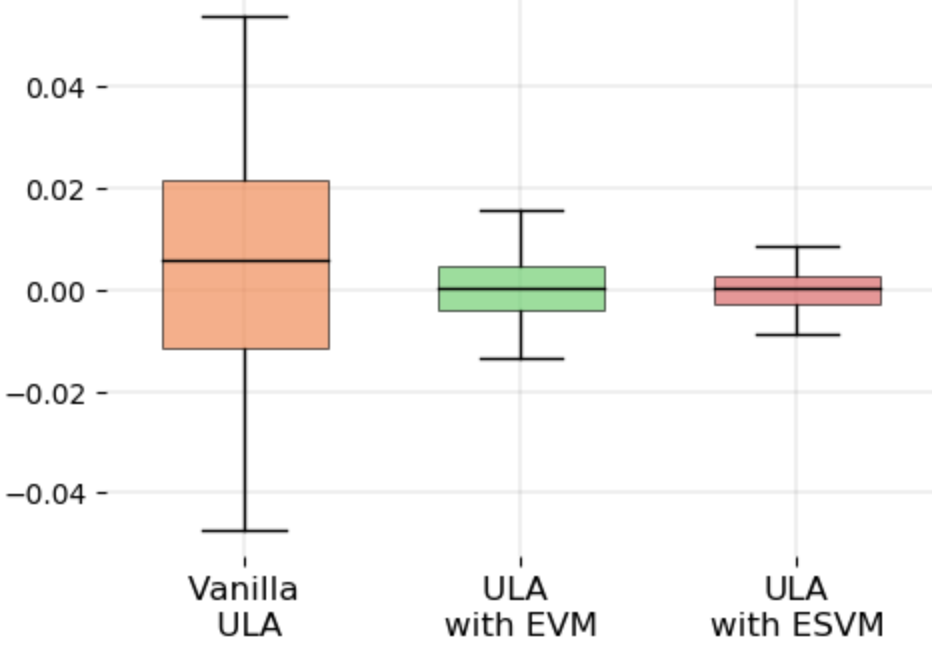}
  \includegraphics[width=0.328\linewidth]{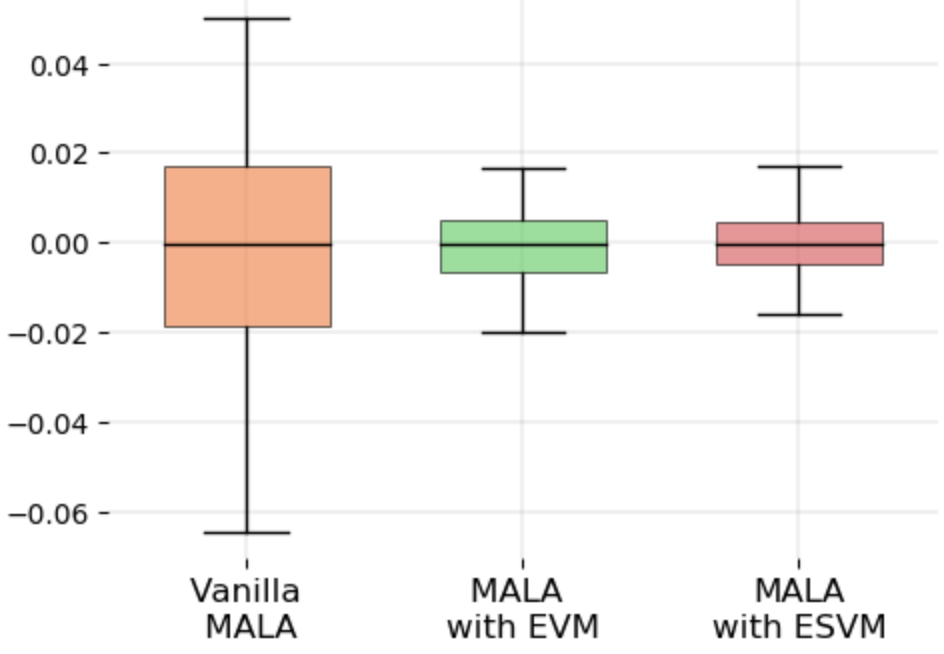}
  \includegraphics[width=0.328\linewidth]{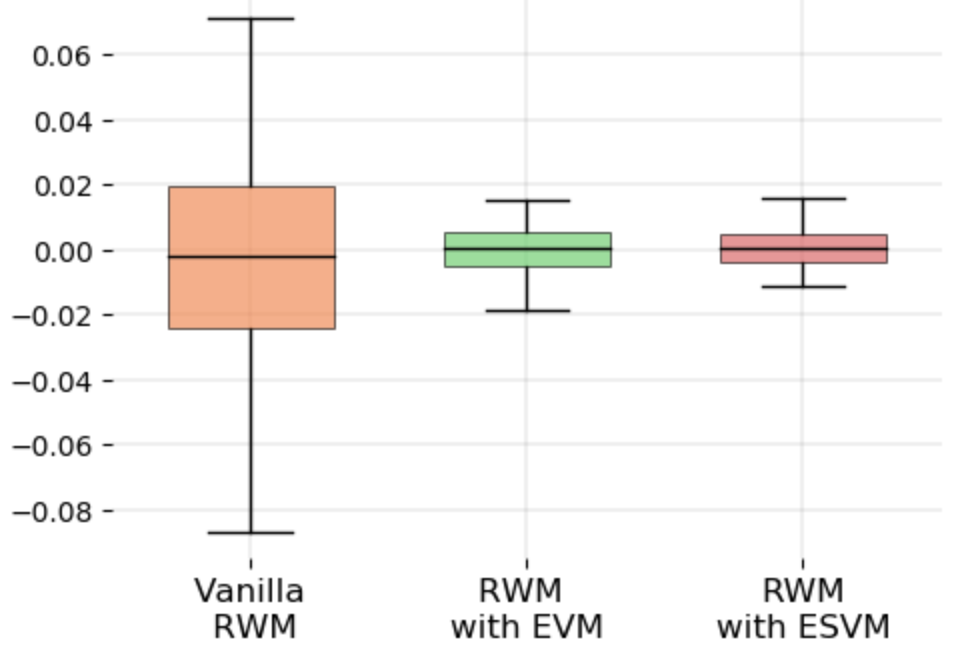}
\end{minipage}
\end{figure}

\paragraph{Gaussian Mixture with isolated modes\label{sec:GMM_isolated}}
{Let us now consider the Gaussian mixture model with different means and covariates,
$\pi= \rho\Norm(\mu_1, \sigma_1) + (1- \rho)\Norm(-\mu_2, \sigma_2)$ with
\(\rho \in [0,1]\). For simplicity, we let $d = 1$.
We are interested in the case when $|\mu_1 - \mu_2| \gg \max\{\sigma_1,\sigma_2\}$. 
When sampling from $\pi$ using ULA, MALA, or RWM, the corresponding Markov chain tends to  ``stuck" at the modes of density $\pi$, which leads to slow convergence.}
We are going to compare the results obtained using ESVM and EVM with the ones from \citet{mijatovic2018} based on a discretized Poisson equation. For comparison purposes, we will reproduce experiments from the aforementioned paper, see Section 5.2.1, and refer to the reported variance reduction factors.
\par
{
Our aim is to estimate $\pi(f)$ with $f(x) = x^3$.
We fix $\rho = 0.4$, $\mu_1 = -3$, $\mu_2 = 4$, $\sigma_1 = 1$, $\sigma_2 = 0.5$,
and use RWM with step size $\gamma = 1.0$ as a generating procedure. 
Results for the second-order control variates (our standard choice) are reported in Table~\ref{tab:Table_gmm_1d}, 
showing that this class of functions $\Phi$ does not allow us to achieve comparable to \citet{mijatovic2018} variance reduction factors. 
Let us consider instead the following set of radial basis functions  
\begin{equation}
\label{eq:cv-gmm}
\Phi(x) = \sum_{k=1}^{r}a_k(x-b_k)\exp\left(-\frac{(x-b_k)^2}{2}\right),
\end{equation}
where $a_k$, $b_k \in \mathbb{R}$, $k = 1,\ldots,r$. 
The ESVM algorithm with control variates determined by $\Phi(x)$ from \eqref{eq:cv-gmm}
will be referred to as the ESVM-r algorithm.
Results for ESVM-r are also given in Table~\ref{tab:Table_gmm_1d} showing comparable results with the Poisson-based approach from \citet{mijatovic2018}
(it is referred to as the Poisson-CV) and even outperforming it for large enough train sample size $n_{\text{train}}$ and number of basis functions $r$.}

\begin{table}[!htb]\centering
\rat{1.1}
\begin{minipage}{0.88\linewidth}\centering
\captionsetup{font=small}
\caption{Variance Reduction Factors in GMM with isolated modes.}
\label{tab:Table_gmm_1d} 
\resizebox{\textwidth}{!}{
\begin{tabular}{@{}lcccccccc@{}}
\toprule
$n_{\text{train}}$  	&\phantom{abc} &EVM-2& ESVM-2& Poisson-CV&ESVM-r, $r = 4$ &  ESVM-r, $r = 10$ & ESVM-r, $r = 20$\\ 
\toprule
$10^4$ 				&& $1.03$ 	& $1.04$ 	& up to $8900$	  &  $95.8$	 & $6457.2$            & $265382.8$  \\ \hline
$10^5$  				&& $1.92$ 	& $1.20$ 	& up to $13200$ & $98.8$ & $7176.5$ & $378249.0$    \\ 
\bottomrule
\end{tabular}}
\end{minipage}
\end{table}

\paragraph{Banana-shape density\label{sec:Banana}}
The ``Banana-shape" distribution, proposed by \citet{Haario1999}, can be obtained from a \(d\)-dimensional Gaussian vector with zero mean and covariance $\mathrm{diag}(p,1,\ldots,1)$ by applying transformation
\[
\varphi_b(x): \R^d \rightarrow \R^d, \quad \varphi(x) = (x_1,\,x_2 + b x_1^2 - p b,\,x_3,\ldots,x_d), 
\]
where $p>0$ and $b > 0$ are parameters; here $b$ controls the curvature of density's level sets. The potential $U$ is given by
\[
U(x_1,\ldots,x_d) = \fraca{x_1^2}{2p} + (x_2 + bx_1^2 - p b)^2 + \sum\nolimits_{k=3}^{d}\fraca{x_k^2}{2}.
\]
As can be easily seen,  the assumption (H3) holds. As to the assumption (H1), it is fulfilled only locally.
The quantity of interest is $\E_\pi[X_2]$. In our simulations, we set $p=100$, $b= 0.1$ and consider \(d=2\) and \(d=8\).  
VRFs are reported in Table~\ref{tab:Table_banana}. Boxplots for $d=8$ are shown in Figure~\ref{fig:banana}. In this problem, ESVM significantly outperforms EVM both for \(d=2\) and \(d=8\). Because of the curvature of the level sets, the step sizes in all considered methods should be chosen small enough, leading to highly correlated samples. This explains a  poor performance of the EVM method in this context.


\begin{table}[htb]\centering
\rat{1.1}
\begin{minipage}{0.88\linewidth}\centering
\captionsetup{font=small}
\caption{Estimation of $\E_\pi[X_2]$ for the banana-shaped density in $d=2$ and $d=8$.}\label{tab:Table_banana}
\resizebox{0.65\textwidth}{!}{
\begin{tabular}{@{}lcccccccc@{}}\toprule
 & \phantom{abc} &  \multicolumn{3}{c}{\(d=2\)} & \phantom{abc} & \multicolumn{3}{c}{\(d=8\)} \\
\cmidrule{3-5} \cmidrule{7-9} 
Method & \phantom{abc} & ULA  & MALA & RWM  &  & ULA  & MALA & RWM  \\
\toprule
ESVM &&  $\bf{4.7}$ & $\bf{2.7}$ & $\bf{42.4}$ &&  $\bf{5.3}$ & $\bf{6.5}$ & $\bf{18.5}$  \\
EVM &&  $1.4$ & $1.3$ & $1.5$  &&  $1.4$ & $4.6$ & $1.7$ \\
\bottomrule
\end{tabular}}
\end{minipage}
\end{table}

\begin{figure}[!htb]\centering
\begin{minipage}{0.88\linewidth}
\captionsetup{font=small}
\caption{Estimation of $\E_\pi[X_2]$ for the banana-shape density in $d = 8$.
Boxplots are given for ULA, MALA, and RWM respectively.}\label{fig:banana}
  \includegraphics[width=0.328\linewidth]{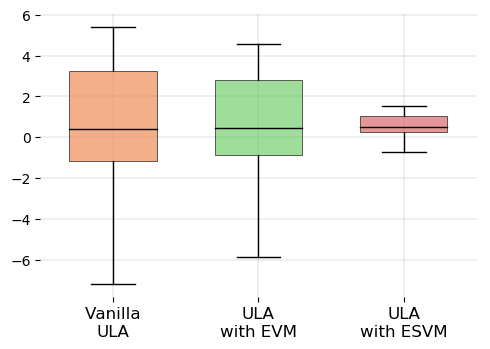}
  \includegraphics[width=0.328\linewidth]{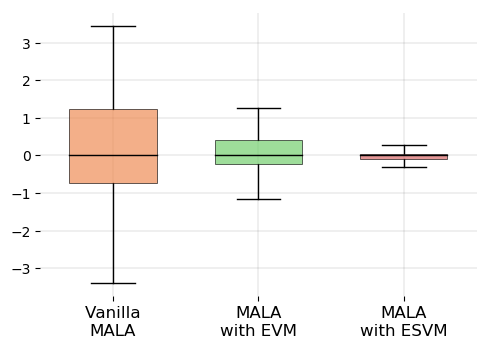}
  \includegraphics[width=0.328\linewidth]{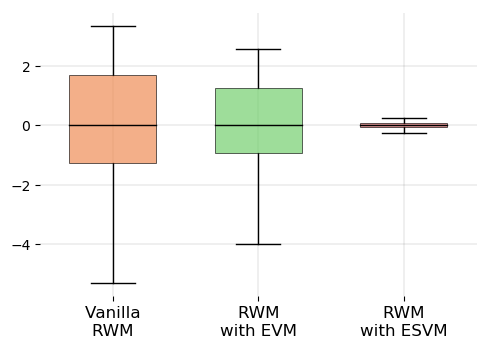}
\end{minipage}
\end{figure}

\paragraph{Logistic and probit regression\label{sec:Regression}}
Let $\yb=(\yb_1,\ldots,\yb_n)\in \{0,1\}^n$ be a vector of binary response variables, $x \in\rset^d$ be a vector of regression coefficients, and $\xb \in \rset^{N \times d}$ be a design matrix. 
The log-likelihood and likelihood of $i$-th point for the logistic and probit regression are given by
\begin{align*}
  \logl{l}(\yb_i | x, \xb_i) & = \yb_i  \xb_i^{\T} x -  \ln(1+\rme^{\xb_i^{\T} x}), \quad \likel{l}(\yb_i | x,\xb_i) = \exp(\logl{l}(\yb_i | x, \xb_i)), \\
  \logl{p}(\yb_i | x, \xb_i) & = \yb_i \ln(\Phi(\xb_i^{\T} x)) + (1-\yb_i)\ln(\Phi(-\xb_i^{\T} x)), \quad \likel{p}(\yb_i | x,\xb_i) = \exp(\logl{p}(\yb_i | x, \xb_i)),
 \end{align*}
where $\xb_i^{\T}$ is the $i$-th row of $\xb$ for $i \in \lbrace1,\ldots,N\rbrace$. We complete the Bayesian model by considering the Zellner $g$-prior for the regression parameter $x$, that is, \(\mathcal{N}_{d}(0,g(\xb^{\T}\xb)^{-1})\). Defining $\tilde{x}= (\xb^{\T} \xb)^{1/2} x$ and $\tilde{\xb}_i= (\xb^{\T} \xb)^{-1/2} \xb_i$, the scalar product is preserved, that is $\langle x, \xb_i \rangle = \langle \tilde{x}, \tilde{\xb}_i \rangle$ and, under the Zellner $g$-prior, $\tilde{x} \sim \mathcal{N}_{d}(0,g I_d)$. In the sequel, we apply the algorithms in the transformed parameter space with normalized covariates and put $g = 100$.
\par
The unnormalized posterior probability distributions $\pib{l}$ and $\pib{p}$ for the logistic and probit regression models are defined for all $\tilde{x} \in \rset^d$ by
\begin{align*}
 \pib{l}(\tilde{x} | \yb, \xb) &\propto \exp(-\Ub{l}(\tilde{x})) \quad \text{with} \quad \Ub{l}(\tilde{x}) = -\sum\nolimits_{i=1}^{N}\logl{l}(\yb_i | \tilde{x}, \xb_i) + (2\sigma^2)^{-1}\norm[2]{\tilde{x}} , \\
 \pib{p}(\tilde{x} | \yb, \xb) &\propto \exp (-\Ub{p}(\tilde{x})) \quad \text{with} \quad \Ub{p}(\tilde{x}) = -\sum\nolimits_{i=1}^{N}\logl{p}(\yb_i | \tilde{x}, \xb_i) + (2\sigma^2)^{-1}\norm[2]{\tilde{x}} .
 \end{align*}
It is straightforward to check that  $\Ub{l}, \Ub{p}$  satisfy  \ref{H1} and \ref{H2}.
\par
We analyze the performance of ESVM algorithm on two datasets from the UCI repository. The first dataset, Pima\footnote{\url{https://www.kaggle.com/uciml/pima-indians-diabetes-database}}, contains  $N = 768$ observations in dimension $d = 9$. {The second one, EEG\footnote{\url{https://archive.ics.uci.edu/ml/datasets/EEG+Eye+State}}, has dimension $d=15$, and for our experiments we take randomly selected subset of size $5000$ (to speed up sampling procedure).} We split each dataset into a training part $\mathcal{T}_N^{\text{train}}= [(y_i,\xb_i)]_{i=1}^N$ and
a test part $\mathcal{T}^{\text{test}}_K= [(y'_i,\xb'_i)]_{i=1}^{K}$ by randomly picking $K$ test points from the data. Then we use ULA, MALA, and RWM algorithms to sample from $\pib{l}(\tilde{x} | \yb, \xb)$ and $ \pib{p}(\tilde{x} | \yb, \xb)$ respectively. 
\par
Given the sample $(\tilde{x}_k)_{k=0}^{n-1}$, we aim at estimating the average likelihood over the test set $\mathcal{T}^{\text{test}}_K$, that is,
\[
\int\nolimits_{\mathbb{R}^d}f(\tilde{x})\pib{l}(\tilde{x} | \yb, \xb) \, \rmd \tilde{x}  \quad \biggl(\text{or } \int\nolimits_{\mathbb{R}^d}f(\tilde{x})\pib{p}(\tilde{x} | \yb, \xb) \, \rmd \tilde{x} \text{ for probit regression}\biggr),
\]
where the function $f$ is given by
\[
f(\tilde{x})= K^{-1} \sum_{i=1}^{K}\likel{l}(y'_i | \xb'_i,\tilde{x}) \quad \biggl(\text{or } K^{-1} \sum_{i=1}^{K}\likel{p}(y'_i | \xb'_i,\tilde{x}) \text{ for probit regression}\biggr).
\]
\par
VRFs are reported for first- and second-order control variates.  Results for logistic regression are given in \Cref{tab:Table_logistic}. Boxplots for the average test likelihood estimation using second-order control variates are shown in \Cref{fig:logistic}. The same quantities for probit regression are reported in \Cref{sec:Fig_Tables},  see \Cref{tab:Table_probit}, \Cref{fig:pima_probit}, and \Cref{fig:eeg_probit}. 
\par
Note that ESVM also outperforms 
EVM in this example. It is worth noting that for ULA and RWM, we show up to $100$ times better performance in terms of VRF. For MALA, the results for EVM and ESVM are similar since the samples are much less positively correlated.

\begin{figure}[!htb]\centering
\begin{minipage}{0.88\linewidth}
\captionsetup{font=small}
\caption{Estimation of the average test likelihood in logistic regression for the Pima dataset (top row) and the EEG dataset (bottom row).
In each {row} boxplots are given {for} ULA, MALA, and RWM respectively.} \label{fig:logistic}\vspace{-0.1cm}

  \includegraphics[width=0.328\linewidth]{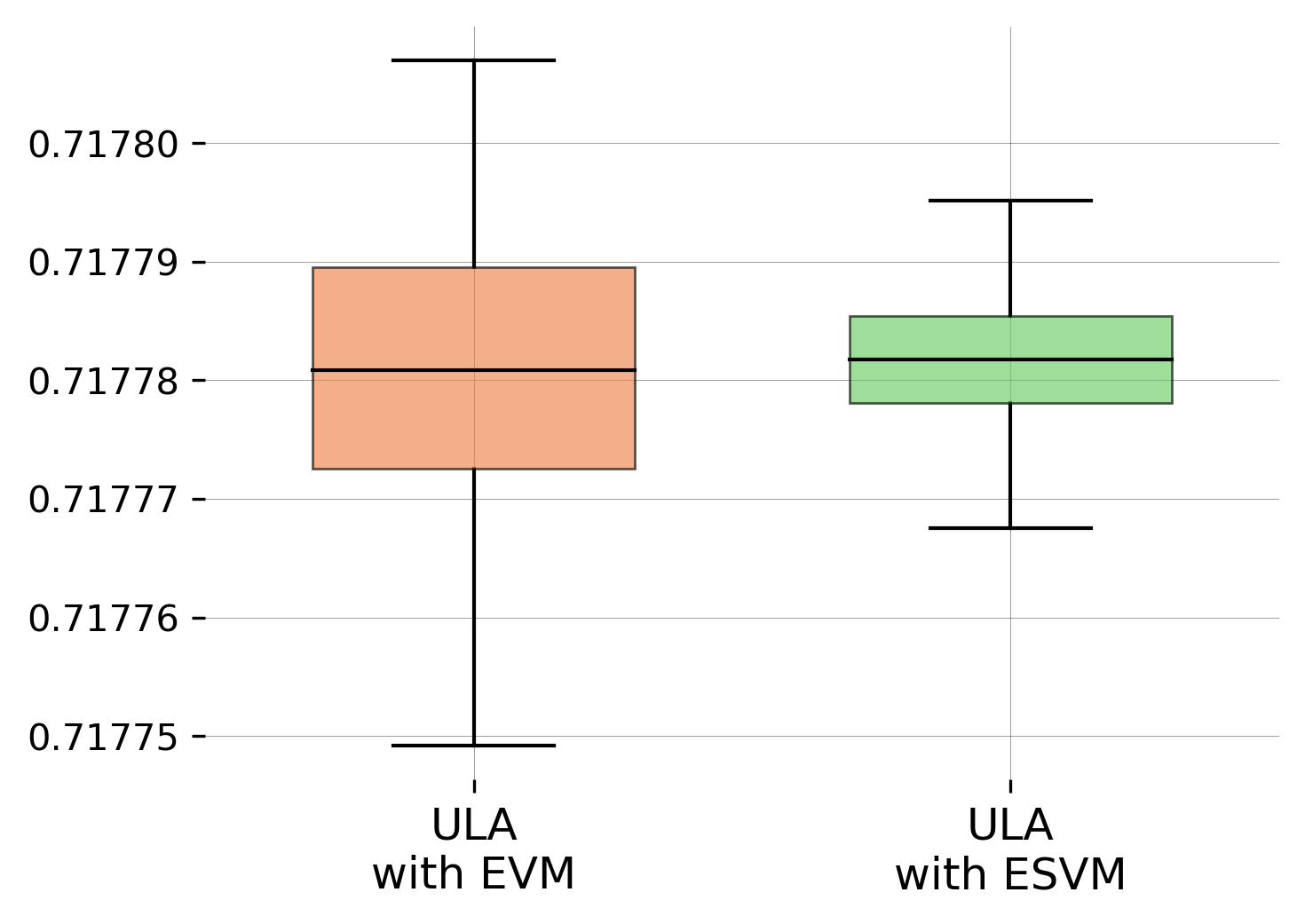}
  \includegraphics[width=0.328\linewidth]{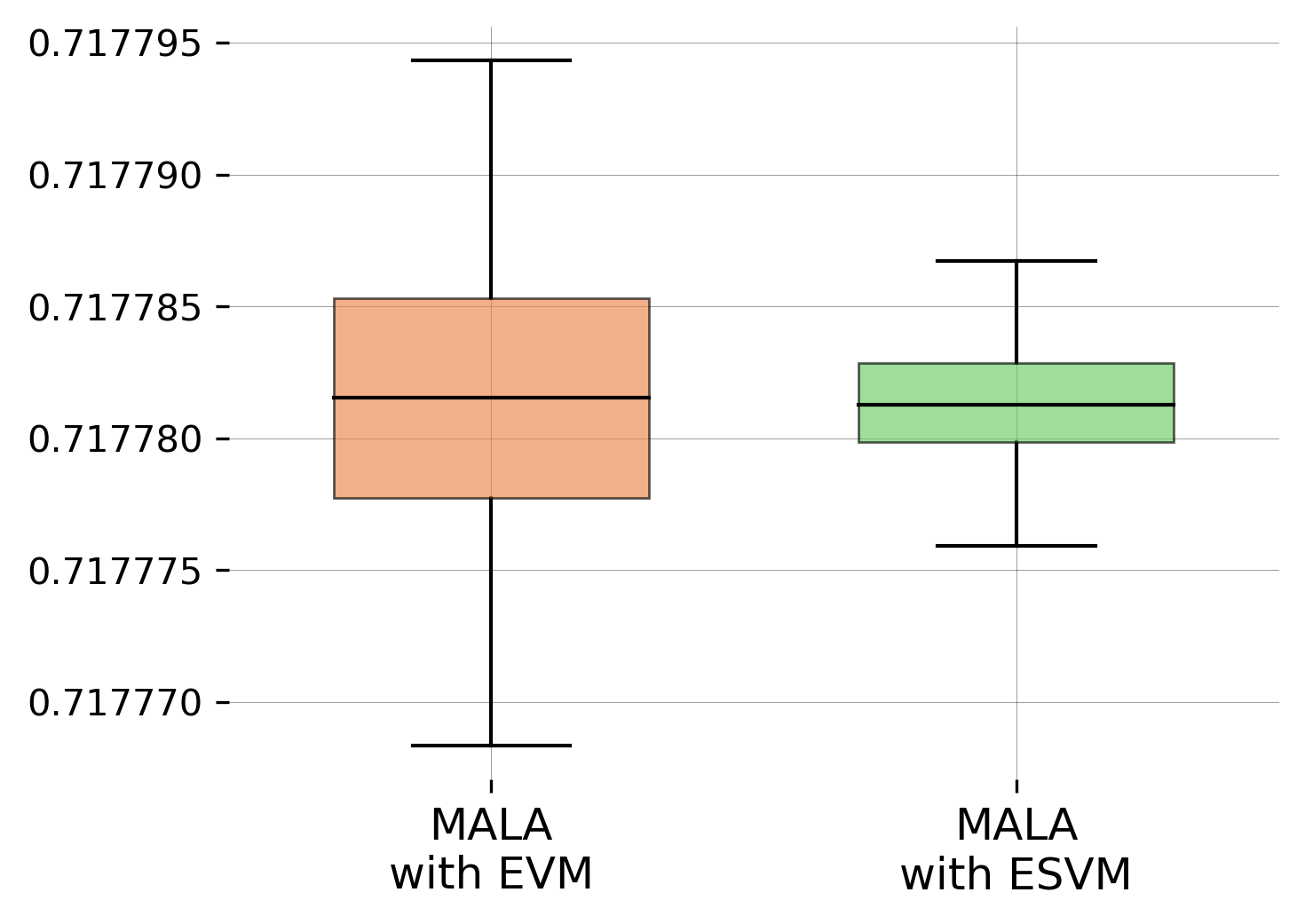}
  \includegraphics[width=0.328\linewidth]{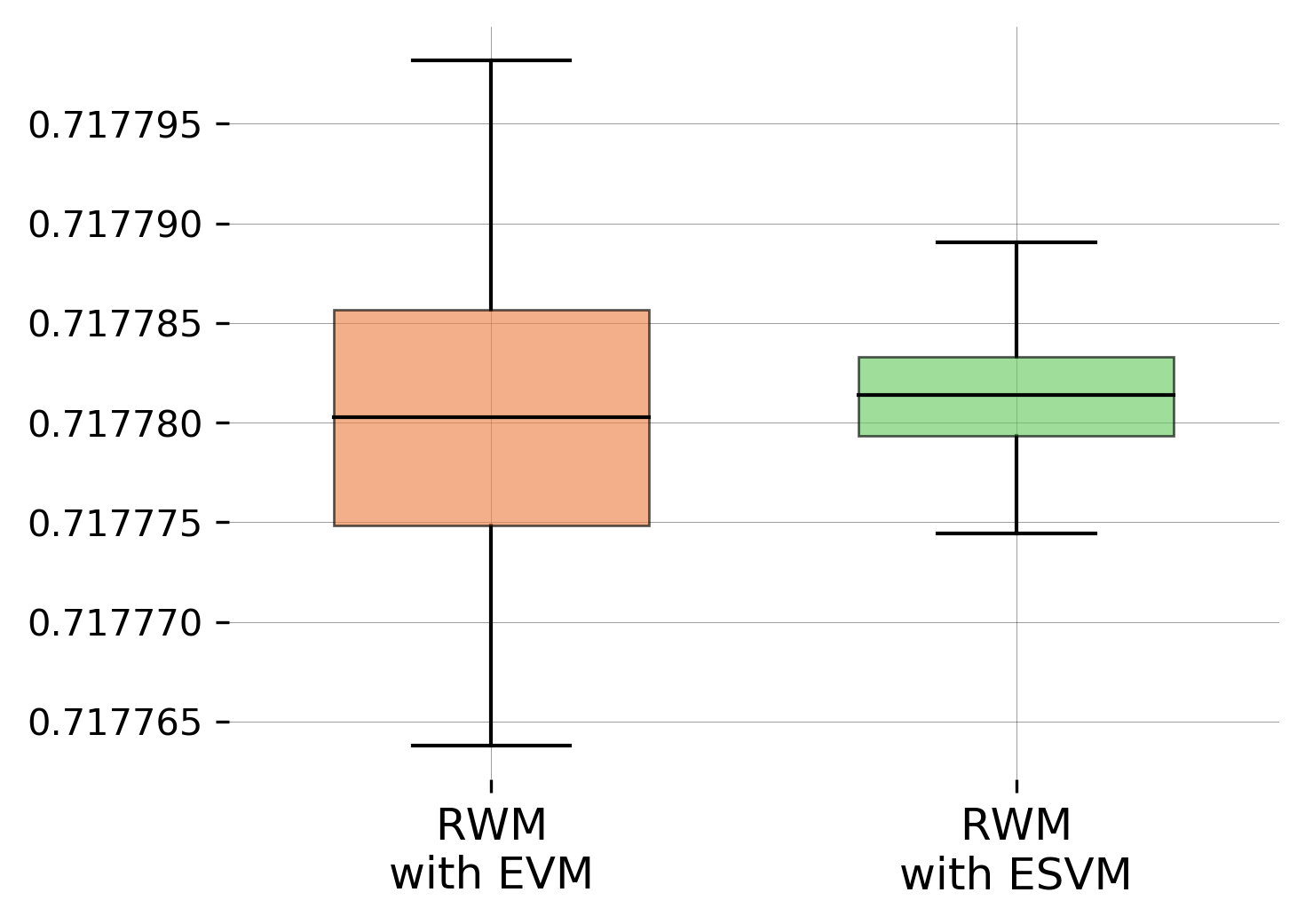}

  \includegraphics[width=0.328\linewidth]{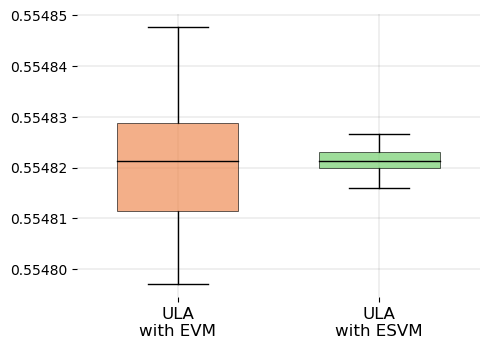}
  \includegraphics[width=0.328\linewidth]{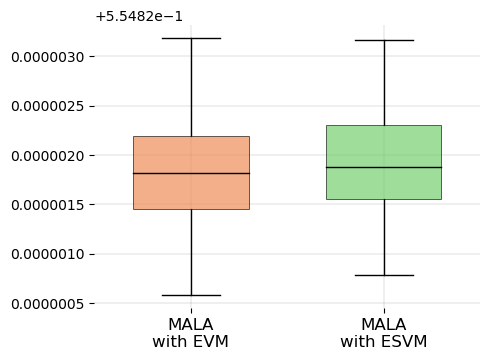}
  \includegraphics[width=0.328\linewidth]{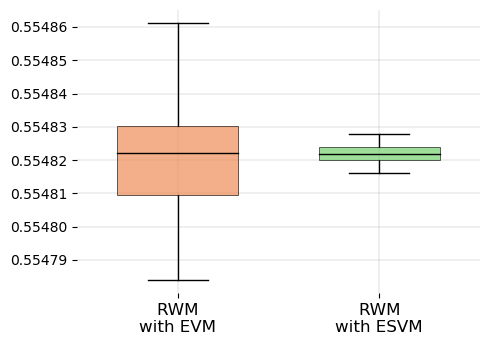}
\end{minipage}
\end{figure}

\begin{table}[!htb]\centering
\rat{1.1}
\begin{minipage}{0.88\linewidth}\centering
\captionsetup{font=small}
\caption{Average test likelihood estimation in logistic regression.}\label{tab:Table_logistic}\vspace{-0.1cm}
\resizebox{0.75\textwidth}{!}{
\begin{tabular}{@{}lcccccccc@{}}\toprule
 & \phantom{abc} &  \multicolumn{3}{c}{PIMA dataset} & \phantom{abc} & \multicolumn{3}{c}{EEG dataset} \\
\cmidrule{3-5} \cmidrule{7-9} 
Method & \phantom{abc} & ULA  & MALA & RWM  &  & ULA  & MALA & RWM  \\
\toprule
ESVM-1 &&   $347.6$ & $535.6$ & $411.7$ &&  $542.3$ & $996.6$ &  $483.5$ \\
EVM-1 &&  $347.9$ & $542.1$ & $415.5$  &&  $548.1$ &  $1020.2$ & $508.9$  \\ \midrule
ESVM-2  &&  $\bf{11387.3}$ & $\bf{28792.8}$ & $\bf{19503.3}$ &&  $\bf{11406.6}$  & $\bf{44612.5}$ & $\bf{11324.9}$ \\
EVM-2  && $2704.8$ & $4087.3$ & $5044.1$ && $350.3$    &    $39985.4$    & $453.3$ \\
\bottomrule
\end{tabular}}
\end{minipage}
\end{table}


\paragraph{Van der Pol oscillator equation\label{sec:VdP}}
The setup of this experiment is much similar to the one reported in \citet{south2018regularised}. Here a position $p_x(t) \in \rset$ evolves in time $t$ according to the second order differential equation
\begin{equation}
\label{eq:vdp}
\frac{\rmd^2 p_x}{\rmd t^2} - x(1-p_x^2)\frac{\rmd p_x}{ \rmd t} + p = 0,
\end{equation}
where $x \in \rset$ is an unknown parameter indicating the non-linearity and the strength of the damping. Letting $q_x = \rmd p_x/ \rmd t$ we can formulate the oscillator as the first-order system
\begin{equation*}
\label{eq:vdp_system}
\begin{cases}
\frac{\rmd p_x}{\rmd t} = q_x, \\
\frac{\rmd q_x}{\rmd t} = x(1-p_x^2)q_x + p_x,
\end{cases}
\end{equation*}
where only the first component $p_x$ is observed. This system was solved numerically using $x_\star = 1$ and starting point $p_{x_\star}(0) = 0$, $q_{x_\star}(0) = 2$. Observations $\yb_i = p_{x_\star}(t_i) + \eps_i$ were made at successive time instants $t_i = i$, $i = 1, \ldots, T$, and Gaussian measurement noise $\eps_i$ of standard deviation $\sigma =  0.5$ was added. We use a normal prior $\pi_0(x)$ with mean $\mu = 1$ and standard deviation $\sigma_0 = 0.5$. The unnormalized posterior probability distribution is defined for all $x > 0$ by
\begin{align*}
\pi(x | \yb) &\propto \exp(-U(x)) \quad \text{with} \quad U(x) =  -\log \pi_0(x) +  \sum_{i=1}^T \frac{(\yb_i - p_x(t_i))^2}{2 \sigma^2}.
\end{align*}
Clearly, $U$  satisfies  \ref{H1} and \ref{H3}. To sample from $\pi(x | \yb)$ we use the MALA algorithm. The quantity of interest is the posterior mean $\int\nolimits_{\mathbb{R}}x\pi(x | \yb)\,\rmd x$. In this example, we use control variates up to degree $3$. Results are presented in \Cref{sec:Fig_Tables} --- VRFs are summarized in \Cref{tab:Table_van_der_pole} and boxplots for the second-order control variates are given in \Cref{fig:van_der_pole}. 
In this problem, ESVM slightly outperforms EVM in terms of variance reduction factor. 

\paragraph{Lotka-Volterra system\label{sec:LV}}
The Lotka-Volterra model is a well-known system of ODEs describing the joint evolution of two interacting biological populations, predators and preys. Denote the population of preys and predators at moment \(t\) by \(u(t)\) and \(v(t)\) respectively, then the corresponding model can be written as the following first-order system
\begin{equation}
\label{eq:lv_system}
\begin{cases}
\frac{\rmd u}{\rmd t} = (\alpha - \beta v)u, \\
\frac{\rmd v}{\rmd t} = (-\gamma + \delta u)v, \\
u(0) = u_{0},\ v(0) = v_{0}.
\end{cases}
\end{equation}
The parameter vector is given by \(x = (\alpha,\beta,\gamma,\delta)\), with all components being non-negative due to the physical meaning of the problem. The system was solved numerically with the true parameters $x_\star = (0.6,0.025,0.8,0.025)$ and starting populations $u_{0} = 30.0$, $v_{0} = 4.0$. The system is observed at successive time moments $t_i = i$, $i = 1,\ldots, T$, with the lognormal measurements $\yb_i \sim \operatorname{Lognormal}(\log{u(t_i)},\sigma^2)$, $\xb_i \sim \operatorname{Lognormal}(\log{v(t_i)},\sigma^2)$ with $\sigma = 0.25$. A weakly informative normal prior $\pi_0(x)$ was used for the model parameters: $\mathcal{N}(1,0.5)$ for $\alpha$ and $\gamma$, $\mathcal{N}(0.05,0.05)$ for $\beta$ and $\delta$. The posterior distribution is given by $\pi(x | \yb,\xb) \propto \exp(-U(x))$, where
\begin{align*}
U(x) =  -\log \pi_0(x) +  \sum_{i=1}^T \biggl( \frac{(\log{\yb_i} - \log{u(t_i)})^2 + (\log{\xb_i} - \log{v(t_i)})^2}{2 \sigma^2} + \log{\yb_i} + \log{\xb_i} \biggr).
\end{align*}
We use the MALA algorithm to sample from $\pi(x | \yb, \xb)$. The quantity of interest is the posterior mean $\int\nolimits_{\mathbb{R}^4}x\pi(x | \yb,\xb)\,\rmd x$. 
VRFs are summarized in \Cref{tab:Table_lv} and boxplots for the second-order control variates are given in \Cref{fig:lv} and \Cref{sec:Fig_Tables}, \Cref{fig:lv_all}. For some model parameters ESVM significantly outperforms EVM in terms of VRF, for others the results are comparable with slight superiority of ESVM. 


\begin{table}[htb]\centering
\rat{1.1}
\begin{minipage}{0.88\linewidth}\centering
\captionsetup{font=small}
\caption{Estimation of the posterior mean in the Lotka-Volterra model.}\label{tab:Table_lv} 
\resizebox{0.6\textwidth}{!}{
\begin{tabular}{@{}lccccc@{}}\toprule
Estimated parameter  	&\phantom{abc} & $\alpha$  		& $\beta$ 		& $\delta$ 	& $\gamma$  \\ 
\toprule
ESVM-1 				&& $\bf{10.5}$ 		& $\bf{6.5}$ 	& $\bf{6.2}$	& $\bf{8.3}$	\\ \hline
EVM-1  				&& $6.6$ 			& $4.2$ 		& $4.9$		& $6.0$		\\ \hline
ESVM-2 				&&   $\bf{757.6}$  	& $\bf{427.8}$ 	& $\bf{277.2}$	& $\bf{446.6}$		\\ \hline
EVM-2  				&&   $642.1$    		& $286.0$  	& $275.0$		& $429.7$	\\ 
\bottomrule
\end{tabular}}
\end{minipage}
\end{table}

\begin{figure}[!htb]\centering
\begin{minipage}{0.88\linewidth}\centering
\captionsetup{font=small}
\caption{Estimation of the posterior mean of $\beta$ (left figure) and $\delta$ (right figure) in the Lotka-Volterra model.}\label{fig:lv}
  \includegraphics[width=0.328\linewidth]{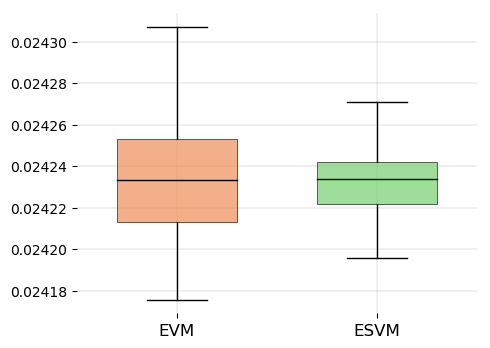}
  \includegraphics[width=0.328\linewidth]{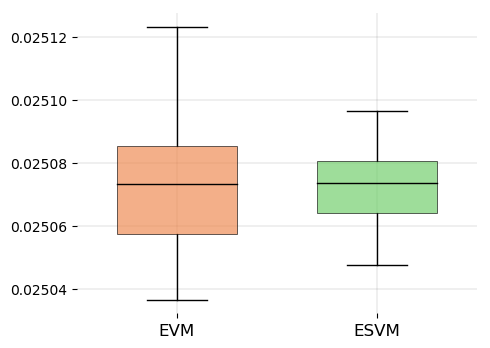}
\end{minipage}
\end{figure}

\section{Proofs}
\label{sec:proofs}
\subsection{Proof of \Cref{th:main_ex}}
Before we proceed to the proof of \Cref{th:main_ex},
let us refer to a general result from \citet{nickl2007bracketing}
which is used below to bound the fixed point
of a subset of a weighted Sobolev space. First we need to introduce some notations.
\par
Let $\mu$ be a (nonnegative) Borel measure.
Given the two
functions $l,u: \Xset \to \R$ in $\lp(\mu)$, the bracket $[l,u]$ is the set of all
functions in $\lp(\mu)$ with $l \leq f \leq u$. The $\lp(\mu)$-size of the bracket $[l,u]$ is defined as $\| l-u\|_{\lp(\mu)}$.
The $\lp(\mu)$-bracketing number $\mathcal{N}^{[\,]}_{\lp(\mu)}(F,\eps)$
of a (non-empty) set $F$ is the minimal number of brackets of
$\lp(\mu)$-size less than or equal to $\eps>0$ necessary to cover $F$.
The logarithm of the bracketing number is called the $\lp(\mu)$-bracketing metric entropy
$H^{[\,]}_{\lp(\mu)}(F, \eps)$.

\begin{theorem}[\protect{\cite[Corollary~4]{nickl2007bracketing}}]\label{nickl}
Let $1 < p <  \infty$, $\beta \in \R$, and $s - d/p > 0$. Let $F$ be a (non-empty) {norm-}bounded subset of
$W^{s,p}(\R^d, \langle x \rangle^{\beta})$. Suppose $M$ is a (non-empty) family of Borel measures on
$\R^d$ such that the condition $\sup_{\mu \in M}\| \langle x \rangle^{\alpha - \beta} \|_{L^r(\mu)} < \infty$
holds for some $1 \leq r \leq \infty$ and for some $\alpha > 0$.
Then
\[
	\sup_{\mu \in M} H^{[\,]}_{L^r(\mu)}(F, \eps) \lesssim
	\begin{cases}
		\eps^{-d/s}  &\text{for } \alpha>s-d/p,\\
		\eps^{-(\alpha/d + 1/p)^{-1}} &\text{for } \alpha < s - d/p.\\
	\end{cases}
\]
\end{theorem}

\begin{proof}[Proof of \Cref{th:main_ex}.]
We first bound the metric entropy of $\H$ by the bracketing metric entropy.
If $h \in \H$ is in the $2\eps$-bracket $[l, u]$, $l, u \in \H$, then it is in the ball of radius $\eps$ around $(l + u)/2$.
So,
\[
	H_{\ltwo(\pi)}(\H, \eps) \leq H^{[\,]}_{\ltwo(\pi)}(\H, 2\eps).
\]
Now our aim is apply Theorem \ref{nickl} to $\H$ which
is a norm-bounded subset of $W^{s,p}(\R^d, \langle x \rangle^{\beta})$ by assumption.
For $M = \{ \pi \}$ and $r =2$, the condition $\sup_{\mu \in M}\| \langle x \rangle^{\alpha - \beta} \|_{L^r(\mu)} < \infty$
also holds by assumption.
Hence,
	\[
	H_{\ltwo(\pi)}(\H,\eps) \lesssim
	\begin{cases}
		\eps^{-d/s}  &\text{for } \alpha>s-d/p,\\
		\eps^{-(\alpha/d + 1/p)^{-1}} &\text{for } \alpha < s - d/p.\\
	\end{cases}	
	\]	
Now we turn to the bound for the fixed point $\gamma_{\ltwo(\pi)}(\H, n)$ (see \eqref{eq:definition-fixed-point}).
Consider first the case $\alpha>s-d/p$.
The solution to the inequality $\eps^{-d/s} \lesssim n\eps^2$ is $\eps \gtrsim n^{-\frac{1}{2+d/s}} $.
Taking
$\eps_0 \sim n^{-\frac{1}{2+d/s}}$, where $\sim$ stands for equality up to a constant, yields
\[
	H_{\ltwo(\pi)}(\H,\eps_0) \lesssim n\eps_0^2,
	\quad\text{for } \alpha>s-d/p.
\]
Since $\gamma_{\ltwo(\pi)}(\H, n)$ is the infimum over all such $\eps > 0$, it holds
$\gamma_{\ltwo(\pi)}(\H, n) \lesssim n^{-\frac{1}{2+d/s}}$.
Repeated computations for $\alpha < s - d/p$ give us
$\gamma_{\ltwo(\pi)}(\H, n) \lesssim n^{-\frac{1}{2+(\alpha/d + 1/p)^{-1}}} $.
Combining these two bounds, we have
	\[
	\gamma_{\ltwo(\pi)}(\H, n) \lesssim
	\begin{cases}
		n^{-\frac{1}{2+d/s}}  &\text{for } \alpha>s-d/p,\\
		n^{-\frac{1}{2+(\alpha/d+ 1/p)^{-1}}} &\text{for } \alpha < s - d/p,\\
	\end{cases}	
	\]
which is the desired conclusion.
\end{proof}

\subsection{Spectral variance estimator}\label{sec:reprsv}
We investigate properties of the spectral variance $\SV{h}$ defined in \eqref{eq:sv}.
Note that  $\SV{h}$ can be represented as a quadratic form $Z_n(h)^{\T}A_nZ_n(h)$, where
$Z_n(h)=(h(X_0),\ldots,h(X_{n-1}))^{\T}$ and
$A_n$ is an $n \times n$ symmetric matrix.
Namely, let $\Id_n$ be the identity $n\times n$ matrix and $\bigone_n= (1,\dots,1)^{\T} \in \rset^n$.
Given the lag window $w_n$, we denote the weight matrix by $W_n=(w_n(j-i))_{i,j=1}^n$.
By rearranging the summations in \eqref{eq:sv}, we have
\begin{equation*}
   \SV{h} = n^{-1}\sum_{k=0}^{n-1}  \sum_{j=0}^{n-1}  w_n(k-j) \Bigl(h(X_k) - \pi_n(h)\Bigr) \Bigl(h(X_{j}) - \pi_n(h)\Bigr),
\end{equation*}
Hence the spectral variance can be represented as
\begin{align}\label{eq:qsv}
	\SV{h} = Z_n(h)^{\T}A_nZ_n(h)
	\quad\text{for} \
	A_n=\frac{1}{n}\,\left(\Id_n - \frac{1}{n}\bigone_n \bigone_n^{\T} \right)^{\T}W_n \left(\Id_n - \frac{1}{n}\bigone_n \bigone_n^{\T}\right).
\end{align}
In the following lemma we provide an upper bound on the operator norm of $A_n$.
\begin{lemma}\label{lem:opsv}
If the truncation point $b_n$ of the lag window $w_n$ satisfies
$b_n \leq n$, then $\left\| A_n \right\| \leq 2{b_n}/{n}$.
\end{lemma}
\begin{proof}
Denote $P= \Id_n - n^{-1} \bigone_n \bigone_n^\T$. Since $P$ is an orthonormal projector, we get
\[
	\left\| A_n \right\| = \frac1n \| PW_nP \| \leq \frac1n \|W_n\|.
\]
To bound the operator norm of $W_n$ (which is a Toeplitz matrix), we  use the standard technique based on
the discrete-time Fourier transform of the sequence $w:[-b_n,b_n]\to[0,1]$, defined, for $\lambda \in [-\pi,\pi)$ by
\[
\hat{w}_n(\lambda) =  \sum_{k=-b_n}^{b_n} w_n(k) \rme^{-\rmi k \lambda } \,.
\]
Obviously, $|\hat{w}_n(\lambda)|\leq2b_n$.
We have $\|W_n\| = \sup_{\|x\|=1} x^{\T}W_nx$.
Moreover, for any unit vector $u=(u_1,\ldots,u_n)^{\T}$ it holds
\begin{align*}
u^{\T}W_nu &=	
\sum_{k,j=1}^n \left( \frac{1}{2\pi}\int_{-\pi}^{\pi} \rme^{\rmi (k-j) \lambda} \hat{w}_n(\lambda) \rmd \lambda \right) u_ku_j = \frac{1}{2\pi}\int_{-\pi}^{\pi} \biggl|\sum_{k=1}^n \rme^{\rmi k \lambda}u_k\biggr|^2 \hat{w}_n(\lambda) \rmd \lambda \leq 2b_n .
\end{align*}
Hence $\|W_n\| \leq 2 b_n$ and $\| A _n\| \leq 2 b_n/n$.
The lemma is proved.
\end{proof}
In the next lemma we prove several technical results on expectation of the operator norm of $Z_n(h)$ and
$\SV{h}$ which hold under \ref{GE} assumption.
\begin{lemma}\label{lem:evar}
	Under \ref{GE}, it holds for any $h,h'\in\H$
	\begin{align*}
		\E_{x_0}\Bigl[\|Z_n(h)\|^2\Bigr] \leq n\|h\|_{\ltwo(\pi)}^2 + \frac{ \varsigma\driftfunc(x_0)}{1-\rho}\|h\|^2_{\driftfunc^{1/2}},
	\end{align*}
	and
	\begin{align*}
		\E_{x_0}\Bigl[\|Z_n(h) - Z_n(h') \|^2\Bigr]
		&\leq
		n\|h-h'\|_{\ltwo(\pi)}^2 + \frac{\varsigma\driftfunc(x_0)}{1-\rho} \|h-h'\|^2_{\driftfunc^{1/2}} .
	\end{align*}	
	Moreover, for any $h\in\H$, this bound implies
	\begin{align*}
		\PE[x_0]{\SV{h}}\leq
		2b_n{\|h\|_{\ltwo(\pi)}^2} + \frac{ 2\|h\|^2_{\driftfunc^{1/2}}\varsigma\driftfunc(x_0)}{1-\rho}\frac{b_n}{n}.
	\end{align*}	
\end{lemma}
\begin{proof}
We first observe that
\begin{align*}
	\E_{x_0}\Bigl[\|Z_n(h)\|^2\Bigr]
	= \PElr[x_0]{\sum\nolimits_{k=0}^{n-1} h^2(X_k)}
	= \sum\nolimits_{k=0}^{n-1} \|h\|_{\ltwo(P^{k}(x_0,\cdot))}^2.
\end{align*}
Now each summand can be bounded in the following way,
\begin{align*}
	\|h\|_{\ltwo(P^{k}(x_0,\cdot))}^2
	&=
	\|h\|_{\ltwo(\pi)}^2 + \Bigl(\|h\|_{\ltwo(P^{k}(x_0,\cdot))}^2  -  \|h\|_{\ltwo(\pi)}^2\Bigr)\\
	&\leq
	\|h\|_{\ltwo(\pi)}^2 + \int |h(x)|^2 |P^{k}(x_0,\cdot) - \pi |(\rmd x)\\
	&\leq
	\|h\|_{\ltwo(\pi)}^2 +   \VnormFunc[2]{h}{\driftfunc^{1/2}} \VnormFunc{P^k(x_0,\cdot) - \pi}{\driftfunc}.
\end{align*}
This inequality and  \ref{GE} together imply
\begin{align*}
	\E_{x_0}\Bigl[\|Z_n(h)\|^2\Bigr]
	&\leq
	n\|h\|_{\ltwo(\pi)}^2 + \frac{ \|h\|^2_{\driftfunc^{1/2}} \varsigma\driftfunc(x_0)}{1-\rho}\eqsp,
\end{align*}
which proves the first inequality.
Repeated computations for $Z_n(h) - Z_n(h')$ yield
\begin{align*}
	\E_{x_0}\Bigl[\|Z_n(h) - Z_n(h') \|^2\Bigr]
	&\leq
	n\|h-h'\|_{\ltwo(\pi)}^2 + \frac{ \varsigma\driftfunc(x_0)}{1-\rho} \|h-h'\|^2_{\driftfunc^{1/2}}.
\end{align*}
The first statement is proved. To prove the second statement we note that
\[
	\PE[x_0]{\SV{h}} = \E_{x_0}\Bigl[Z_n(h)^{\T}A_nZ_n(h)\Bigr]
	\leq \|A_n \| \E_{x_0}\Bigl[\|Z_n(h)\|^2\Bigr].
\]
By \Cref{lem:opsv} we have  $\|A_n\| \leq 2b_n/n$.
Substituting this  we deduce our claim.	
\end{proof}
\par
It is known that the spectral variance $\SV{h}$ is a biased estimate of the asymptotic variance
$\AV{h}$. In the following proposition we show how close is the expected value of $\SV{h}$ to $\AV{h}$.

\begin{proposition}
\label{prop:bias}
Assume \ref{GE}.  Then for any \(h\in \H\) and any $x_0 \in \Xset$,
\begin{equation*}
	\Bigl|\PE[x_0]{\SV{h}}-V_{\infty}(h)\Bigr|
	\leq
	\frac{\varsigma^{1/2}\pi(\driftfunc)\VnormFunc[2]{\th}{\driftfunc^{1/2}} }{1-\rho^{1/2}}
	\biggl( \frac{9\varsigma\driftfunc(x_0)}{(1-\rho)\pi(\driftfunc)}\frac{b_n}{n^2} +\frac{9b_n}{n} + 2\rho^{b_n/2}  \biggr)\eqsp,
\end{equation*}
where $\th= h - \pi(h)$.
Moreover, if $n \ge \varsigma\driftfunc(x_0)/((1-\rho)\pi(\driftfunc))$ then
\begin{equation*}
	\Bigl|\PE[x_0]{\SV{h}}-V_{\infty}(h)\Bigr|
	\leq
	\frac{20\varsigma^{1/2}\pi(\driftfunc) \VnormFunc[2]{\th}{\driftfunc^{1/2}}}{1-\rho^{1/2}}	
	\biggl( \frac{b_n}{n} \vee \rho^{b_n/2}  \biggr)\eqsp,
\end{equation*}
where $a \vee b \eqdef \max\{ a,b \}$.
\end{proposition}
\begin{proof}
Recall that the asymptotic variance $\AV{h}$ may be written as $\AV{h}=\sum_{|s|\ge0} \covcoeff[\pi]{h}{|s|}$ with $\covcoeff[\pi]{h}{s} = \PE[\pi]{\th(X_0)\th(X_{s})}$ and, by definition, $\SV{h} = \sum_{|s|<b_n} w_n(s) \ecovcoeff[n]{h}{|s|}\eqsp$, where the lag $s$ empirical autocovariance coefficient $\ecovcoeff[n]{h}{s}$ is given in \eqref{eq:empirical-autorcovariance}.
We have
\begin{align}
	\Bigl|\PE[x_0]{\SV{h}}-V_{\infty}(h)\Bigr|
	&\leq
	2\sum_{s=0}^{b_n-1} w_n(s) \bigl| \PE[x_0]{\ecovcoeff[n]{h}{s}} - \covcoeff[\pi]{h}{s} \bigr|\nonumber\\
	&\qquad\qquad\qquad
	+ 2\sum_{s=0}^{b_n-1} |1 - w_n(s)| | \covcoeff[\pi]{h}{s} | + 2\sum_{s=b_n}^{\infty} | \covcoeff[\pi]{h}{s} |\eqsp.
\label{eq:prop-bias-decomp}
\end{align}
To bound each summand in this decomposition,
we  need the following lemma.
\begin{lemma}
	Assume \ref{GE}.  Then for any $h\in\H$, $x\in\X$, and $s \in \zset_+$,
	\begin{equation}
	\label{eq:prop-bias-1}
		\Bigl|\PE[x]{ \th(X_0) \th(X_{s}) } \Bigr|
		\leq \varsigma^{1/2}\rho^{s/2} \driftfunc(x)\VnormFunc[2]{\th}{\driftfunc^{1/2}}  \eqsp,
	\end{equation}
	and
	\begin{equation}
	\label{eq:bound-stationary-covariance}
	\bigr| \covcoeff[\pi]{\th}{s}  \bigl| \leq \varsigma^{1/2}\rho^{s/2} \pi(\driftfunc)\VnormFunc[2]{\th}{\driftfunc^{1/2}}  \eqsp.
\end{equation}
\end{lemma}
\begin{proof}
The proof is straightforward.
Since $\pi(\th) = 0$, we have
\begin{align*}
	\Bigl|\PE[x]{ \th(X_0) \th(X_{s}) } \Bigr|
	&\leq \bigl|\th(x)\bigr| \biggl|\int_{\Xset}\th (y) \bigl(P^s(x,\cdot) - \pi\bigr)(\rmd y)\biggr|\\
	&\leq \VnormFunc[2]{\th}{\driftfunc^{1/2}} W^{1/2}(x) \int_{\Xset} W^{1/2}(y)\bigl|P^s(x,\cdot) - \pi\bigr|(\rmd y)\eqsp.
\end{align*}
By H\"older's inequality,
\begin{align*}
	\int_{\Xset} W^{1/2}(y)\bigl|P^s(x,\cdot) - \pi\bigr|(\rmd y)
	&\leq \bigl|P^s(x,\Xset) - \pi(\Xset)\bigr|^{1/2}
		\biggl(\int_{\Xset} W(y)\bigl|P^s(x,\cdot) - \pi\bigr|(\rmd y)\biggr)^{1/2}\\
	&\leq\VnormFunc[1/2]{P^s(x,\cdot) - \pi}{\driftfunc}\eqsp.
 \end{align*}
Combining these bounds and using \ref{GE}, we conclude
\begin{align*}
	\Bigl|\PE[x]{ \th(X_0) \th(X_{s}) } \Bigr|
	&\leq  \varsigma^{1/2}\rho^{s/2} W(x)  \VnormFunc[2]{\th}{\driftfunc^{1/2}}\eqsp,
\end{align*}
and \eqref{eq:prop-bias-1} is proved.
Integrating this relation with respect to the stationary distribution $\pi$, we obtain the second inequality.
The lemma is proved.
\end{proof}
Let us first bound the last two summands in the decomposition \eqref{eq:prop-bias-decomp}.
By definition, $w_n(s) = 1$ for all $s\in[-b_n/2,b_n/2]$.
From \eqref{eq:bound-stationary-covariance} we have the second summand
\begin{align}
	\sum_{s=0}^{b_n-1} |1 - w_n(s)| | \covcoeff[\pi]{h}{s} |
	&\leq \sum_{s=\lceil{b_n/2}\rceil}^{b_n-1}  | \covcoeff[\pi]{h}{s} |
	\leq  \varsigma^{1/2}\pi(\driftfunc) \VnormFunc[2]{\th}{\driftfunc^{1/2}}  \frac{\rho^{b_n/2}}{1-\rho^{1/2}}\eqsp.
\label{eq:prop-bias-s2}
\end{align}
where $\lceil{b_n/2}\rceil$ is the nearest integer greater than or equal to $b_n/2$.
Similar arguments apply to the last summand in \eqref{eq:prop-bias-decomp},
\begin{align}
	\sum_{s=b_n}^{\infty}| \covcoeff[\pi]{h}{s} |
	\leq \varsigma^{1/2}\pi(\driftfunc)\VnormFunc[2]{\th}{\driftfunc^{1/2}}\frac{\rho^{b_n}}{1-\rho^{1/2}}
	\leq \varsigma^{1/2}\pi(\driftfunc)\VnormFunc[2]{\th}{\driftfunc^{1/2}}\frac{\rho^{b_n/2}}{1-\rho^{1/2}}\eqsp.
\label{eq:prop-bias-s3}
\end{align}
It remains to bound the first summand in \eqref{eq:prop-bias-decomp}.
We note that lag $s$ empirical autocovariance coefficient satisfies $\ecovcoeff[n]{h}{s}=\ecovcoeff[n]{\th}{s}$.
Moreover,
for any $s< n$, it may be decomposed as $\ecovcoeff[n]{\th}{s}= \sum_{i=1}^3 A_{n,i}(s)$, where
\begin{align*}
	A_{n,1}(s) \eqdef \frac{1}{n} \sum_{k=0}^{n-s-1} \th(X_k) \th(X_{k+s}), \quad
	A_{n,2}(s) \eqdef  \frac{\pi_n(\th)}{n} \Biggl\{ \sum_{k=0}^{n-s-1} \th(X_k) + \sum_{k=s}^{n-1} \th(X_k) \Biggr\}\eqsp, \
\end{align*}
and $A_{n,3}(s) \eqdef (1-s/n)\pi^2_n(\th)$.
Since $|w_n(s)|\leq1$ by definition, it holds by the triangle inequality
\begin{align}
	\sum_{s=0}^{b_n-1} w_n(s) \bigl| \PE[x_0]{\ecovcoeff[n]{\th}{s}} - \covcoeff[\pi]{h}{s} \bigr|
	&\leq
	\sum_{s=0}^{b_n-1}\bigl| \PE[x_0]{A_{n,1}(s)} - \covcoeff[\pi]{h}{s} \bigr|\nonumber\\
	&\qquad+ \sum_{s=0}^{b_n-1}\bigl| \PE[x_0]{A_{n,2}(s)}\bigr|
	+ \sum_{s=0}^{b_n-1} \bigl| \PE[x_0]{A_{n,3}(s)}\bigr|\eqsp.
	\label{eq:prop-bias-4}
\end{align}
For any $s\in\{0,\ldots,n-1\}$, by the Markov property, \ref{GE}, and \eqref{eq:prop-bias-1} we obtain
\begin{align}
	&\Bigl|  \PE[x_0]{ \th(X_k) \th(X_{k+s}) } -  \covcoeff[\pi]{h}{s} \Bigr|
	= \bigg| \int \PE[x]{\th(X_0) \th(X_s)}(P^k(x_0,\cdot) - \pi) (\rmd x) \bigg| \nonumber\\
	&\qquad\qquad\leq  \varsigma^{1/2}\rho^{s/2}  \VnormFunc[2]{\th}{\driftfunc^{1/2}}
\VnormFunc{P^k(x_0,\cdot) - \pi}{\driftfunc}
	\leq \varsigma^{3/2}\rho^{s/2+k} \driftfunc(x_0) \VnormFunc[2]{\th}{\driftfunc^{1/2}}  \eqsp.
	\label{eq:prop-bias-2}
\end{align}
Therefore by \eqref{eq:bound-stationary-covariance} and \eqref{eq:prop-bias-2},
\begin{align*}
	 \sum_{s=0}^{b_n-1}\Bigl| \PE[x_0]{A_{n,1}(s)} - \covcoeff[\pi]{h}{s} \Bigr|
	&\leq n^{-1} \sum_{s=0}^{b_n-1}\sum_{k=0}^{n-s-1} \left| \PE[x_0]{ \th(X_k) \th(X_{k+s}) }  - \covcoeff[\pi]{h}{s} \right| + n^{-1} \sum_{s=0}^{b_n-1}  s|\covcoeff[\pi]{h}{s}| \\
	&\leq
	\frac{ \varsigma^{3/2} \driftfunc(x_0)\VnormFunc[2]{\th}{\driftfunc^{1/2}}}{n}
	\sum_{s=0}^{b_n-1}\sum_{k=0}^{n-s-1} \rho^{s/2+k}
	+
	\frac{ \varsigma^{1/2} \pi(\driftfunc)\VnormFunc[2]{\th}{\driftfunc^{1/2}}}{n}
	\sum_{s=0}^{b_n-1} s\rho^{s/2}	\\
	&\leq
	\frac{ \varsigma^{3/2} \driftfunc(x_0)\VnormFunc[2]{\th}{\driftfunc^{1/2}}}{n(1-\rho)(1-\rho^{1/2})}
	+
	\frac{ b_n\varsigma^{1/2} \pi(\driftfunc)\VnormFunc[2]{\th}{\driftfunc^{1/2}}}{n (1-\rho^{1/2})}  \eqsp.
\end{align*}
Note that  \eqref{eq:prop-bias-2} also yields
\begin{align}
	\PE[x_0]{{\pi}^2_n(\th)}
	&\leq 2n^{-2}  \sum_{k=0}^{n-1} \sum_{s=0}^{n-k-1} \bigl| \PE[x_0]{\th(X_k) \th(X_{k+s})} \bigr| \nonumber\\
	&\leq 2n^{-2}\sum_{k=0}^{n-1} \sum_{s=0}^{n-k-1}  \varsigma^{3/2}\rho^{s/2+k} \driftfunc(x_0) \VnormFunc[2]{\th}{\driftfunc^{1/2}}
	+  2n^{-2}\sum_{k=0}^{n-1} \sum_{s=0}^{n-k-1} \bigl|\covcoeff[\pi]{h}{s} \bigr|\nonumber\\
	&\leq
	\frac{2 \varsigma^{3/2}\driftfunc(x_0) \VnormFunc[2]{\th}{\driftfunc^{1/2}}}{n^2 (1-\rho)(1-\rho^{1/2})}
	+  \frac{2 \varsigma^{1/2}\pi(\driftfunc)\VnormFunc[2]{\th}{\driftfunc^{1/2}}}{n(1-\rho^{1/2})}\eqsp.
	\label{eq:prop-bias-3}
\end{align}
We now turn to $A_{n,2}(s)$.
By the Cauchy-Schwarz inequality and similar argument to \eqref{eq:prop-bias-3},
\begin{align*}
	\bigl| \PE[x_0]{A_{n,2}(s)} \bigr|
	&\leq \frac{\sqrt{2}}{n} \left\{ \PE[x_0]{\pi^2_n(\th)  }\right\}^{1/2}
	\left\{ \PElr[x_0]{\left( \sum\nolimits_{k=0}^{n-s-1} \th(X_k) \right)^2} + \PElr[x_0]{\left( \sum\nolimits_{k=s}^{n-1} \th(X_k) \right)^2}  \right\}^{1/2}\\
	&\leq  \frac{4\sqrt{2}\varsigma^{3/2}\driftfunc(x_0) \VnormFunc[2]{\th}{\driftfunc^{1/2}}}{n^2 (1-\rho)(1-\rho^{1/2})}
	+  \frac{4\sqrt{2} \varsigma^{1/2}\pi(\driftfunc)\VnormFunc[2]{\th}{\driftfunc^{1/2}} }{n(1-\rho^{1/2})}\eqsp.
\end{align*}
This gives
\begin{align*}
 \sum_{s=0}^{b_n-1}\bigl| \PE[x_0]{A_{n,2}(s)}\bigr| \leq  \frac{4\sqrt{2} b_n\varsigma^{3/2}\driftfunc(x_0) \VnormFunc[2]{\th}{\driftfunc^{1/2}} }{n^2 (1-\rho)(1-\rho^{1/2})}
	+  \frac{4\sqrt{2}b_n \varsigma^{1/2}\pi(\driftfunc)\VnormFunc[2]{\th}{\driftfunc^{1/2}} }{n(1-\rho^{1/2})}\eqsp.
\end{align*}
Finally, for $A_{n,3}(s)$ it follows from \eqref{eq:prop-bias-3} that
\begin{align*}
 	\sum_{s=0}^{b_n-1}\bigl| \PE[x_0]{A_{n,3}(s)}\bigr| \leq
 	\frac{2b_n \varsigma^{3/2}\driftfunc(x_0)\VnormFunc[2]{\th}{\driftfunc^{1/2}} }{n^2 (1-\rho)(1-\rho^{1/2})}
	+  \frac{2b_n \varsigma^{1/2}\pi(\driftfunc)\VnormFunc[2]{\th}{\driftfunc^{1/2}} }{n(1-\rho^{1/2})}\eqsp.
\end{align*}
Substituting these bounds into \eqref{eq:prop-bias-4} we obtain
\begin{align}\label{eq:prop-bias-s1}
	\sum_{s=0}^{b_n-1} w_n(s) \bigl| \PE[x_0]{\ecovcoeff[n]{\th}{s}} - \covcoeff[\pi]{h}{s} \bigr|
	\leq
 	\frac{9b_n \varsigma^{3/2}\driftfunc(x_0)\VnormFunc[2]{\th}{\driftfunc^{1/2}}}{n^2 (1-\rho)(1-\rho^{1/2})}
	+  \frac{9b_n \varsigma^{1/2}\pi(\driftfunc)\VnormFunc[2]{\th}{\driftfunc^{1/2}} }{n(1-\rho^{1/2})}\eqsp.
\end{align}
Collecting the estimates \eqref{eq:prop-bias-s2}, \eqref{eq:prop-bias-s3}, \eqref{eq:prop-bias-s1} and
substituting them into \eqref{eq:prop-bias-decomp} we conclude
\begin{align*}
	\Bigl|\PE[x_0]{\SV{h}}-V_{\infty}(h)\Bigr|
	&\leq
	\frac{\varsigma^{1/2}\pi(\driftfunc)\VnormFunc[2]{\th}{\driftfunc^{1/2}}}{1-\rho^{1/2}}
	\biggl( \frac{9\varsigma\driftfunc(x_0)}{(1-\rho)\pi(\driftfunc)}\frac{b_n}{n^2} +\frac{9b_n}{n} + 2\rho^{b_n/2}  \biggr)\eqsp,
\end{align*}
which is our claim.
If additionally $n \ge \varsigma\driftfunc(x_0)/((1-\rho)\pi(\driftfunc))$ then
\begin{align*}
	\Bigl|\PE[x_0]{\SV{h}}-V_{\infty}(h)\Bigr|
	&\leq
	\frac{20 \varsigma^{1/2}\pi(\driftfunc)\VnormFunc[2]{\th}{\driftfunc^{1/2}}}{(1-\rho^{1/2})}	
	\biggl( \frac{b_n}{n} \vee \rho^{b_n/2}  \biggr)\eqsp,
\end{align*}
and the proof is complete.
\end{proof}

\subsection{Proof of Theorem \ref{th:main_slow}}
For simplicity of notation, without loss of generality, we assume that
 functions $h\in\H$ are zero-mean, since, by definition, $\SV{h} = \SV{h - \pi(h)}$ and hence $h$ may be replaced by $\th = h - \pi(h)$
which also satisfies assumptions imposed on $h$.
Further, we write  \(\ESV{h}= \PE[x_0]{\SV{h}}\) and set
\begin{equation}
\label{eq:definition-H-M}
	H \eqdef \sup_{h \in \H} \|h\|_{\ltwo(\pi)}
	\quad\text{and}\quad
	M \eqdef \sup_{h \in \H} \|h\|_{\driftfunc^{1/2}}.
\end{equation}
Without loss of generality we may assume that $M < \infty$ since otherwise the statement of the theorem is obviously true.
\newline\newline\noindent
It follows from \Cref{prop:bias} that if $n \ge \varsigma\driftfunc(x_0)/((1-\rho)\pi(\driftfunc))$ then
\begin{equation*}
	\sup_{h\in\H}
	\bigl|\AV{h}-\ESV{h}\bigr| \lesssim G  \bigg ( \frac{b_n}{n} \vee \rho^{b_n/2} \bigg),
	\quad\text{where}\ \
	G \eqdef \frac{\varsigma^{1/2}M^2\pi(\driftfunc)}{1-\rho^{1/2}}.
\end{equation*}
Hence
\begin{equation}\label{eq:er_bound1}
	\AV{\hh_\eps} - \inf_{h\in\H} \AV{h}
	\leq
	\ESV{\hh_\eps} - \inf_{h\in\H} \ESV{h}
	+ 2 G  \bigg ( \frac{b_n}{n} \vee \rho^{b_n/2} \bigg).
\end{equation}
We are reduced to bounding the difference $\ESV{\hh_\eps} - \inf_{h\in\H} \ESV{h}$.
Let us denote by $h^*$  a function in $\H$ minimizing $\ESV{h}$, that is,
\begin{align}\label{eq:bvm}
	h^* \eqdef \argmin_{h \in \H} \ESV{h}.
\end{align}
We  assume that such a minimizer exists (a simple modification of the proof
is possible if $h^*$ is an approximate solution of \eqref{eq:bvm}).
Let also $\sh_\eps\in\mathcal{H}_\eps$ be the closest point to $\sh\in\H$ in $\ltwo(\pi)$.
By the definition of $\hh_\eps$, $\SV{\hh_\eps}- \SV{\sh_\eps}<0$.
We have
\begin{align}\label{th 2 all summands}
	&\ESV{\hh_\eps} - \ESV{\sh}
		\leq
	\ESV{\hh_\eps}- \ESV{\sh} -
	\bigl(\SV{\hh_\eps} - \SV{\sh_\eps}\bigr) \nonumber\\
		&\qquad\qquad=
	\ESV{\hh_\eps}- \ESV{\sh} - \bigl(\SV{\hh_\eps} - \SV{\sh} \bigr) +
	\bigl(\SV{\sh_\eps} - \SV{\sh}\bigr) \nonumber \\
		&\qquad\qquad\leq
	{\sup_{h\in\mathcal{H}_\eps} \Bigl\{\ESV{h}  -  \SV{h} \Bigr\}}
	+ {\bigl(\SV{\sh} -\ESV{\sh} \bigr)}
	+  {\bigl(\SV{\sh_\eps} - \SV{\sh}\bigr)}.
\end{align}
It remains to bound each summand in the right hand side of the decomposition~\eqref{th 2 all summands}.
To do this, we  need an exponential concentration for $\SV{h}$.
Let us remind that we consider two cases, Lipschitz and bounded functions $h\in\H$.
Depending on the case we consider, it follows from
\Cref{th:conc_var} (equation  \eqref{eq:conc_var_slow}) or \Cref{th:conc_var2} that,
for a fixed $\tau>0$, for all $t<\tau$, and all $h\in\H$,
\begin{align}\label{eq:conc_var_proof}
	\P_{x_0}\Bigl(\bigl|\SV{h}-\ESV{h}\bigr|>t\Bigr)
	\leq {2}
	\exp\biggl(-\frac{ t^2 n}{ c K_\tau^2 b_n^2} \biggr),
\end{align}
where $c>0$ is an absolute constant
and
\begin{align*}
	K_\tau^2 \eqdef \frac{\TC L^2}{(1-r)^2}\left( H^2+ \frac{ \varsigma M^2\driftfunc(x_0)}{1-\rho} + \frac{\tau}{b_n}\right)
	\quad \text{or} \quad
	K_\tau^2 \eqdef \beta^2 B^4
\end{align*}	
in the Lipschitz and bounded cases correspondingly.
Note that $K_\tau$ does not depend on $\tau$ in the bounded case.
The value of $\tau>0$ is specified later.
For the first summand in the decomposition~\eqref{th 2 all summands},
using the union bound and the concentration inequality \eqref{eq:conc_var_proof},
we obtain
\begin{align*}
	\P_{x_0} \left(\sup_{h\in\mathcal{H}_\eps} \Bigl\{\ESV{h}  -  \SV{h}\Bigr\}>t\right)
		&\leq
	|\mathcal{H}_\eps|	
	\sup_{h\in\mathcal{H}_\eps}
	\P_{x_0} \Bigl(\ESV{h}  -  \SV{h}>t\Bigr)\\
	&\leq 2|\mathcal{H}_\eps|	
	\sup_{h\in\mathcal{H}_\eps}
	\exp\biggl(-\frac{n t^2}{cK^2_{\tau}b_n^2} \biggr).
\end{align*}
For any $\eps \ge \gamma_{\ltwo(\pi)}(\mathcal{H},n) $
it holds \(|\mathcal{H}_\eps|\leq  \rme^{{n}\eps^2}\).
We can select $t = \sqrt{c}K_{\tau} b_n\bigl(\eps + n^{-1/2} \log^{1/2}(8/\delta)\bigr)$
to obtain
\begin{align}\label{th 2 summand 1}
	\P_{x_0} \left(\sup_{h\in\mathcal{H}_\eps} \Bigl\{\ESV{h}  -  \SV{h}\Bigr\}>t\right)
		\leq \delta/4 \eqsp.
\end{align}
In the same manner we can bound the second term in the right hand side
of the decomposition~\eqref{th 2 all summands}.
For \(t = \sqrt{c}K_{\tau} b_n n^{-1/2} \log^{1/2}(8/\delta) \),
it holds
\begin{align}\label{th 2 summand 2}
\P_{x_0} \left(\SV{\sh}-\ESV{\sh} >t\right) \leq \delta/4 \eqsp.
\end{align}
It remains to estimate the last summand in~\eqref{th 2 all summands}. This term is small since $\sh_\eps$ is $\eps$-close to $\sh$ in $\ltwo(\pi)$. We represent this summand in the following way
\begin{align*}
	\SV{\sh} - \SV{\sh_\eps}
	&=
	\SV{\sh} - \SV{\sh_\eps} - \Bigl[\ESV{\sh} - \ESV{\sh_\eps}\Bigr] + \Bigl[\ESV{\sh} - \ESV{\sh_\eps}\Bigr] .
\end{align*}
Now we have by the union bound and the concentration result \eqref{eq:conc_var_proof},
\begin{align}\label{th 2 summand 3_1}
	\P_{x_0}  \Bigl( \SV{\sh} - \SV{\sh_\eps} - \ESV{\sh} - \ESV{\sh_\eps} > t \Bigr) \leq \frac{\delta}{2}
\end{align}
for $t = \sqrt{c}K_{\tau} b_n n^{-1/2}\log^{1/2}(8/\delta)$.
Furthermore, let us represent $\SV{h}$ as a quadratic form $Z_n(h)^{\T}A_nZ_n(h)$
with $\|A_n\|\leq 2b_n/n$, see \Cref{sec:reprsv} for details.
It holds by the Cauchy-Schwarz inequality
\begin{align*}
	\ESV{\sh} - \ESV{\sh_\eps}
		&=
	\PE[x_0]{Z_n(h^{*})^{\top}A_nZ_n(h^{*}) - Z_n(h_\eps^{*})^{\top}A_nZ_n(h_\eps^{*})}\\
		&=
	\PE[x_0]{Z_n(h^{*})^{\top}A_n\bigl(Z_n(h^{*})-Z_n( h_\eps^{*})\bigr) + \bigl(Z_n( h^{*})-Z_n(h_\eps^{*})\bigr)^{\top}A_nZ_n(h_\eps^{*})}\\
	&\leq
		\| A_n \| \, \left(\E_{x_0}\bigl[\|Z_n(h^{*})-Z_n(h_\eps^{*})\|^2\bigr] \right)^{1/2}
		\cdot2\sup_{h\in\H} \left( \E_{x_0}\|Z_n(h)\|^2\right)^{1/2}.
\end{align*}
Let $R^2 \eqdef \varsigma M^2  \driftfunc(x_0)(1-\rho)^{-1}$. Then
\Cref{lem:evar} yields
\begin{align}\label{th 2 summand 3_2}
	\ESV{\sh} - \ESV{\sh_\eps}
	\leq
	{4b_n} \left(\eps + \frac{\sqrt{2}R}{\sqrt{n}} \right)\left(H + \frac{R}{\sqrt{n}}\right).
\end{align}
Combining the bounds \eqref{th 2 summand 1}, \eqref{th 2 summand 2}, \eqref{th 2 summand 3_1}, and \eqref{th 2 summand 3_2} for all summands and substituting them into \eqref{th 2 all summands}, we can assert that
for $\eps \ge \gamma_{\ltwo(\pi)}(\mathcal{H},n),$
with probability at least $1-\delta$,
\begin{align*}
	\ESV{\hh_\eps} - \ESV{\sh}
	&\lesssim
	\left(K_{\tau}+H+ \frac{R}{\sqrt{n}}\right) b_n\eps
	+ \frac{b_n R}{\sqrt{n}}\left(H+ \frac{R}{\sqrt{n}}\right)
	+K_{\tau} \frac{b_n{\log^{1/2}(\frac{8}{\delta})}}{\sqrt{n}},
\end{align*}
where \(\lesssim\) stands for inequality up to an absolute constant.
Now we can set $\tau$ to be an upper bound for the chosen
$t$, namely, $\tau = \sqrt{c}K_{\tau} b_n\bigl(\eps + n^{-1/2} \log^{1/2}(8/\delta)\bigr)$.
In the bounded case, $K_{\tau}$ does not depend on $\tau$, but in the Lipschitz case
this choice leads to a quadratic equation
\[
	K_{\tau}^2 = \frac{\TC L^2}{(1-r)^2} \biggl(H^2 +\frac{R^2}{n} +K_{\tau} \sqrt{c}\biggl(\eps +\frac{\log^{1/2}(\fraca{8}{\delta})}{\sqrt{n}}\biggr)\biggr),
\]
For a large $c>0$, this quadratic equation always has a solution which may be written as
$K_{\tau}  \lesssim   \frac{\sqrt{\TC}L}{1-r} \bigl(H+Rn^{-1/2} +\bigl(\eps + n^{-1/2}\log^{1/2}(\fraca{8}{\delta})\bigr)\bigr)$.
Let $n \ge n_0$, where $n_0$ satisfies
\[
	n_0 \ge \frac{\varsigma\driftfunc(x_0)}{(1-\rho)\pi(\driftfunc)},
	\quad
	n_0 \ge \frac{\max\{R^2,\log(\fraca{8}{\delta})\}}{H^2},
	\quad\text{and}\quad
	\gamma_{\ltwo(\pi)}(\mathcal{H},n_0) \leq H.
\]
Then $K_{\tau}  \lesssim  \sqrt{\TC}HL/(1-r)$ (in the Lipschitz case) and $H+Rn^{-1/2} \lesssim H$.
We set $\eps = \gamma_{\ltwo(\pi)}(\mathcal{H},n)$ and obtain
\begin{align*}
	\ESV{\hh_\eps} - \ESV{\sh}
	&\lesssim
	(K_{\tau}  + H) b_n\gamma_{\ltwo(\pi)}(\mathcal{H},n)
	+ (K_{\tau} +HR)\frac{b_n\log(\fraca{1}{\delta})}{\sqrt{n}}.
\end{align*}
Substituting this into \eqref{eq:er_bound1} and taking \(b_n=2(\log(\fraca{1}{\rho}))^{-1}\log(n)\), we conslude
\begin{align*}
	\AV{\hh_\eps} - \inf_{h\in\H} \AV{h}
	&\lesssim
	 (\log(\fraca{1}{\rho}))^{-1}(K_{\tau}  + H) \log(n) \gamma_{\ltwo(\pi)}(\mathcal{H},n)\\
	&\quad
	+ (\log(\fraca{1}{\rho}))^{-1} \left(K_{\tau} +\frac{\varsigma^{1/2} MH \driftfunc(x_0)}{(1-\rho)^{1/2}}+\frac{\varsigma^{1/2} M^2 \pi(\driftfunc) }{\sqrt{n}(1-\rho^{1/2})}
	\right)\frac{\log(n)\log(\fraca{1}{\delta})}{\sqrt{n}},
\end{align*}
Note that $H \lesssim L$ or $H \lesssim B$ in the Lipschitz and bounded cases correspondingly,
and $H \lesssim H^2 \lesssim K_\tau$ in both cases. Taking $K^2=K_{\tau}$ and simplifying last expression, we get the desired conclusion.

\subsection{Proof of Theorem \ref{th:main_fast}}
As above, we assume that functions $h\in\H$ are zero-mean and set \(\ESV{h}= \PE[x_0]{\SV{h}}\).
It follows from \Cref{prop:bias} that if $n \ge \varsigma\driftfunc(x_0)/((1-\rho)\pi(\driftfunc))$ then
\begin{equation*}
	\sup_{h\in\H}
	\bigl|\AV{h}-\ESV{h}\bigr| \lesssim G  \bigg ( \frac{b_n}{n} \vee \rho^{b_n/2} \bigg),
	\quad\text{where}\ \
	G \eqdef \frac{\varsigma^{1/2}M^2\pi(\driftfunc)}{1-\rho^{1/2}} \,,
\end{equation*}
where $M$ is defined in \eqref{eq:definition-H-M}. Hence
\begin{equation}\label{eq:er_bound2}
	\AV{\hh_\eps}
	\leq
	\ESV{\hh_\eps}
	+ G  \bigg ( \frac{b_n}{n} \vee \rho^{b_n/2} \bigg).
\end{equation}
We are reduced to bounding $\ESV{\hh_\eps}$.
Let us denote by $h^*$  a constant function in $\H$ exising by assumption.
Let also $\sh_\eps\in\mathcal{H}_\eps$ be the closest point to $\sh$ in $\mathcal{H}_\eps$ in $\ltwo(\pi)$.
By the definition of $\hh_\eps$, $\SV{\hh_\eps}-\SV{\sh_\eps}<0$.
We have for any $c>0$,
\begin{align}
\label{th1 all summands}
	\ESV{\hh_\eps}
		&\leq
	\ESV{\hh_\eps} - (1+c) \bigl(\SV{\hh_\eps}-\SV{\sh_\eps}\bigr)
		=
	\ESV{\hh_\eps} - (1+c)\SV{\hh_\eps} + (1+c)\SV{\sh_\eps}\nonumber\\
		&\leq
	{\sup_{h\in\mathcal{H}_\eps} \Bigl\{\ESV{h} - (1+c)\SV{h} \Bigr\}}
	+  {(1+c)\SV{\sh_\eps}}.
\end{align}
We take $c=1$ and bound the two summands in the right hand side of~\eqref{th1 all summands} separately.
To do this, we  need an exponential concentration for $\SV{h}$.
It follows from \Cref{th:conc_var} (equation  \eqref{eq:conc_var_fast}) that,
for all $t>0$ and for all $h\in\H$,
\begin{align}\label{eq:conc_var_proof2}
	\P_{x_0}\Bigl(\bigl|\SV{h} - \ESV{h} \bigr|>t\Bigr)
	\leq {2}
	\exp\Biggl(-\frac{n t^2}{ c K^2\, b_n\bigl(\ESV{h} + t\bigr)} \Biggr),
\end{align}
where $c>0$ is some universal constant, $K^2=\TC L^2/(1-r)^2$, and
$b_n$ is the size of the lag window.
For the first summand in the right hand side of the decomposition \eqref{th1 all summands},
using the union bound and the concentration inequality \eqref{eq:conc_var_proof2}, we obtain
\begin{align*}
	&\P_{x_0}\biggl(\sup_{h\in\mathcal{H}_\eps}
	\Bigl\{\ESV{h} - 2\SV{h}\Bigr\}>t\biggr)
	\leq
	|\mathcal{H}_\eps| \sup_{h\in\mathcal{H}_\eps }
	\P_{x_0}\biggl(\ESV{h} - 2\SV{h}>t\biggr)\\
	&\qquad\qquad\qquad\leq
	2|\mathcal{H}_\eps| \sup_{h\in\mathcal{H}_\eps }
	\exp\biggl(-\frac{ n \left(t+\SV{h}\right)}{cK^2b_n} \biggr)
	\le 2|\mathcal{H}_\eps|
	\exp\biggl(-\frac{ nt}{cK^2b_n} \biggr),
\end{align*}
where the last inequality holds since $\SV{h}\ge0$.
For any $\eps \ge \gamma_{\ltwo(\pi)}(\mathcal{H},n) $
it holds \(|\mathcal{H}_\eps|\leq e^{n\eps^2}\).
Hence we can select $t = cK^2 b_n\left(\eps^2 + n^{-1} \log(4/\delta)\right)$
to obtain
\begin{equation}
\label{th 1 summand 1}
	\P_{x_0}\biggl(\sup_{h\in\mathcal{H}_\eps}
	\Bigl\{\ESV{h} - 2\SV{h}\Bigr\}>t\biggr)
	\leq  \delta/2.
\end{equation}
The second term in~\eqref{th1 all summands} is small since $\sh_\eps$ is $\eps$-close to $\sh$ in $\ltwo(\pi)$.
First we note that
\begin{align*}
	\SV{\sh_\eps}=\SV{\sh_\eps}-2\ESV{\sh_\eps}+\ESV{\sh_\eps}.
\end{align*}
By the union bound and the concentration inequality \eqref{eq:conc_var_proof2}, we have
\begin{align}
\label{th 1 summand 2}
	\P_{x_0} \Bigl(\SV{\sh_\eps} - 2 \ESV{\sh_\eps} > t \Bigr)
	\leq
	2\exp\biggl(-\frac{n\left(t+\SV{\sh_\eps}\right)}{cK^2b_n} \biggr)
	\leq
	2\exp\biggl(-\frac{nt}{cK^2b_n} \biggr).
\end{align}
Hence for $t = cK^2 b_n n^{-1} \log(4/\delta)$ this probability is bounded by $\delta/2$. Furthermore, let us represent $\SV{h}$ as a quadratic form $Z_n(h)^{\T}A_nZ_n(h)$
(see \Cref{sec:reprsv} for details). By assumption, $\sh$ is a constant function, and hence
$A_nZ_n(\sh)$ is the zero vector. Since $\|A_n\|\leq 2b_n/n$ (see  \Cref{lem:opsv}), it holds
\begin{align}
	\ESV{\sh_\eps}
	&=
	\PE[x_0]{Z_n(\sh_\eps)^\top A_n Z_n(\sh_\eps)}
	=
	\PE[x_0]{(Z_n(\sh_\eps)-Z_n(\sh))^{\top}A_n(Z_n(\sh_\eps)-Z_n(\sh))} \nonumber\\
	&\leq \frac{2b_n}{n}  \PE[x_0]{\|Z_n(\sh_\eps)-Z_n(\sh)\|^2}.
	\label{eq:bound-PE_V_n}
\end{align}
Let $R^2 \eqdef \varsigma M^2 \driftfunc(x_0)(1-\rho)^{-1}$. Then \Cref{lem:evar} yields
\begin{align}
\label{th 1 summand 3}
	\ESV{\sh_\eps}
	\leq
	2b_n\eps^2+ 8R^2\frac{b_n}{n} .
\end{align}
Combining the bounds~\eqref{th 1 summand 1},~\eqref{th 1 summand 2} and~\eqref{th 1 summand 3}  for all summands and substituting them into \eqref{th1 all summands}, we can assert that
for $\eps \ge \gamma_{\ltwo(\pi)}(\mathcal{H},n)$,
with probability at least $1-\delta$, we have
\begin{align*}
	\ESV{\hh_\eps}
	\lesssim
	K^2 b_n\eps^2+(K^2+R^2)\frac{b_n\log(\frac{4}{\delta})}{n}.
\end{align*}
Substituting this bound into \eqref{eq:er_bound2} with $\eps = \gamma_{\ltwo(\pi)}(\mathcal{H},n)$
and $b_n = 2(\log(\fraca{1}{\rho}))^{-1}\log(n)$ yields
\begin{align*}
	\AV{\hh_\eps}
	&\lesssim
	\frac{K^2}{\log(\fraca{1}{\rho})}
	 \log(n) \gamma^2_{\ltwo(\pi)}(\mathcal{H},n)
	+\frac{K^2+R^2+G}{\log(\fraca{1}{\rho})} \cdot \frac{\log(n)\log(\frac{1}{\delta})}{n}\\
	&\lesssim
	\frac{\TC L^2}{(1-r)^2\log(\fraca{1}{\rho})} \log(n) \gamma^2_{\ltwo(\pi)}(\mathcal{H},n)\\
	&\qquad\qquad\qquad\qquad\qquad
	+\left(\frac{\TC L^2}{(1-r)^2\log(\fraca{1}{\rho})}+\frac{\varsigma M^2(\pi(\driftfunc) + \driftfunc(x_0))}{(1-\rho)^{1/2}\log(\fraca{1}{\rho})}
	\right)\frac{\log(n)\log(\frac{1}{\delta})}{n},
\end{align*}
which is the desired conclusion.

\section{Tables and Figures}
\label{sec:Fig_Tables}
\begin{table}[!htp]\centering
\rat{1.1}
\begin{minipage}{0.88\linewidth}\centering
\captionsetup{font=small}
\caption{Experimental setup details.}\label{tab:Table_setup}
\resizebox{0.91\textwidth}{!}{\centering
\begin{tabular}{@{}l c  c c c c c c@{}}
\toprule
Experiment& $n_{\text{burn}}$ & $n_{\text{train}}$ &  $n_{\text{test}}$ & $\gamma_{\text{ULA}}$ & $\gamma_{\text{MALA}}$ & $\gamma_{\text{RWM}}$ & $b_n$\\ \toprule 
GMM, $\E_\pi[X_2]$,   $\Sigma = \Id$&   $10^4$&$10^5$&$10^5$&$0.1$&$1.0$&$0.5$&$50$\\   
GMM, $\E_\pi[X_2]$,   $\Sigma = \Sigma_0$    &   $10^4$&$10^5$&$10^5$&$0.1$&$0.2$&$0.1$&$50$    \\   
GMM, $\E_\pi[X_2^2]$, $\Sigma = \Id$ & $10^4$&$10^5$&$10^5$&$0.1$&$1.0$&$0.5$&$50$  \\   
GMM, $\E_\pi[X_2^2]$, $\Sigma = \Sigma_0$ & $10^4$&$10^5$&$10^5$&$0.1$&$0.1$&$0.1$&$50$    \\   
Banana-shape, $d = 2$&$10^5$&$10^6$&$10^6$&$0.01$&$0.5$&$0.5$&$300$    \\ 
Banana-shape, $d = 8$&$10^5$&$10^6$&$10^6$&$0.01$&$0.2$&$0.1$&$300$   \\ 
Logistic and probit regression, Pima&$10^3$&$10^4$&$10^4$&$0.1$&$0.5$&$0.5$&   $10$   \\ 
Logistic regression, EEG\hspace{1pt}&	$10^3$&$10^4$&$10^4$ &$0.1$&$1.0$&$0.1$&$10$\\ 
Probit regression, EEG&$10^3$&$10^4$&$10^4$ &$0.1$&$0.5$&$0.1$&$10$\\ 
Van der Pol oscillator&$10^2$&$10^3$&$10^3$ & $-$&$10^{-3}$&$-$&$10$\\
Lotka-Volterra model&$10^3$&$10^4$&$10^4$ & $-$&$\num{5e-6}$&$-$&$10$\\
\bottomrule
\end{tabular}}
\end{minipage}
\end{table}

\begin{table}\centering
\rat{1.1}
\begin{minipage}{0.88\linewidth}\centering
\captionsetup{font=small}
\caption{Variance Reduction Factors in probit regression, average test likelihood.}\label{tab:Table_probit}
\resizebox{0.91\textwidth}{!}{
\begin{tabular}{@{}lcccccccc@{}}\toprule
 & \phantom{abc} &  \multicolumn{3}{c}{PIMA dataset} & \phantom{abc} & \multicolumn{3}{c}{EEG dataset}  \\
\cmidrule{3-5} \cmidrule{7-9}
Method & \phantom{abc} & ULA  & MALA & RWM  &  & ULA  & MALA & RWM  \\
\toprule
ESVM-1 &&   $263.2$ & $419.7$ & $251.4$ &&  $1317.0$ & $1515.0$ &  $938.5$\\
EVM-1 &&  $270.1$ & $430.1$ & $261.6$  &&  $1331.6$ &  $1572.7$ & $948.1$ \\ \midrule
ESVM-2  &&  $\bf{26835.7}$ & $\bf{55373.7}$ & $\bf{28905.0}$ &&  $\bf{45059.2}$  & $45964.5$ & $\bf{34957.1}$ \\
EVM-2  && $6660.7$ & $29710.4$ & $14187.1$ && $29620.4$    &    $\bf{71095.6}$    & $6340.1$ \\
\bottomrule
\end{tabular}}
\end{minipage}
\end{table}

\begin{table}[!htb]\centering
\rat{1.1}
\begin{minipage}{0.88\linewidth}\centering
\captionsetup{font=small}
\caption{Variance Reduction Factors for Van der Pol oscillator, posterior mean estimation.}\label{tab:Table_van_der_pole}
\resizebox{0.62\textwidth}{!}{
\begin{tabular}{@{}lcccc@{}}\toprule
Method 	&\phantom{abc} & $1$st order CV  & $2$nd order CV   & $3$rd order CV  \\ 
\toprule
ESVM 	&&  $30.7$         & $\bf{49.1}$		& $\bf{243.2}$ \\ \hline
EVM   	&& $\bf{33.9}$      & $44.1$   		& $183.7$\\ 
\bottomrule
\end{tabular}}
\end{minipage}
\end{table}

\begin{figure}[!htb]
\captionsetup{font=small}
\caption{Estimation of $\E_\pi[X^2_2]$ in GMM with \(\Sigma = \Id\). Left figure: boxplot for ULA estimates compared to the corresponding boxplots for EVM and ESVM estimates. Next three figures: boxplots for EVM and ESVM estimates for ULA, MALA, and RWM with second-order control variates being used.}
\label{fig:gmm_white_2nd}
\minipage{0.24\textwidth}
  \includegraphics[width=\linewidth]{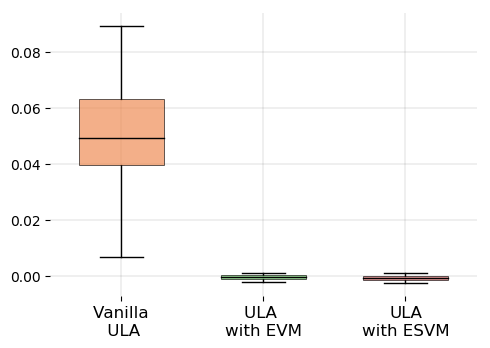}
\endminipage\hfill
\minipage{0.24\textwidth}
  \includegraphics[width=\linewidth]{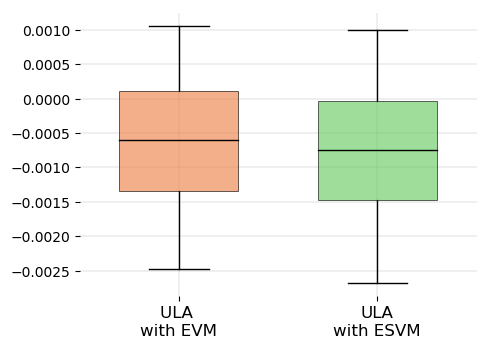}
\endminipage\hfill
\minipage{0.24\textwidth}%
  \includegraphics[width=\linewidth]{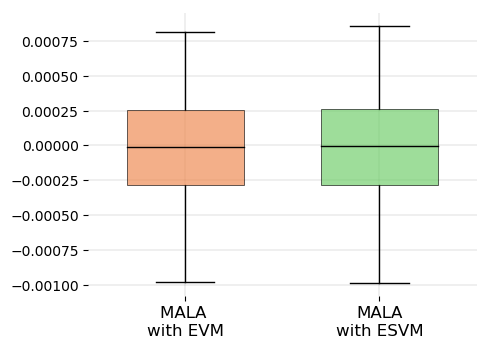}
\endminipage
\minipage{0.24\textwidth}%
  \includegraphics[width=\linewidth]{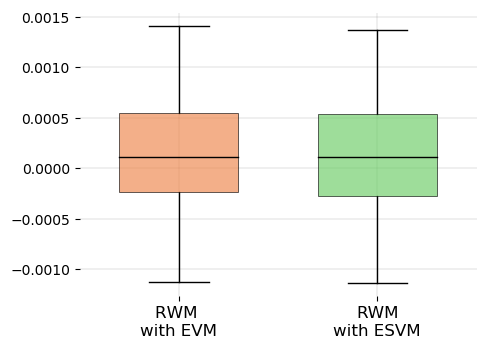}
\endminipage
\end{figure}

\begin{figure}[!htb]
\captionsetup{font=small}
\caption{Estimation of $\E_\pi[X^2_2]$ in GMM with \(\Sigma = \Sigma_0\). Left figure: boxplot for ULA estimates compared to the corresponding boxplots for EVM and ESVM estimates. Next three figures: boxplots for EVM and ESVM estimates for ULA, MALA, and RWM with second-order control variates being used.}
\label{fig:gmm_coloured_2nd}
\minipage{0.32\textwidth}
  \label{fig:gmm_coloured_2nd_ula}
  \includegraphics[width=\linewidth]{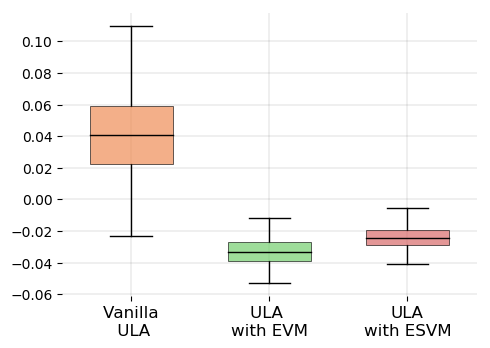}
\endminipage\hfill
\minipage{0.32\textwidth}
  \includegraphics[width=\linewidth]{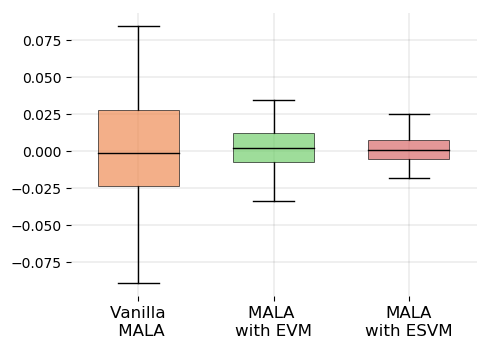}
\endminipage\hfill
\minipage{0.32\textwidth}%
  \includegraphics[width=\linewidth]{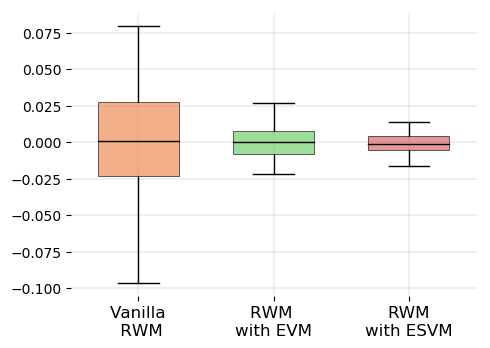}
\endminipage
\end{figure}

\begin{figure}[!htb]
\captionsetup{font=small}
\caption{Estimation of the average test likelihood in probit regression for the Pima dataset. Left figure: boxplot for ULA estimates compared to the corresponding boxplots for EVM and ESVM estimates. Next three figures: boxplots for EVM and ESVM estimates for ULA, MALA, and RWM with second-order control variates being used.}
\label{fig:pima_probit}
\minipage{0.24\textwidth}
  \includegraphics[width=\linewidth]{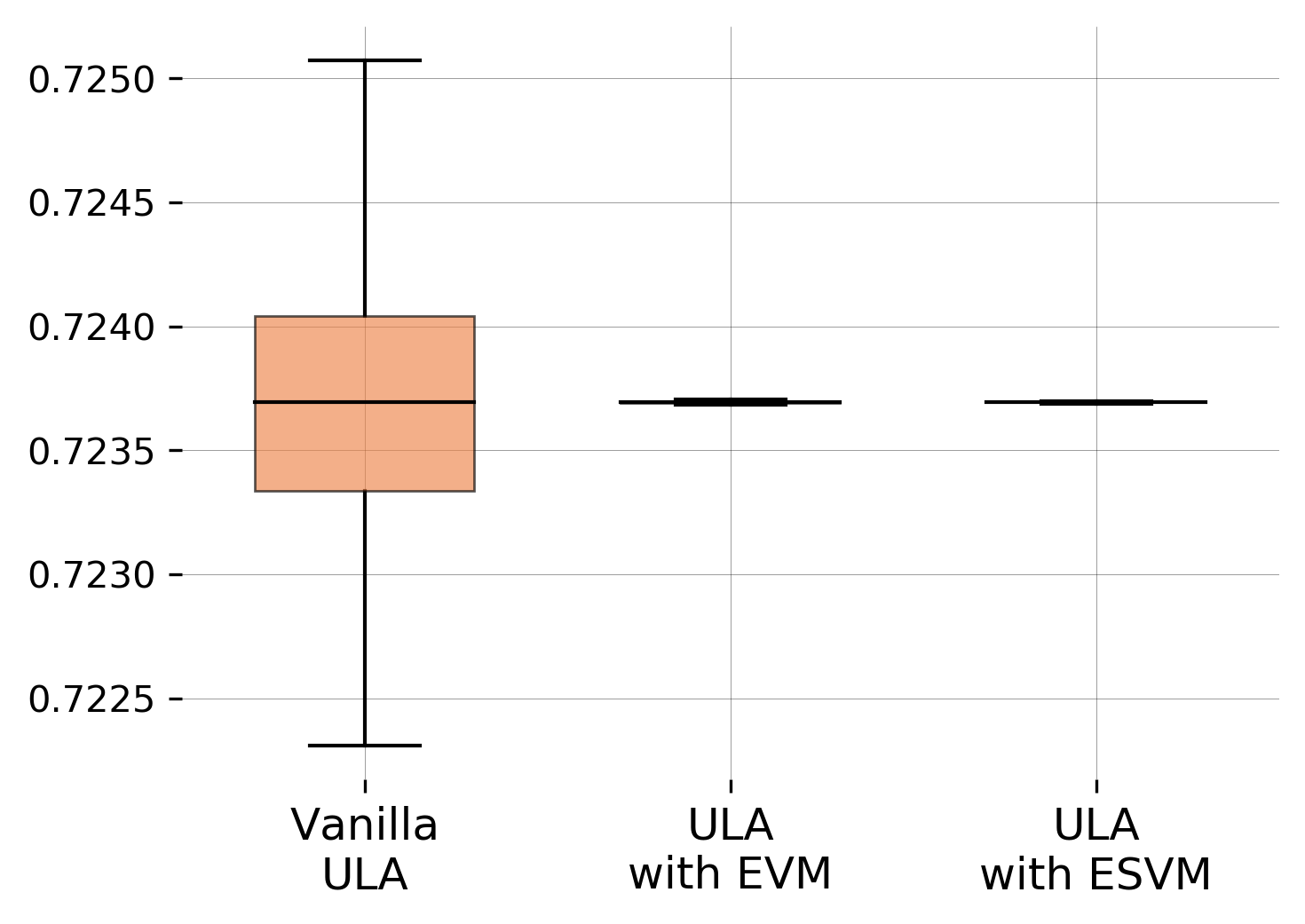}
\endminipage\hfill
\minipage{0.24\textwidth}
  \includegraphics[width=\linewidth]{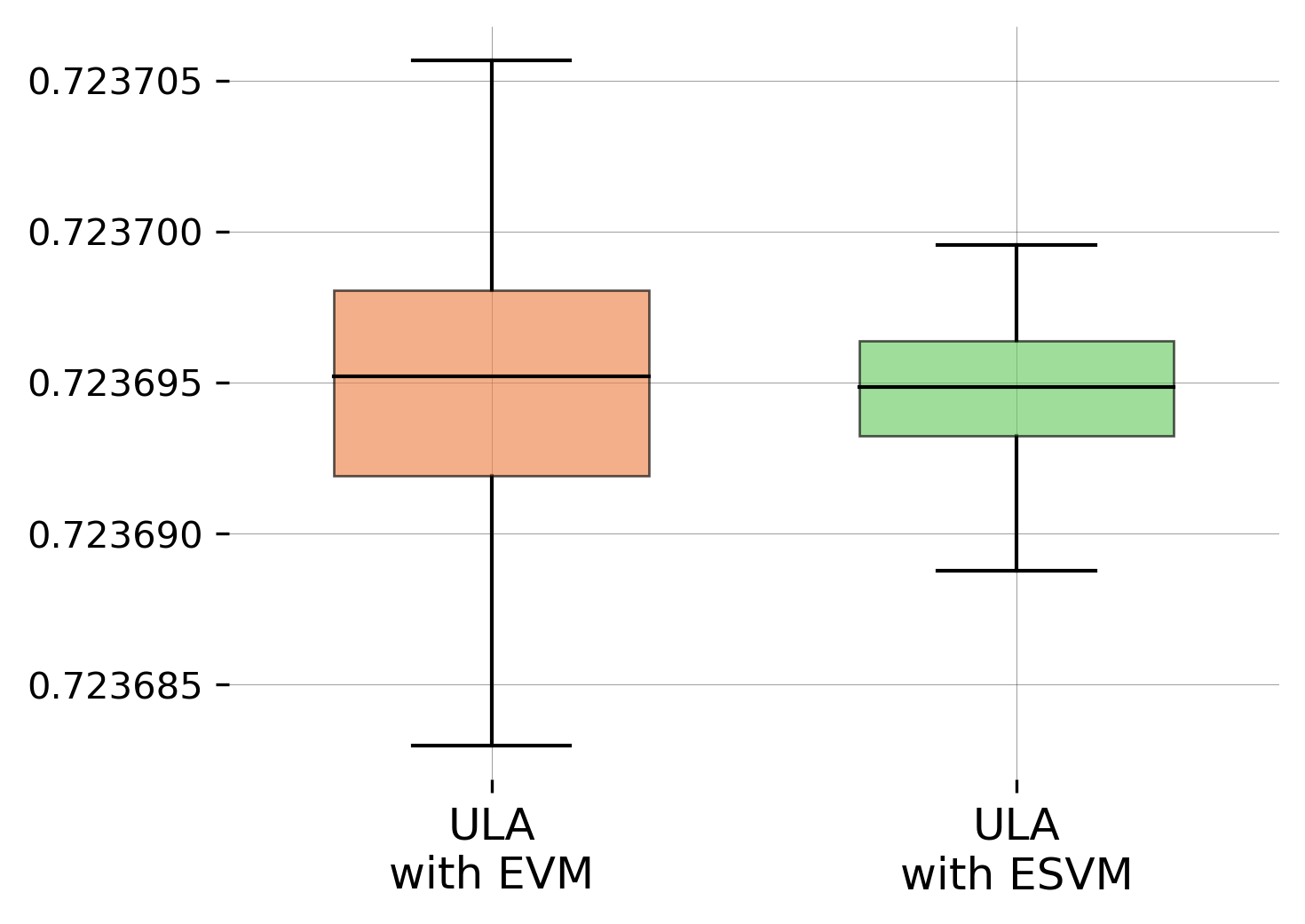}
\endminipage\hfill
\minipage{0.24\textwidth}%
  \includegraphics[width=\linewidth]{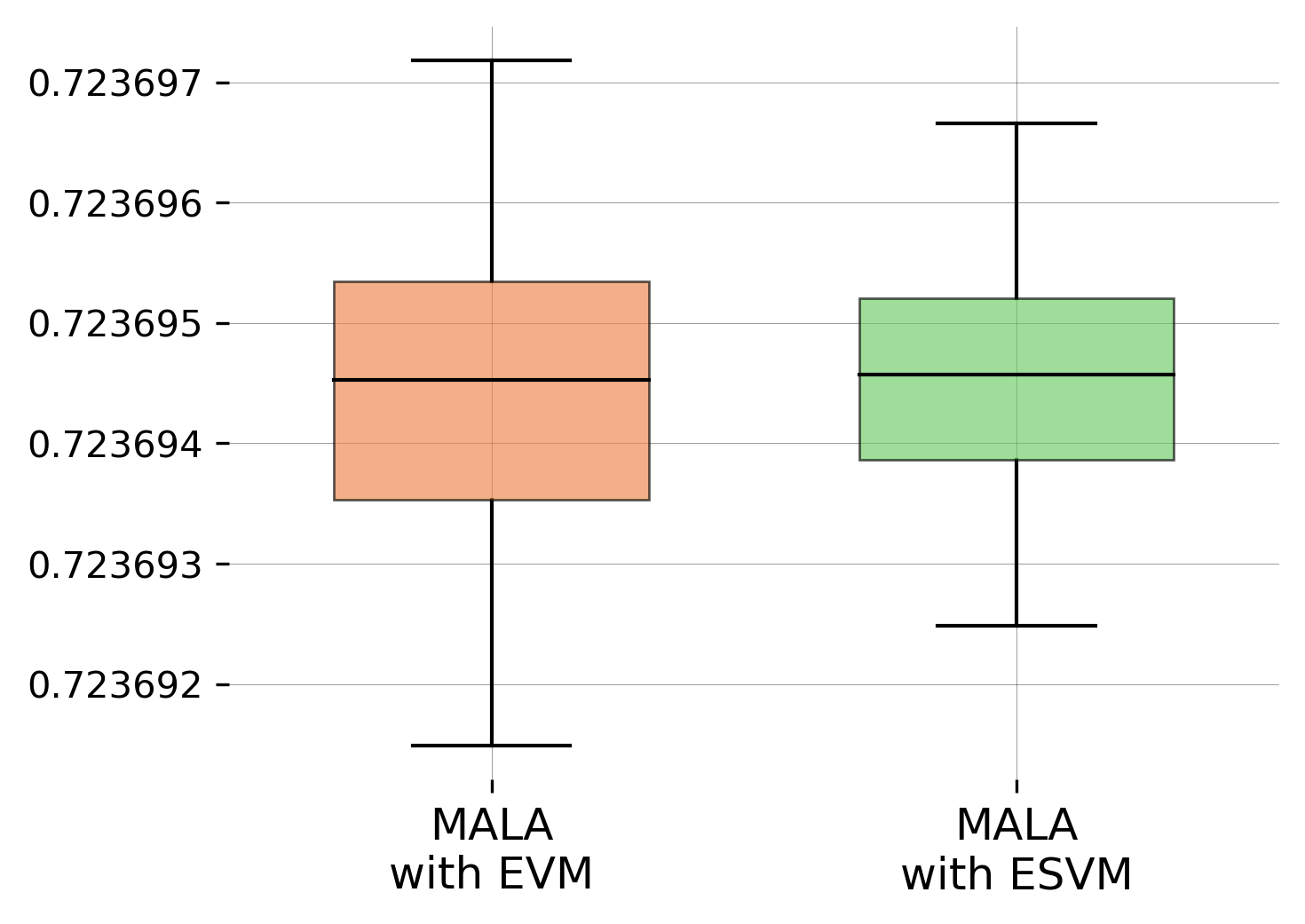}
\endminipage
\minipage{0.24\textwidth}%
  \includegraphics[width=\linewidth]{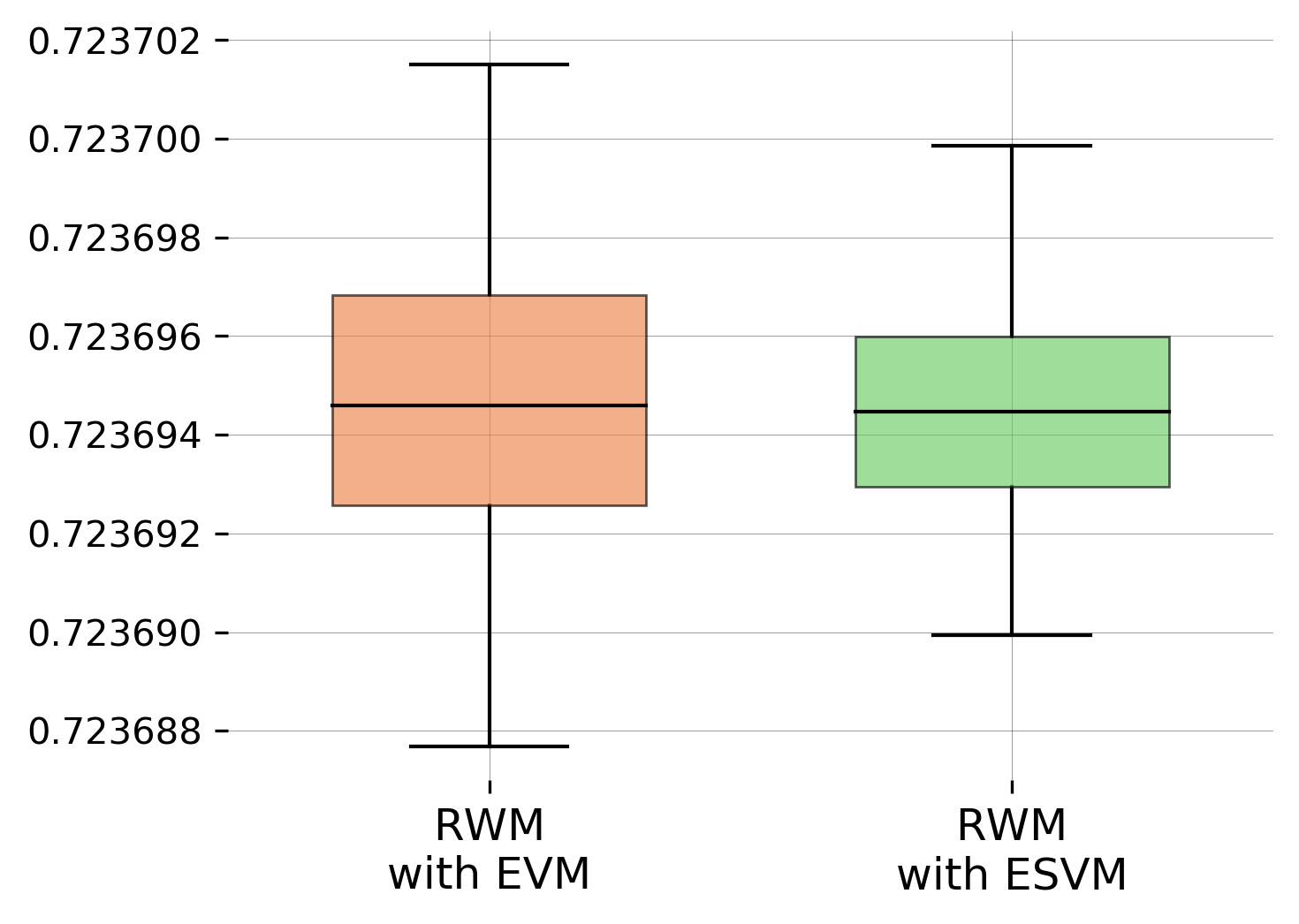}
\endminipage
\end{figure}

\begin{figure}[!htb]
\captionsetup{font=small}
\caption{Estimation of the average test likelihood in probit regression for the EEG dataset. Left figure: boxplot for ULA estimates compared to the corresponding boxplots for EVM and ESVM estimates. Next three figures: boxplots for EVM and ESVM estimates for ULA, MALA, and RWM with second-order control variates being used.}
\label{fig:eeg_probit}
\minipage{0.24\textwidth}
  \includegraphics[width=\linewidth]{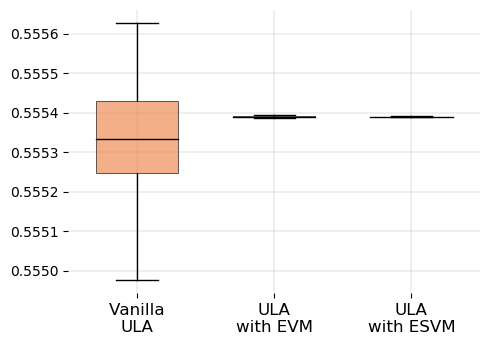}
\endminipage\hfill
\minipage{0.24\textwidth}
  \includegraphics[width=\linewidth]{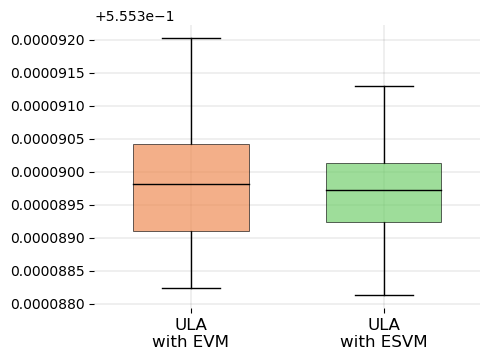}
\endminipage\hfill
\minipage{0.24\textwidth}%
  \includegraphics[width=\linewidth]{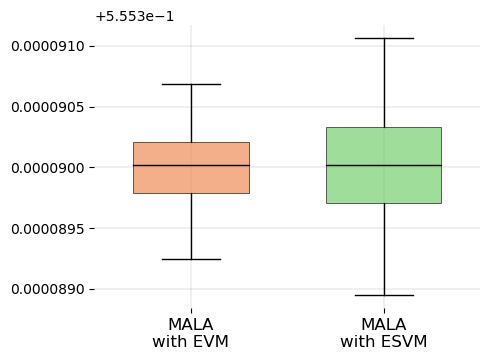}
\endminipage
\minipage{0.24\textwidth}%
  \includegraphics[width=\linewidth]{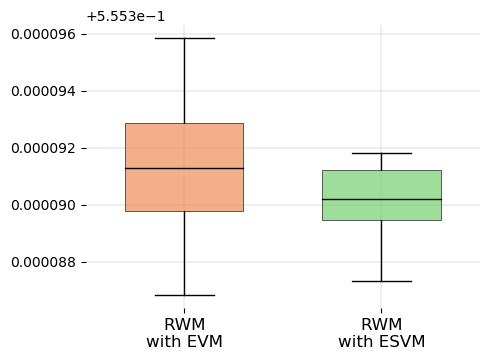}
\endminipage
\end{figure}

\begin{figure}[!htb]
\captionsetup{font=small}
\caption{Estimating the mean of the posterior distribution in the Van der Pol model. From left to right: boxplots for vanilla estimates and the corresponding EVM and ESVM estimates with third-order polynomials being used as control variates, EVM and ESVM comparison for second-order polynomials, and EVM and ESVM comparison for third-order polynomials.}
\label{fig:van_der_pole}
\minipage{0.32\textwidth}
  \includegraphics[width=\linewidth]{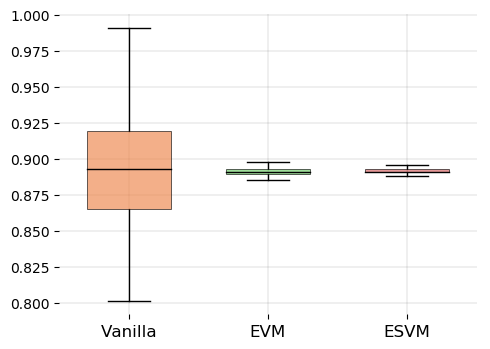}
\endminipage\hfill
\minipage{0.32\textwidth}
  \includegraphics[width=\linewidth]{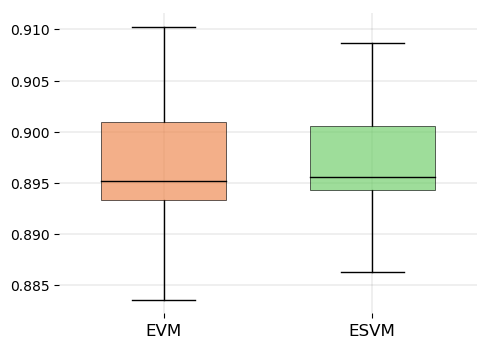}
\endminipage\hfill
\minipage{0.32\textwidth}
  \includegraphics[width=\linewidth]{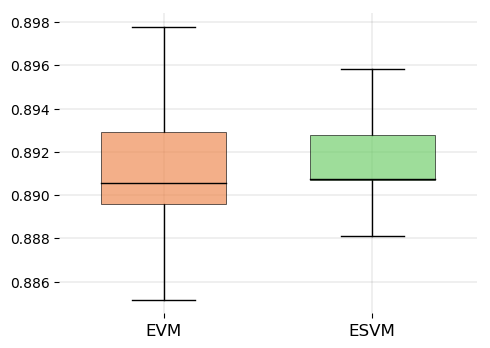}
\endminipage\hfill
\end{figure} 

\begin{figure}[!htb]
\captionsetup{font=small}
\caption{Estimating the mean of the posterior distribution in the Lotka-Volterra model. From left to right: posterior mean for parameters $\alpha, \beta, \gamma$, and $\delta$.}
\label{fig:lv_all}
\minipage{0.24\textwidth}
  \includegraphics[width=\linewidth]{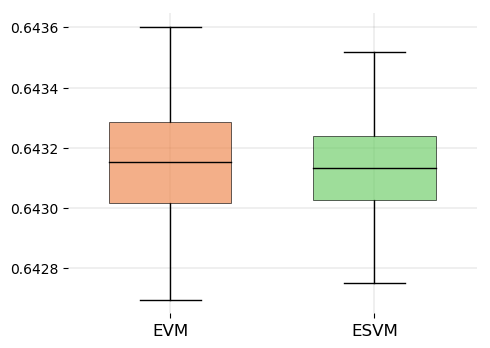}
\endminipage\hfill
\minipage{0.24\textwidth}
  \includegraphics[width=\linewidth]{lv_beta_final.png}
\endminipage\hfill
\minipage{0.24\textwidth}
  \includegraphics[width=\linewidth]{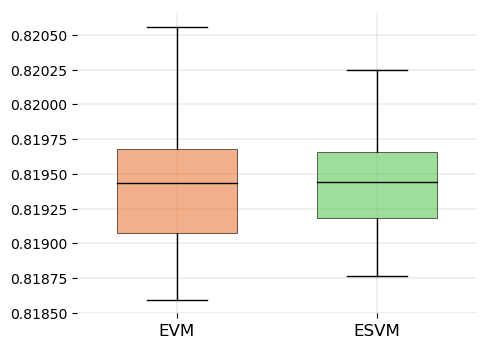}
\endminipage\hfill
\minipage{0.24\textwidth}
  \includegraphics[width=\linewidth]{lv_delta_final.png}
\endminipage\hfill
\end{figure}

\clearpage

\appendix
\section{Appendix}\label{sec:appendix}
\subsection{Concentration of the spectral variance estimator for Lipschitz functions}
\label{app:transportion-information}
The proof of a concentration inequality for Lipschitz functions falls naturally into three steps.
First we show, using a result from \citet{Guillin2004}, that the joint distribution of \((X_k)_{k=0}^{n-1}\) satisfies \(\operatorname{T}_2(\TC)\) model. Then we note that \(\operatorname{T}_2(\TC)\) implies Gaussian concentration for all Lipschitz functions. And, finally, this Gaussian concentration property implies a concentration inequality for  quadratic forms from \citet{ra15}, which we apply to the spectral variance estimator.
For the sake of completeness we provide all necessary details below.
\par
\paragraph{Tensorization of $\operatorname{T}_2(\TC)$ for Markov chains.}
Let $\P^n_{x_0}$ be the joint distribution of the Markov chain \((X_k)_{k=0}^{n-1}\)
with the Markov kernel $P$ under $\P_{x_0}$.
Since here we consider distributions on the product space $\Xset^{n-1}$,
additional definitions are needed.
We define the distance between points $x^{n-1} = (x_1,\ldots,x_{n-1}) \in \Xset^{n-1}$ and $y^{n-1} = (y_1,\ldots,y_{n-1}) \in \Xset^{n-1}$  by
\begin{equation}
	\operatorname{d}_2(x^{n-1},y^{n-1}) \eqdef \biggl(\sum\limits_{j=1}^{n-1} \| x_j - y_j \|^2 \biggr)^{1/2},
	\label{eq:dmetric}
\end{equation}
The \(L^p\)-Wasserstein distance between probability measures \(\mu\) and \(\nu\) on $\Xset^{n-1}$
with respect to the metric $\operatorname{d}_2$ is given by
\begin{eqnarray*}
W_p^{\operatorname{d}_2}(\mu,\nu) \eqdef \inf_{\zeta} \bigg(  \int_{\Xset^{n-1}\times\Xset^{n-1}} \operatorname{d}_2^p(x,y) \,  \rmd\zeta (x,y)\bigg)^{1/p},
\end{eqnarray*}
where the infimum is taken over all probability measures \(\zeta\) on the product space \(\Xset^{n-1}\times\Xset^{n-1}\) with marginal distributions \(\mu\) and \(\nu\).
And finally, we say that the probability measure \(\mu\) on $\Xset^{n-1}$ satisfies
\(\operatorname{T}_p(\TC)\) if there is a constant \(\TC > 0\) such that for any probability measure \(\nu\) on $\Xset^{n-1}$
\begin{equation*}
	W_p^{\operatorname{d}_2}(\mu,\nu) \leq \sqrt{2\TC \KL{\mu}{\nu}}.
\end{equation*}
The following theorem provides sufficient conditions for the measure \(\P^n_{x_0}\) to satisfy \(\operatorname{T}_2(\TC)\). 
\begin{theorem}[{\citet[Theorem~2.5]{Guillin2004}}]
	\label{T2CostGuillin}
	Assume that there exists $\TC>0$, such that
	$P(x,\cdot) \in \operatorname{T}_2(\TC)$ for any $x\in\Xset$, and
	there exists \(0 < r < 1\), such that for any $x,y \in \Xset$,
	\begin{equation*}
	W_2(P(x, \cdot),P(y, \cdot))  \leq r \| x - y \|.
	\end{equation*}
	Then for any probability measure \(\Q\) on \(\Xset^{n-1}\),
	the product measure \(\P^n_{x_0}\) satisfies  \(\operatorname{T}_2(\TC/(1-r)^2)\),
	\ie
	\begin{eqnarray*}
	W_2^{\operatorname{d}_2}(\Q,\P^n_{x_0}) \leq \frac{1}{1-r}\sqrt{2 \TC\KL{\Q}{\P^n_{x_0}}}.
	\end{eqnarray*}
\end{theorem}

\paragraph{Gaussian concentration for Lipschitz functions.}
A probability measure which satisfies \(\operatorname{T}_2(\TC)\) inequality is known to satisfy Gaussian concentration inequality for all Lipschitz functions. Together with \Cref{T2CostGuillin} this implies the following result.

\begin{theorem}\label{th:clf}
Assume that $P$ satisfies~\ref{assumptionW}. 
	Then for any $L$-Lipschitz function $\phi: \Xset^{n-1}\to\rset$ with respect to the
	metric $\operatorname{d}_2$ from \eqref{eq:dmetric}, it holds
	\begin{equation}\label{eq:appendix-gc}
		\P_{x_0}\bigl(\bigl|\phi(X_0,\ldots,X_{n-1})-\PE[x_0]{\phi(X_0,\ldots,X_{n-1})}\bigr|\ge t\bigr)
		\leq 2 \exp \left( -\frac{t^2}{2\TC L^2/(1-r)^2} \right).
	\end{equation}
\end{theorem}
\begin{proof}
It follows from \citet[Section~9.2]{GentilBakryLedoux} that 
\(\operatorname{T}_2(\TC)\) implies \(\operatorname{T}_1(\TC)\)
with the same constant $\TC>0$ and with respect to the same metric $\operatorname{d}_2$. In its turn 
\(\operatorname{T}_1(C)\)  imply the Gaussian concentration~\eqref{eq:appendix-gc} due to the result of \citet{BobkovGotze}. 
It remains to note that $\P^n_{x_0}$ satisfies $\operatorname{T}_2(\TC/(1-r)^2)$ by \Cref{T2CostGuillin}.
\end{proof}
\paragraph{Gaussian concentration for quadratic forms.}
Once we have proved the Gaussian concentration for Lipschitz functions,
we can obtain the Bernstein-type inequality for quadratic forms.
This idea is due to \citet{ra15}, but since we use a modified version of the inequality,
we provide the details for readers convenience.
\begin{definition}[Concentration property]
Let $Z$ be a random vector in $\rset^n$. We  say
that $Z$ has the concentration property with constant $K$ if for every $1$-Lipschitz function
$\phi: \rset^n \to \rset$, we have $\E|\phi(X)|<\infty$ and for every $t > 0$,
	\[
		\P\Bigl(\bigl|\phi(Z)-\PE{\phi(Z)}\bigr|\ge t\Bigr)
		\leq 2 \exp \left( -\fraca{t^2}{K^2} \right).
	\]
\end{definition}
The following theorem shows that the concentration property implies a
concentration inequality for quadratic forms.
\begin{theorem}\label{th:cqf}
Let $Z$ be a random vector in $\rset^n$. If $Z$ has the  concentration
property with constant $K$, then for any $n \times n$ matrix $A$ and every $t > 0$,
\begin{align*}
	\P\Bigl(\bigl|Z^{\T}A Z -\PE{Z^{\T}AZ} \bigr|>t\Bigr)
	\leq
	2  \exp\biggl(-\frac{t^2}{c K^2\left(\PE{\| AZ \|^2} + t\|A\|\right)}\biggr),
\end{align*}
where $c>0$ is a universal constant.
\end{theorem}
\begin{proof}
Without loss of generality one may assume that \(A\) is symmetric and positively semidefinite.
Let \(\varphi(z) \eqdef z^\T A z\), \( z\in\R^n \). Define \(\psi(z) \eqdef \|\nabla \varphi(z)\|\).
Since \(\| \nabla \varphi(z)\| \leq 2\|A\|\|z\|\), the function \(\psi\) is \((2\| A\|)\)-Lipschitz.
By the concentration property
\[
	\P\bigl(\bigl| \psi(Z) - \PE{ \psi(Z)}\bigr| \geq t \bigr)
	\leq
	2 \exp \left( -\frac{t^2}{4K^2\|A\|^2} \right).
\]
Note that \(\PE{\psi(Z)} = 2 \PE{\|A Z\|}\) and set for $t > 0$,
\[
B_t \eqdef \bigl\{z \in \R^n: \psi(z) \leq 2 \PE{\|A Z\|} + \sqrt{t \| A\|}\bigr\}.
\]
It holds
\[
	\P(Z \notin B_t) \leq 2\exp \biggl(-\frac{ t}{4 K^2 \| A\|}\biggr).
\]
Define \(\widetilde{\varphi}(z) \eqdef \sup_{y \in B_t} (\langle \nabla \varphi(y), z - y \rangle + \varphi(y))\).
This function is Lipschitz, since for any $z,x\in B_t$,
\(|\widetilde{\varphi}(z_1)-\widetilde{\varphi}(z_2)| \leq \sup_{y \in B_t}\|\nabla \varphi(y)\|\|z_1-z_2\| \leq M\|z_1-z_2\|\) with \(M \eqdef 2 \PE{\|A Z\|} + \sqrt{t\| A\|}\). Hence, again by the concentration property, for any $s>0$,
\begin{align*}
	\P\bigl(\bigl|\widetilde \varphi(Z) - \PE{\widetilde \varphi(Z)}\bigr| \geq s \bigr)
	&\leq 2 \exp \biggl( -\frac{s^2}{K^2(2 \PE{\|A Z\|} + \sqrt{t\| A\|})^2} \biggr)\\
	&\leq 2 \exp \biggl( -\frac{s^2}{4K^2(\PE{\|A Z\|} + \sqrt{t\| A\|})^2} \biggr).
\end{align*}
Moreover, by convexity of $\varphi$, we have $\widetilde\varphi(z)\leq\varphi(z)$ and for $z\in B_t$,
$\widetilde\varphi(z)=\varphi(z)$. Consider two random variables $Y = {\varphi}(Z)$ and
$\widetilde Y = \widetilde{\varphi}(Z)$. We have proved that $Y$ and $\widetilde Y$ coincide on the set $B_t$ of large probability
and $\widetilde Y$ has the concentration property.
It follows from \Cref{Lem3_1} (given below) that in this case
we have the Gaussian concentration for $Y$ around median $\Med Y$ of the form
\begin{align*}
	\P\Bigl(\bigl|Z^\T A Z - \Med[Z^\T A Z]\bigr| \geq t\Bigr)
	&\leq 2 \exp{\biggl(-\frac{t^2}{c K^2(\PE{\|A Z\|} + \sqrt{t \|A\|})^2}\biggr)}\\
	&\leq 2 \exp{\biggl(-\frac{t^2}{2c K^2(\PE{\|A Z\|^2} + t \|A\|)}\biggr)}.
\end{align*}
By a standard argument (see, for example, \citet[Lemma 3.2]{ra15}), we replace the median by the mean
at the cost of a universal factor. This completes the proof for a new absolute constant $c>0$.
\end{proof}

\begin{lemma}
\label{Lem3_1}
Assume that there exist positive constants \(a,b,t > 0\)  such that for any  \(s > 0\)  random variables
\(Y\), \(\widetilde{Y}\) satisfy
\[
\P\left(\bigl|\widetilde{Y}-\E \widetilde{Y}\bigr| \geq s\right) \leq 2 \exp{\biggl(-\fraca{s^2}{(a + b\sqrt{t})^2}\biggr)}
\]
and $\P\left(\widetilde{Y} \neq Y\right) \leq 2 \exp{\bigl(-{t}/{b}\bigr)}$.
Then for some positive constant $c>0$ and all \(t>0\),
\[
\P(|Y - \Med Y| \geq t) \leq 2 \exp{\biggl(-\fraca{t^2}{\{c( a + b\sqrt{t})^2\}}\biggr)}.
\]
\end{lemma}
\begin{proof}
This lemma is proved in \citet[Lemma 3.2]{ra15}. We just note that
the quantity $-\min\bigl(\fraca{t^2}{a^2},\fraca{t}{b}\bigr)$, which appears
in the result of \citet{ra15}, is bounded by the quantity $-t^2/( a + b\sqrt{t})^2$.
\end{proof}

We have arrived at the following concentration result for quadratic forms of Lipschitz function of
a Markov chain. This result is of independent interest.
\begin{corollary}
\label{th:conc}
	Assume that there exists $\TC>0$, such that
	$P(x,\cdot) \in \operatorname{T}_2(\TC)$ for any $x\in\Xset$, and
	there exists \(0 < r < 1\), such that for any $x,y \in \Xset$,
	\begin{equation*}
	W_2(P(x, \cdot),P(y, \cdot))  \leq r \| x - y \|.
	\end{equation*}
	Let also $h:\Xset\to\R$ be a $L$-Lipschitz function.
	Denote $Z_n(h) \eqdef (h(X_0),\ldots,h(X_{n-1}))^{\T}$. Then for any $n \times n$ matrix $A$ and any $t > 0$,
\begin{multline}\label{eq:expHW}
\P_{x_0}\Bigl(\bigl|Z_n(h)^{\T}A Z_n(h) -\PE[x_0]{Z_n(h)^{\T}AZ_n(h)} \bigr|>t\Bigr)
\\ \leq
2  \exp\biggl(-\frac{t^2}{c K^2\left(\PE[x_0]{\| AZ_n(h) \|^2} + t\|A\|\right)}\biggr),
\end{multline}
where $c>0$ is some universal constant and $K^2=\TC L^2/(1-r)^2$.
\end{corollary}
\begin{proof}
	The statement follows from the fact $Z_n(h)$ has the concentration property with $K=2\TC L^2/(1-r)^2$.
	Indeed, for any $1$-Lipschitz function $\phi: \rset^n \to \rset$ and any $x^{n-1} \eqdef (x_1,\ldots,x_{n-1}) \in \Xset^{n-1}$, $y^{n-1} \eqdef (y_1,\ldots,y_{n-1}) \in \Xset^{n-1}$, it holds
\begin{align*}
	\bigl|\phi(h(x_0),\ldots,h(x_{n-1}))-\phi(h(y_0),\ldots,h(y_{n-1}))\bigr|
	&\leq
	\bigg(\sum\limits_{j=1}^{n-1} ( h(x_j) - h(y_j) )^2 \bigg)^{1/2}\\
	&\leq
	L d(x^{n-1},y^{n-1}).
\end{align*}
 Hence the concentration property follows from \Cref{th:clf}.
Application of \Cref{th:cqf} to $Z_n(h)$ finishes the proof.
\end{proof}

\paragraph{Gaussian concentration of the spectral variance estimator}\label{sec:conc_var}

The main result of this section is the following.
\begin{theorem}
\label{th:conc_var}
	Assume that functions $h\in\H$ and the Markov kernel
	$P$ satisfy \ref{assumptionL} and \ref{assumptionW} with parameters $L>0$, $\TC>0$, and $0<r<1$.
	Then for all $t>0$,
\begin{align}\label{eq:conc_var_fast}
	\P_{x_0}\Bigl(\bigl|\SV{h}-\E_{x_0} \bigl[\SV{h}\bigr]\bigr|>t\Bigr)
	\leq {2}
	\exp\Biggl(-\frac{n t^2}{ c K^2\, b_n\bigl(\PE[x_0]{\SV{h}} + t\bigr)} \Biggr),
\end{align}
where $c>0$ is some universal constant, $K^2=\TC L^2/(1-r)^2$, and
$b_n$ is the size of the lag window. Moreover, if additionally  \((X_k)_{k=0}^{n-1}\)
satisfies \ref{GE} with parameters \(\varsigma\), \(\rho\), and function $W$, then
for all $t<\tau$,
\begin{align}\label{eq:conc_var_slow}
	\P_{x_0}\Bigl(\bigl|\SV{h}-\E_{x_0} \bigl[\SV{h}\bigr]\bigr|>t\Bigr)
	&\leq
	2\exp\Biggl(-\frac{n t^2}{cK_{\tau}^2b_n^2} \Biggr),
\end{align}
where
\begin{align*}
	K_\tau^2 \eqdef \frac{\TC L^2}{(1-r)^2}\left( {\|h\|_{\ltwo(\pi)}^2} + \frac{\varsigma\driftfunc(x_0)\|h\|^2_{\driftfunc^{1/2}}}{1-\rho} + \frac{\tau}{b_n}\right).
\end{align*}
\end{theorem}
\begin{proof}
The proof is straightforward.
We have showed that the spectral variance estimator
can be represented as a quadratic form $\SV{h}= Z_n(h)^{\T}A_nZ_n(h)$ with $\|A_n\|\leq 2 b_n/n$,
see \Cref{sec:reprsv} and \Cref{lem:opsv} therein. Now \Cref{th:conc} yields for $K^2=\TC L^2/(1-r)^2$
and all $t>0$, that
\begin{align*}
	\P_{x_0}\Bigl(\bigl|\SV{h} - \PE[x_0] {\SV{h}}\bigr|>t\Bigr)
	&\leq
	2\exp\Biggl(-\frac{t^2}{cK^2\bigl(\PE[x_0]{\| A_nZ_n(h) \|^2} + t\|A_n\|\bigr)}\Biggr)\\
	&\leq
	2\exp\Biggl(-\frac{n t^2}{2cK^2b_n\left(\PE[x_0]{\SV{h}} + t\right)} \Biggr),
\end{align*}
which establishes \eqref{eq:conc_var_fast} for a new absolute constant $c>0$.
To prove the second inequality we note that by \Cref{lem:opsv} and \Cref{lem:evar},
\[
	\PE[x_0]{\SV{h}}
	\leq \|A_n \| \PE[x_0]{\|Z_n(h)\|^2}
	\leq 2b_n{\|h\|_{\ltwo(\pi)}^2} + \frac{ 2 \varsigma\driftfunc(x_0) \|h\|^2_{\driftfunc^{1/2}} }{1-\rho}\frac{b_n}{n}\eqsp.
\]
Hence for any $0<t<\tau$, we have
\[
	\PE[x_0]{\SV{h}} + t
	\leq
	b_n \left( 2 {\|h\|_{\ltwo(\pi)}^2} + \frac{ 2 \varsigma\driftfunc(x_0)\|h\|^2_{\driftfunc^{1/2}}}{1-\rho}\frac{1}{n} + \frac{\tau}{b_n}\right).
\]

Substituting this into \eqref{eq:conc_var_fast} we deduce
\begin{align*}
	\P_{x_0}\Bigl(\bigl|\SV{h} - \PE[x_0] {\SV{h}}\bigr|>t\Bigr)
	&\leq
	2\exp\Biggl(-\frac{n t^2}{cK_{\tau}^2b_n^2} \Biggr)
\end{align*}
for a new absolute constant $c>0$ and
\begin{align*}
	K_\tau^2 \eqdef \frac{\TC L^2}{(1-r)^2}\left( {\|h\|_{\ltwo(\pi)}^2} + \frac{\varsigma\driftfunc(x_0)\|h\|^2_{\driftfunc^{1/2}}}{1-\rho} + \frac{\tau}{b_n}\right),
\end{align*}
which completes the proof.
\end{proof}

\subsection{Concentration of the spectral variance estimator for bounded functions}

\begin{theorem}\label{th:conc_var2}
Assume that
$P$ satisfies \ref{GE} and \ref{BR} with parameters \(\varsigma,\rho,\ls>0\), function $W$, and set $\Cset$.
Assume also that functions $h\in\H$ satisfy \ref{assumptionB} with parameter $B>0$.
Then for $x_0 \in \Cset$, for all functions $h\in\H$, and all $t>0$,
\begin{align}\label{eq:conc_var_db}
	\P_{x_0}\Bigl(\bigl|\SV{h} - \PE[x_0]{\SV{h}}\bigr|>t\Bigr)
	\leq {2}
	\exp\Biggl(-\frac{ t^2 n}{c K^2 b_n^2} \Biggr),
\end{align}
where $b_n$ is the size of the lag window, $K = \beta B^2$, and $\beta$ is given by
  \begin{equation}\label{eq:beta_def}
  \beta=\frac{\varsigma \ls}{1- \rho}\biggl(\frac1{\log u}+\frac{J \varsigma \ls }{1-\rho}\biggr)\eqsp.
   \end{equation}
\end{theorem}
\begin{proof}
The main idea of the proof is to show that the spectral variance satisfies the bounded difference property.
First we rewrite the lag $s$ sample autocovariance function as
\begin{align*}
	\ecovcoeff[n]{h}{s} &= \frac1n  \sum_{k=0}^{n-s-1} \Bigl(h(X_k) - \pi_n(h)\Bigr) \Bigl(h(X_{k+s}) - \pi_n(h)\Bigr)\\
	&=\frac1n \sum_{k=0}^{n-s-1} h(X_k)h(X_{k+s}) - \frac{\pi_n(h)}{n}  \sum_{k=s}^{n-s-1}h(X_{k}).
\end{align*}
Let $\ecovcoeff[n]{h, i}{s}$ and $V^{(i)}_{n}(h)$ be the sample autocovariance function and the
spectral variance determined on another sample $X_0,\ldots,X_{i-1},X'_i,X_{i-1},\ldots,X_{n-1}$,
where we have replaced $X_i$ by $X'_i$. It holds
\begin{align*}
	\bigl|\ecovcoeff[n]{h}{s} - \ecovcoeff[n]{h,i}{s}\bigr| \leq 2 B^2 + \frac{2(n-2s+n)}{n^2}B^2 \leq \frac{6B^2}{n},
\end{align*}
and since $|w_n(s)|\leq1$ by definition,
\begin{align*}
	\bigl|\SV{h} - \SV[(i)]{h}\bigr|
	\leq 2 b_n  \sup_{s}|w_n(s)| \cdot |\ecovcoeff[n]{h}{s} - \ecovcoeff[n]{h,i}{s}|
	\leq \frac{12b_nB^2}{n}.
\end{align*}
The bounded differences inequality for Markov chains from \citet[Theorem~23.3.1]{douc:moulines:priouret:2018}) with explicit constants from \citet{havet:moulines:2019b}  yields
\begin{align*}
	\P_{x_0}\Bigl(\bigl|\SV{h} - \PE[x_0]{\SV{h}}\bigr|>t\Bigr)
	\leq {2}
	\exp\Biggl(-\frac{ t^2 n}{ 144 \beta B^4  b_n^2} \Biggr),
	\ \text{with}\ 
	\beta=\frac{\varsigma \ls}{1- \rho}\biggl(\frac1{\log u}+\frac{J \varsigma \ls }{1-\rho}\biggr)\eqsp.
\end{align*}
which completes the proof.
\end{proof}

\bibliographystyle{plainnat}

\end{document}